\documentclass[11pt]{article}
\usepackage{amsthm,amsmath,amssymb,amsfonts,bbm,graphics,a4wide,rotating,graphicx}
\usepackage{graphicx, multicol, latexsym, amsmath, amssymb}
\usepackage{blindtext}

\usepackage{caption}
\usepackage{capt-of}
\usepackage{tabu}
\usepackage{hyperref}
\usepackage{times}
\usepackage{graphicx}
\usepackage{color,graphics}
\usepackage{dsfont}
\usepackage{amsmath,amssymb}
\usepackage{graphicx,subfigure}
\usepackage[plain,noend]{algorithm2e}

\numberwithin{equation}{section}
\theoremstyle{plain}

\usepackage{graphicx}
\usepackage{mathrsfs}
 \usepackage{dsfont}
\usepackage{natbib}
\usepackage{color}
\usepackage{multirow}

\renewcommand{\L}{\mathbb{L}}

\newcommand{\E}{\ensuremath{\mathbb{E}}}
\renewcommand{\P}{\ensuremath{\mathbb{P}}}
\newcommand{\R}{\ensuremath{\mathbb{R}}}

\newcommand{\e}{\ensuremath{\epsilon}}

\newcommand{\var}{\mbox{Var}}

\def\1{\mathds{1}}

\def\Num{\bar N}

\numberwithin{equation}{section}

\newtheorem{remark}{Remark}

 \newtheorem{lemma}{Lemma}
\newtheorem{theorem}{Theorem}
\newtheorem{cor}{Corollary}

\date{\today}
\begin{document}

\title{Nonparametric Bayesian estimation of multivariate Hawkes processes}

\author{{
\sc Sophie Donnet},\\[2pt]
 AgroParisTech, INRA, France\\[6pt]
{\sc Vincent Rivoirard} \\[2pt]
Universit\'e Paris-Dauphine, Paris, France\\[6pt]
{\sc Judith Rousseau} \\[2pt]
Department of Statistics, University of Oxford and CEREMADE, Universit\'e Paris-Dauphine\\[6pt]
}

\maketitle

\begin{abstract}
This paper studies nonparametric estimation of parameters of multivariate Hawkes processes. We consider the Bayesian setting and derive posterior concentration rates. First rates are derived for $\L_1$-metrics for stochastic intensities of the Hawkes process. We then deduce rates for the $\L_1$-norm of interactions functions of the process. Our results are exemplified by using priors based on piecewise constant functions, with regular or random partitions  and  priors based on mixtures of Betas distributions. Numerical illustrations are then proposed with in mind applications for inferring functional connectivity graphs of neurons.
\end{abstract}

%\keywords{Multivariate counting process, Hawkes processes, nonparametric Bayesian estimation, posterior concentration rates}

%%%%%%%%%%%%%%%%%%%%%%%%%%
%%%%%%%%%%%%%%%%%%%%%%%%%%
\section{Introduction}

In this paper we study the properties of Bayesian nonparametric procedures in the context of multivariate Hawkes processes. The aim of this paper is to give some general results on posterior concentration rates for such models and to study some  families of nonparametric priors. 
%%%%%%%%%%%%%%%%%%%%%%%%%%
\subsection{Hawkes processes}\label{hawkes}
Hawkes processes, introduced by Hawkes (1971), are specific point processes which are extensively used to model data whose occurrences depend on previous occurrences of the same process. To describe them, we first consider $N$ a point process on $\R$. We denote by ${\mathcal B}(\R)$ the Borel $\sigma$-algebra on $\R$ and for any Borel set $A\in {\mathcal B}(\R)$, we denote by $N(A)$ the number of occurrences of $N$ in $A$.  For short, for any $t\geq 0$, $N_t$ denotes the number of occurrences in $[0,t]$. We assume that for any $t\geq 0$, $N_t<\infty$ almost surely. If $\mathcal G_t$ is the history of $N$ until $t$,
%$$\mathcal G_t=\sigma\left\{N(C): \ C\in {\mathcal B}(\R), \ C\subset (-\infty,t]\right\},$$
then, $\lambda_t$, the predictable intensity of $N$ at time $t$, which represents the probability to observe a new occurrence at time $t$ given previous occurrences, is defined by
$$\lambda_tdt=\P(dN_t=1 \, |\, \mathcal G_{t^-}),$$
where $dt$ denotes an arbitrary small increment of $t$ and $dN_t=N([t,t+dt]).$
For the case of {\it univariate Hawkes processes},  we have
$$\lambda_t=\phi\left(\int_{-\infty}^{t^-}h(t-s)dN_s\right),$$
for $\phi:\R\mapsto\R_+$  and $h:\R\mapsto\R$. We recall that the last integral means
$$\int_{-\infty}^{t^-}h(t-s)dN_s=\sum_{T_i\in N : \,T_i<t}h(t-T_i).$$
The case of {\it linear Hawkes processes} corresponds to $\phi(x)=\nu+x$ and $h(t)\geq 0$ for any $t$. The parameter $\nu\in\R_+^*$ is  the {\it spontaneous rate} and $h$ is the {\it self-exciting function}. We now assume that $N$ is a {\it marked} point process, meaning that each occurrence $T_i$ of $N$ is associated to a mark $m_i\in\{1,\ldots,K\}$, see \cite{MR1950431}. In this case, we can identify $N$ with a {\it multivariate point process} and for any $k\in\{1,\ldots,K\},$ $N^k(A)$ denotes the number of occurrences of $N$ in $A$ with mark $k$. In the sequel, we only consider linear multivariate Hawkes processes, so we assume that $\lambda^k_t$, the intensity of $N^k$, is 
\begin{equation}\label{intensi}
\lambda^k_t=\nu_k+\sum_{\ell=1}^K\int_{-\infty}^{t^-}h_{\ell, k}(t-u)dN^{\ell}_u,
\end{equation}
where $\nu_k>0$ and $h_{\ell, k}$, which is assumed to be non-negative and supported by $\R_+$, is the {\it interaction function} of $N^{\ell}$ on $N^k$. Theorem 7 of \cite{BM1} shows that if the $K\times K$ matrix  $\rho$, with
\begin{equation}\label{rho-matrix}
\rho_{\ell,k}=\int_0^{+\infty} h_{\ell, k}(t)dt,\quad \ell,k=1,\ldots,K,
\end{equation} 
has a spectral radius strictly smaller than 1, then there exists a unique stationary distribution for the multivariate process $N=(N^k)_{k=1,\ldots,K}$ with the previous dynamics and finite average intensity. 

Hawkes processes have been extensively  used in a wide range of applications. They are used to model earthquakes \cite{VJ-O, ogata88, MR1941459}, interactions in social
networks \cite{arXiv:1203.3516, icml2013_zhou13, Li:2014:LPM:2893873.2893891, arXiv:1501.00725, crane2008robust, mitchell2009hawkes, yang2013mixture}, financial data \cite{embrechts2011multivariate, bacry2015hawkes, bacry2016estimation, bacry2013modelling, ait2015modeling}, violence rates \cite{mohler2011self, porter2012self}, genomes \cite{gusto2005fado, CSWH, MR2722456} or neuronal activities \cite{Brillinger1988, chornoboy1988maximum, okatan2005analyzing, paninski2007statistical, pillow2008spatio, HRR, reynaud2014goodness, reynaud2013inference}, to name but a few.

Parametric inference for Hawkes models based on the likelihood is the most common in the literature and we refer the reader to \cite{ogata88, CSWH} for instance. Non-parametric estimation has first been considered by Reynaud-Bouret and Schbath \cite{MR2722456} who proposed a procedure based on minimization of an $\ell_2$-criterion penalized by an $\ell_0$-penalty for univariate Hawkes processes. Their results have been extended to the multivariate setting by Hansen, Reynaud-Bouret and Rivoirard \cite{HRR} where the  $\ell_0$-penalty is replaced with an $\ell_1$-penalty. The resulting Lasso-type estimate leads to an easily implementable procedure providing sparse estimation of the structure of the underlying connectivity graph. To generalize this procedure  to the high-dimensional setting, Chen, Witten and Shojaie \cite{MR3634334} proposed a simple and computationally inexpensive edge screening approach, whereas  Bacry, Ga\"{i}ffas and Muzy \cite{arXiv:1501.00725} combine $\ell_1$ and trace norm penalizations to take into account the low rank property of their self-excitement matrix. Very recently, to deal with non-positive interaction functions,  Chen, Shojaie, Shea-Brown and Witten \cite{chen:Shoj:shea:witten} combine the thinning process representation and a coupling construction to bound the dependence coefficient of the Hawkes process. Other alternatives based on spectral methods \cite{Bacry2012} or estimation through the resolution of a Wiener-Hopf system \cite{MR3480107} %and moment matching methods (M. Achab, E. Bacry, S. Gaiffas, I. Mastromatteo, and J.-F. Muzy, 2017) 
can also been found in the literature. These are all frequentist methods; Bayesian approaches for Hawkes models have received much less attention. To the best of our knowledge, the only contributions for the Bayesian inference are due to Rasmussen \cite{MR3085883} and Blundell, Beck and Heller \cite{NIPS2012_4834} who explored parametric approaches and used MCMC to approximate the posterior distribution of the parameters. 
%
%
%
%Their statistical inference is relatively well understood, too, from
%a classical parametric angle (Ogata [
%34]) 
%
%Inference for Hawkes models based on the likelihood can be found in the literature,
%in particular, for parametric models [
%Ogata (1988),
%
%together with recent significant advances
%in nonparametrics (Reynaud-Bouret and Schbath [
%41], Hansen et al. [
%18]). E. Bacry, S. Ga�ffas, J.-F. Muzy (2016) (sparse mathods). Hansen et al combined with screening by Chen, Witten and Shojaie (2017).
%
%
%spectral methods Emmanuel Bacry, Khalil Dayri, and Jean-Francois Muzy (2012) or esti-
%mation through the resolution of a Wiener-Hopf system (Emmanuel Bacry and Jean-Fran?cois Muzy (2014ab) moment matching methods (M. Achab, E. Bacry, S. Gaiffas, I. Mastromatteo, and J.-F. Muzy, 2017),  have been proposed.
%
%%%%%%%%%%%%%%%%%%%%%%%%%%
\subsection{Our contribution}\label{contribution}
In this paper, we study nonparametric posterior concentration rates when $T\to+\infty$, for estimating the parameter $f=((\nu_k)_{k=1,\ldots,K}, (h_{\ell,k})_{k,\ell=1,\ldots,K})$ by using realizations of the multivariate process $(N_t^k)_{k=1,\ldots,K}$ for $t\in [0,T]$. Analyzing asymptotic properties in the setting where $T\to +\infty$ means that the observation time becomes very large hence providing a large number of observations. Note that along the paper, $K$, the number of observed processes, is assumed to be fixed and can be viewed as a constant. Considering $K\to +\infty$ is a very challenging problem beyond the scope of this paper. Using the general theory of \citet{ghosal:vdv:07}, we express the posterior concentration rates in terms of simple and usual quantities associated to the prior on $f$ and under mild conditions on the true parameter. Two types of posterior concentration rates are provided: the first one is in terms of the $\L_1$-distance on the stochastic intensity functions $(\lambda^k)_{k=1,\ldots,K}$ and the second one is in terms of the $\L_1$-distance on the parameter $f$ (see precise notations below).  To the best of our knowledge, these are  the first  theoretical results on Bayesian nonparametric inference in Hawkes models. Moreover, these are the first results on $\L_1$-convergence rates for the interaction functions $h_{\ell, k}$. In the frequentist literature, theoretical results are given in terms of  either the $\L_2$-error of the stochastic intensity, as in \cite{arXiv:1501.00725} and \cite{MR3480107}, or in terms of the $\L_2$-error on the interaction functions themselves, the latter being much more involved, as in  \cite{MR2722456} and  \cite{HRR}. In \cite{MR2722456}, the estimator is constructed using  a frequentist model selection procedure with a specific family of models based on piecewise constant functions. In the multivariate setting of \cite{HRR}, more generic families of approximation models are considered (wavelets of Fourier dictionaries) and then combined with a Lasso procedure, but under a somewhat restrictive assumption on the size of models that can be used to construct the estimators (see Section 5.2 of \cite{HRR}). Our general results do not involve such strong conditions and therefore allow us to work with approximating families of models that are quite general. In particular, we can apply them to two families of prior models on the interaction functions $h_{\ell, k}$:  priors based on piecewise constant functions, with regular or random partitions  and  priors based on mixtures of Betas distributions. From the posterior concentration rates, we also deduce a frequentist convergence rate for the posterior mean, seen as a point estimator.
 We finally propose an MCMC algorithm to simulate from the posterior distribution for the priors constructed from piecewise constant functions and a simulation study is conducted to illustrate our results. 
%%%%%%%%%%%%%%%%%%%%%%
\subsection{Overview of the paper}
In Section~\ref{sec:main}, Theorem~\ref{th:d1} first states the posterior convergence rates obtained for stochastic intensities. Theorem~\ref{th:slices} constitues a variation of this first result. From these results, we derive $\L_1$-rates for the parameter $f$ (see Theorem~\ref{th:L1}) and for the posterior mean (see Corollary~\ref{coro:L1hatf}). Examples of prior models satisfying conditions of these theorems are given in Section~\ref{sec:prior-model}. In Section \ref{sec:numerical Hawkes}, numerical results are provided. 
%%%%%%%%%%%%%%%%%%%%%%
\subsection{Notations and assumptions} \label{sec:notation}
We denote by $f_0 =((\nu_k^0)_{k=1,\ldots,K}, (h_{\ell,k}^0)_{k,\ell=1,\ldots,K})$  the true parameter and assume that the interaction functions $h_{\ell,k}^0$ are supported by a compact interval $[0,A]$, with $A$ assumed to be known. Given a parameter $f =  ((\nu_k)_{k=1,\ldots,K}, (h_{\ell,k})_{k,\ell=1,\ldots,K})$, we denote by $\|\rho\|$ the spectral norm of the matrix $\rho$ associated with $f$ and defined in \eqref{rho-matrix}. We recall that $\|\rho\|$ provides an upper bound of the spectral radius of $\rho$ and we set 
$$\mathcal H = \left\{(h_{\ell,k})_{k,\ell=1,\ldots,K}; \,  h_{\ell,k} \geq 0, \, \mbox{support}(h_{\ell,k})\subset [0,A], \, \rho_{\ell,k}<\infty, \, \forall\, k, \ell =1,\ldots,K, \, \| \rho \| < 1\right\}$$ 
and 
$$\mathcal F=  \{ f = ((\nu_k)_{k=1,\ldots,K}, (h_{\ell,k})_{k,\ell=1,\ldots,K}); \ 0<\nu_k<\infty, \ \forall\, k=1,\ldots,K, \ (h_{\ell,k})_{k,\ell=1,\ldots,K} \in \mathcal H\}.$$ 
We assume that $f_0\in\mathcal F$ and denote by $\rho^0$ the matrix such that 
$\rho^0_{\ell,k}=\int_0^A h^0_{\ell,k}(t)dt.$

For any function $h:\R\mapsto\R$, we denote by $\|h\|_p$ the $\L_p$-norm of $h$. With a slight abuse of notations, we also use for $f=((\nu_k)_{k=1,\ldots,K}, (h_{\ell,k})_{k,\ell=1,\ldots,K})$ and $f'=((\nu_k)_{k=1,\ldots,K}, (h'_{\ell,k})_{k,\ell=1,\ldots,K})$ belonging to~$\mathcal F$
\begin{equation}\label{L1-dist}
\| f- f'\|_1 = \sum_{k=1}^K |\nu_k  - \nu'_k| + \sum_{k=1}^K\sum_{\ell=1}^K \| h_{\ell,k} - h'_{\ell,k}\|_1.
\end{equation}
Finally, we consider $d_{1,T}$, the following stochastic distance on $\mathcal F$:
%$$d_2^2(f,f')=\sum_{k=1}^K\left((\nu_k-\nu'_k)^2+\sum_{\ell=1}^K\|h_{\ell,k}-h'_{\ell,k}\|_2^2\right),$$ 
%$$ d_{2,T}^2(f,f')=\frac{1}{T}\sum_{k=1}^K\int_0^T(\lambda^k_t(f)-\lambda^k_t(f'))^2dt\quad\mbox{and}\quad 
$$d_{1,T} (f, f') = \frac{ 1 }{ T } \sum_{k=1}^K \int_0^T | \lambda^k_t(f) - \lambda^k_t(f')| dt,$$
where $ \lambda^k_t(f)$ and $\lambda^k_t(f')$ denote the stochastic intensity (introduced in \eqref{intensi}) associated with $f$ and $f'$ respectively. We denote by ${\mathcal N}(u, \mathcal H_0, d)$ the covering number of a set $\mathcal H_0$ by balls with respect to the metric $d$ with radius $u$. 
We set for any $\ell$, $\mu_\ell^0$ the mean of $\lambda_t^{\ell}(f_0)$ under $\P_0$
$$\mu_\ell^0=\E_0[\lambda^\ell_t(f_0)],$$
where $\P_0$ denotes the stationary distribution associated with $f_0$ and $\E_0$ is the expectation associated with $\P_0$. We also write $u_T \lesssim v_T$ if $|u_T/v_T|$ is bounded when $T\to+\infty$ and similarly $u_T \gtrsim v_T$ if $|v_T/u_T|$ is bounded.
%%%%%%%%%%%%%%%%%%%%%%
%%%%%%%%%%%%%%%%%%%%%%
\section{Main results} \label{sec:main}
This section contains main results of the paper. We first provide an expression for the posterior distribution. 
%%%%%%%%%%%%%%%%%%%%%%
\subsection{Posterior distribution} \label{sec:likelihood}
Using Proposition 7.3.III of \cite{MR1950431}, and identifying a multivariate Hawkes process as a specific marked Hawkes process,
we can write the log-likelihood function of the process observed on the interval $[0,T]$, conditional on $\mathcal G_{0^-}=\sigma \left( N_t^k, \ t <0 , 1\leq k\leq K \right) $,   as
%$$
%L(\lambda):=\int_0^T\int_{\mathcal K}\log(\lambda(t,k))dN_{t,k}-\int_0^T\int_{\mathcal K}\lambda(u,k)dudc(k)
%$$
%where $\mathcal K=\{1,\ldots,K\}$ and $c$ denotes the counting measure on $\mathcal K$. Writing 
%$N_{t,k}=N^k_t$ and
%$$\lambda(t,k)=\lambda_k(t)=\nu_k+\sum_{\ell=1}^K\int_{-\infty}^{t-}h_{\ell, k}(t-u)dN^{\ell}_u,$$
%we obtain
\begin{eqnarray}\label{loglikely}
L_T(f)&:=&\sum_{k=1}^K\left[\int_0^T\log(\lambda_t^k(f))dN^k_t-\int_0^T\lambda_t^k(f)dt\right].
%\nonumber\\
%&=&\sum_{k=1}^K\left[\int_0^T\log\left(\nu_k+\sum_{\ell=1}^K\int_{-\infty}^{t-}h_{\ell, k}(t-u)dN^{\ell}_u\right)dN^k_t-\int_0^T\left(\nu_k+\sum_{\ell=1}^K\int_{-\infty}^{t-}h_{\ell, k}(t-u)dN^{\ell}_u\right)dt\right].
\end{eqnarray}
With a slight abuse of notation, we shall also denote $L_T(\lambda)$ instead of $L_T(f)$.

Recall that we restrict ourselves to the setup where for all $\ell,k$,  $h_{\ell, k}$ has support included in $[0,A]$ for some fixed $A>0$. This hypothesis is very common in the context of Hawkes processes, see \cite{HRR}. 
Note that, in this case, the conditional distribution of $(N^k)_{k=1,\ldots,K}$ observed on the interval $[0,T]$ given  $\mathcal G_{0^-}$  is equal to its conditional distribution given $\mathcal G_{[-A,0[} = \sigma \left( N_t^k,-A\leq  t <0 , 1\leq k \leq K \right) $. 

Hence, in the following, we assume that we observe the process $(N^k)_{k=1,\ldots,K}$ on $[-A, T]$, but we base our inference on the log-likelihood \eqref{loglikely}, which is  associated to the observation of $(N^k)_{k=1,\ldots,K}$ on $[0,T]$. We consider a Bayesian nonparametric approach and denote by $\Pi$ the prior distribution on the parameter $f=((\nu_k)_{k=1,\ldots,K}, (h_{\ell,k})_{k,\ell=1,\ldots,K})$.
%We will hereafter denote the conditional (on $\mathcal G_0$) log-likelihood either by $L(\lambda^*)$ of by $L(f)$. 
The posterior distribution  is then formally equal to 
\begin{equation*}
\Pi\left( B | N, \mathcal G_{0^-} \right) = \frac{ \int_{B} \exp(L_T(f)) d\Pi(f |\mathcal G_{0^-})}{ \int_{\mathcal F} \exp(L_T(f)) d\Pi(f |\mathcal G_{0^-})}.
\end{equation*}
We approximate it by the following pseudo-posterior distribution, which we write $\Pi\left( \cdot | N\right)$
\begin{equation}\label{posterior}
\Pi\left( B | N \right) = \frac{ \int_{B}\exp(L_T(f))  d\Pi(f)}{ \int_{\mathcal F} \exp(L_T(f))  d\Pi(f )},
\end{equation}
which thus corresponds to choosing $d\Pi(f) = d\Pi(f | \mathcal G_0^{-})$. 
%Note that Equation \eqref{posterior} can be interpreted as defining as prior distribution $d\Pi(f |\mathcal G_{0^-})= d\Pi(f )$ where $\mathcal G_{0^-}$ is understood as past observations which could be used to construct an informative prior. 
%%%%%%%%%%%%%%%%%%%%%%%%%%%%%
\subsection{Posterior convergence rates for $d_{1,T}$ and $\L_1$-metrics}
In this section we give two results of posterior concentration rates, one in terms of the stochastic distance $d_{1,T} $ and another one in terms of the $\L_1$-distance, which constitutes the main result of this paper.  We define  
$$ \Omega_T =\left\{\max_{\ell\in\{1,\ldots,K\}}\sup_{t\in [0,T]}N^{\ell}[t-A,t)\leq C_\alpha\log T\right\}\cap \left\{ \sum_{\ell=1}^K \left| \frac{N^\ell[-A,T]}{T}-\mu_\ell^0\right| \leq \delta_T \right\}$$ with $\delta_T = \delta_0 (\log T)^{3/2}/\sqrt{T}$ and $\delta_0>0$ and $C_{\alpha}$ two positive constants not depending on $T$. From Lemmas~\ref{control} and \ref{lem:N0T} in Section~\ref{sec:lemmas}, we have that for all $\alpha >0$ there exist $C_\alpha>0$ and $\delta_0>0$ only depending on $\alpha$ and $f_0$ such that 
\begin{equation}\label{boundNT}
\P_0 \left( \Omega_{T}^c \right) \leq T^{-\alpha},
\end{equation}
when $T$ is large enough. In the sequel, we take $\alpha>1$ and $C_\alpha$ accordingly. Note in particular that, on~$\Omega_T $, $$\sum_{\ell=1}^K N^\ell[-A,T]  \leq N_0 T,$$ with $N_0  =1+ \sum_{\ell=1}^K \mu_\ell^0$, when $T$ is large enough. 
% and denote by $\mathcal F_{j} = \{ f \in  \mathcal F;  \nu_\ell \leq \nu_{\ell}^0 + K N_0  + 2 j \epsilon_T, \, \forall \ell \leq K \}$.
 We then have the following theorem. 
\begin{theorem}\label{th:d1}
Consider the multivariate Hawkes process $(N^k)_{k=1,\ldots,K}$ observed on $[-A, T]$, with likelihood given by \eqref{loglikely}.
Let $\Pi$ be a prior distribution on ${\mathcal F}.$
Let $\epsilon_T$ be a positive sequence such that $\epsilon_T= o(1)$ and $$\log\log(T)\log^3(T)=o(T \epsilon_T^2).$$ For $B>0$, we consider
 $$B(\e_T, B):=\left\{(\nu_k,(h_{\ell,k})_{\ell})_k:\quad \max_k|\nu_k-\nu^0_k|\leq\e_T, \, \max_{\ell,k}\|h_{\ell,k}-h^0_{\ell,k}\|_2 \leq \e_T, \,\max_{\ell,k}\|h_{\ell,k} \|_\infty \leq B  \right\}$$
and assume following conditions are satisfied for $T$ large enough.
\begin{itemize}
\item[(i)] There exists $c_1>0$ and $B >0 $ such that 
$$ \Pi\left( B(\epsilon_T , B) \right)  \geq e^{ - c_1 T \epsilon_T^2}.$$
\item[(ii)] There exists a subset $\mathcal H_T \subset \mathcal H$, such that  $$ \frac{\Pi\left( \mathcal H_T^c \right)}{ \Pi\left( B(\epsilon_T , B) \right)}  \leq e^{- (2\kappa_T+3)T\epsilon_T^2},$$
where $\kappa_T:=\kappa\log(r_T^{-1}) \asymp \log \log T$, with $r_T$ defined in \eqref{deltaT} and $\kappa$ defined in \eqref{kappa}.
\item[(iii)]   There exist $\zeta_0>0$ and $x_0>0$ such that 
$$\log {\mathcal N}(  \zeta_0 \epsilon_T , \mathcal H_T , \| . \|_1) \leq  x_0 T   \epsilon_T^2.$$
\end{itemize}
Then, there exist $M>0$ and $C>0$ such that 
\begin{equation*}
\E_0\left[\Pi\left( d_{1,T}( f_0, f) > M\sqrt{\log \log T} \epsilon_T | N \right) \right]  \leq \frac{C\log\log(T)\log^3(T)}{T \epsilon_T^2}+\P_0 (\Omega_T^c) +o(1) = o(1).
\end{equation*}
\end{theorem}
Assumptions (i), (ii) and (iii) are very common in the literature about posterior convergence rates. As expressed by Assumption (ii), some conditions are required on the prior on $\mathcal H_T$ but not on $\mathcal F_T$. Except the usual concentration property of $\nu$ around $\nu^0$ expressed in the definition of $B(\e_T, B)$, which is in particular satisfied if $\nu$ has a positive continuous density with respect to Lebesgue measure,  we have no further condition on the tails of the distribution of $\nu$.
\begin{remark}\label{rem:supsorm}
As appears in the proof of Theorem \ref{th:d1},  the term $\sqrt{\log \log T}$ appearing in the posterior concentration rate can be dropped if $B(\epsilon_T, B)$ is replaced by 
$$B_\infty (\epsilon_T, B) = \left\{(\nu_k,(h_{\ell,k})_{\ell})_k:\quad \max_k|\nu_k-\nu^0_k|\leq\e_T, \, \max_{\ell,k}\|h_{\ell,k}-h^0_{\ell,k}\|_\infty \leq \e_T \right\},$$ in Assumption (i). In this case, $r_T=1/2$ in Assumption (ii) and $\kappa_T$ does not depend on $T$. This is used for instance in Section \ref{sec:sec:randomhistogram} to study random histograms priors whereas mixtures of Beta priors are controlled using the $\L_2$-norm. 
\end{remark}
Similarly to other general theorems on posterior concentration rates, we can consider some variants. Since the metric $d_{1,T}$ is stochastic, we cannot use slices in the form $d_{1,T} (f_0,f) \in (j \epsilon_T, (j+1)\epsilon_T)$ as in Theorem 1 of \citet{ghosal:vdv:07}, however we can consider other forms of slices, using a similar idea as in Theorem 5 of \citet{ghosal:vdv:mixture:07}. This is presented in the following theorem.
\begin{theorem}\label{th:slices}
Consider the setting and assumptions of Theorem~\ref{th:d1} except that assumption (iii) is replaced by the following one: There exists a sequence of sets $(\mathcal H_{T,i})_{i\geq 1} \subset \mathcal H$ with $\cup_i \mathcal H_{T,i}=\mathcal H_T$ and $\zeta_0>0$ such that 
\begin{equation}\label{cond:slices} 
\sum_{i=1}^\infty {\mathcal N}(\zeta_0 \epsilon_T , \mathcal H_{T,i} , \| . \|_1)\sqrt{\Pi(\mathcal H_{T,i}) } e^{- x_0 T\epsilon_T^2} = o(1),
\end{equation}
for some positive constant $x_0>0$.
Then, there exists $M>0$ such that 
$$
\E_0\left[\Pi\left( d_{1,T}( f_0, f) > M \sqrt{\log \log T} \epsilon_T | N \right) \right]   = o(1).
$$

\end{theorem}

The posterior concentration rates of Theorems \ref{th:d1} and \ref{th:slices} are in terms of the metric $d_{1,T}$ on the intensity functions, which are data dependent and therefore not completely satisfying to understand concentration around the objects of interest namely $f_0$. We now use Theorem \ref{th:d1} to provide a general result to derive a posterior concentration rate in terms of the $\L_1$-norm.
\begin{theorem}\label{th:L1}
Assume that the prior $\Pi$ satisfies following assumptions.
\begin{itemize}
\item[(i)] There exists $\varepsilon_T = o(1)$ such that $ \varepsilon_T\geq \delta_T$ (see the definition of $\Omega_T$) and $c_1>0$ such that 
\begin{equation*}
\E_0 \left[ \Pi\left( A_{\varepsilon_T}^c | N \right)\right] = o(1) \quad \& \quad \P_0\left( D_T < e^{-c_1 T \varepsilon_T^2 } \right) = o(1), 
\end{equation*}
where $D_T = \int_{\mathcal F} e^{L_T(f) - L_T(f_0)}d\Pi(f) $ and $A_{\varepsilon_T}=\{ f ; d_{1,T}(f_0,f) \leq \varepsilon_T\}$.
\item[(ii)] 
The prior on $\rho$ satisfies : for all $u_0>0$, when $T$ is large enough,
 \begin{equation}\label{uT}
 \Pi( \| \rho\| > 1 -u_0 (\log T)^{1/6} \varepsilon_T^{1/3} ) \leq e^{- 2c_1 T\varepsilon_T^2}.
 \end{equation}
 
%\textcolor{red}{The prior on $\rho$ satisfies
% $$ \Pi( \| \rho\| > 1 - u_T) \leq e^{- 2c_1 T\varepsilon_T^2 },$$
% for $T$ large enough, with $u_T\to 0$ such that
% \begin{equation}\label{uT}
%\varepsilon_T^2 \log T=o(u_T^6).
% \end{equation}
% }
\end{itemize}
Then, for any $w_T \rightarrow + \infty$,
\begin{equation}\label{L1conc}
\E_0\left[\Pi \left( \| f - f_0\|_1 > w_T\varepsilon_T | N\right)\right]=o(1).
\end{equation}
\end{theorem}
\begin{remark}
Condition (i) of Theorem \ref{th:L1} is in particular verified under the assumptions of Theorem~\ref{th:d1}, with $\varepsilon_T = M\epsilon_T \sqrt{\log \log T} $ for $M$ a constant.
\end{remark}
\begin{remark}
Compared to Theorem \ref{th:d1}, we also assume (ii), i.e. that the prior distribution puts very little mass near the boundary of space $\{f;\|\rho\|<1\}$. In particular, if under $\Pi$, $\|\rho\|$ has its support included in $[0,1-\epsilon]$ for a fixed small $\epsilon>0$ then \eqref{uT} is verified.
\end{remark}
A consequence of previous theorems is that the posterior mean $\hat f = \E^\pi[ f| N]$ is converging to $f_0$ at the rate $\varepsilon_T$, which is described by the following corollary. 
\begin{cor} \label{coro:L1hatf}
Under the assumptions of Theorem~\ref{th:d1} or Theorem \ref{th:slices}, together with \eqref{uT} with $\varepsilon_T = \sqrt{\log \log T} \epsilon_T$ and if 
 $ \int_{\mathcal F} \| f\|_1 d\Pi(f) < +\infty$, then for any $w_T \rightarrow +\infty$
$$\P_0 \left( \|\hat f - f_0 \|_1 > w_T\varepsilon_T \right) = o(1).$$
\end{cor}
The proof of Corollary \ref{coro:L1hatf} is given in Section \ref{pr:corL1}. We now illustrate these general results on specific prior models.
%%%%%%%%%%%%%%%%%%%%%%%%%%%%%
\subsection{Examples of prior models}\label{sec:prior-model}
The advantage of Theorems \ref{th:d1} and \ref{th:L1} is that the conditions required on the priors on the functions $h_{k,\ell}$ are quite standard, in particular  if the functions $h_{k,\ell}$ are parameterized in the following way
$$h_{k,\ell} = \rho_{k,\ell} \bar h_{k,\ell}, \quad \int_0^A \bar h_{k,\ell} (u) du = 1.$$
We thus consider priors on $\theta = (\nu_\ell, \rho_{k,\ell}, \bar h_{k,\ell}, k,\ell \leq K)$ following the scheme 
\begin{equation}
\begin{split}\label{schemeprior}
\nu_\ell \stackrel{iid}{\sim} \Pi_\nu , \quad \rho = (\rho_{k,\ell})_{k,\ell\leq K} \sim \Pi_\rho, \quad \bar h_{k,\ell} \stackrel{iid}{\sim} \Pi_h.
\end{split}
\end{equation}
We consider $\Pi_\nu$ absolutely continuous with respect to the Lebesgue measure on $\mathbb R_+$ with positive and continuous density $\pi_\nu$, 
$\Pi_\rho$ a probability distribution on the set of matrices with positive entries and spectral norm $\|\rho\|<1$, with positive density with respect to Lebesgue measures and satisfying \eqref{uT}. We now concentrate on  the nonparametric part, namely the prior distribution $\Pi_h$.
Then, from Theorems \ref{th:d1} and \ref{th:L1} it is enough that $\Pi_h$ satisfies for each $k,\ell\leq K$,
\begin{equation*}
\Pi_h \left( \|\bar h - \bar h_{k,\ell}^0\|_2 \leq \epsilon_T, \quad \|\bar h\|_\infty \leq B \right) \geq e^{-cT\epsilon_T^2 },  
\end{equation*}
for some $B>0$ and $c>0$ such that there exists $\mathcal F_{1,T}$ with $$\mathcal F_{1,T} \subset \left\{ h : [0,A]\rightarrow \mathbb R^+, \int_0^Ah(x)dx = 1\right\}$$ satisfying 
\begin{equation}\label{cond-prior}\Pi_h\left( \mathcal F_{1,T}^c \right) \leq e^{- CT\epsilon_T^2\log\log T}, \quad N( \zeta \epsilon_T; \mathcal F_{1,T};  \| . \|_1) \leq x_0 T \epsilon_T^2,
\end{equation}
 for $\zeta >0$, $x_0>0$ and  $C>0$ large enough. Note that from remark \ref{rem:supsorm}, if we have that for all $\ell, k $ 
$$ \Pi_h \left( \|\bar h - \bar h_{k,\ell}^0\|_\infty \leq \epsilon_T, \quad \|\bar h\|_\infty \leq B \right) \geq e^{-cT\epsilon_T^2 }$$
then it is enough to verify
\begin{equation}\label{cond-prior2}\Pi_h\left( \mathcal F_{1,T}^c \right) \leq e^{- CT\epsilon_T^2}, \quad N( \zeta \epsilon_T; \mathcal F_{1,T};  \| . \|_1) \leq x_0 T \epsilon_T^2,
\end{equation}
in place of  \eqref{cond-prior}. 
 
These conditions have been checked for a large selection of types of priors on the set of densities. We discuss here two cases: one based on random histograms,  these priors make sense in particular in the context of modeling neuronal interactions and the second based on mixtures of Betas, because it leads to adaptive posterior concentration rates over a large collection of functional classes. To simplify the presentation we assume that $A=1$ but generalization to any $A>0$ is straightforward.

\subsubsection{ Random histogram prior}\label{sec:sec:randomhistogram}
These priors are motivated by the neuronal application, where one is interested in characterizing time zones when neurons are or are not interacting (see Section~\ref{sec:numerical Hawkes}). 
Random histograms have been studied quite a lot recently for density estimation, both in semi and non parametric problems.
We consider two types of random histograms: regular partitions and random partitions histograms. 
Random  histogram priors are defined  by: for $J\geq 1$, 
 \begin{equation}\label{rand:hist} %{randpart:hist}
 \begin{split}
 \bar h_{w,t,J} = \delta \sum_{j=1}^J \frac{ w_j}{ t_j - t_{j-1}} \1_{I_j}, \quad I_j = (t_{j-1},t_j), \quad \sum_{j=1}^J w_j = 1, \quad \delta \sim  \mathcal{B}ern(p)
 \end{split}
 \end{equation}
and 
$$T_0=0 < t_1< \cdots < t_J=1.$$
In both cases, the prior is constructed in the following hierarchical manner: 
\begin{equation}\label{condhistreg}
\begin{split}
J & \sim \Pi_J , \quad e^{-c_1 x L_1(x)} \lesssim  \Pi_J( J = x) , \quad \Pi_J(J>x) \lesssim e^{-c_2 x L_1(x)}, \\ &L_1(x) = 1 \mbox{ or } L_1(x) =\log x\\
(w_1, \dots, w_J)| J &  \sim \Pi_w, %\mathcal D(\alpha_{1,J}, \cdots, \alpha_{J,J}), \quad c_1J^{-a}\leq \alpha_{i,J}\leq c_2, 
\end{split}
\end{equation}
where $c_1$ and $c_2$ are two positive constants. Denoting $\mathcal S_J$ the $J$-dimensional simplex,  we assume that the prior on $(w_1, \cdots , w_J)$  satisfies : for all  $M>0$, for all $w_0 \in \mathcal S_J$ with for any $j$, $w_{0j} \leq M/J$ and all $u>0$ small enough, there exists $c>0$ such that 
\begin{equation}\label{cond:w}
\Pi_w\left((w_{01}-u/J^2, w_{01} + u/J^2)\times\cdots\times(w_{0J}-u/J^2, w_{0J} + u/J^2)\right) > e^{- c J \log J}.
\end{equation}
 Many probability distributions on $\mathcal S_J$ satisfy \eqref{cond:w}. For instance, if $ \Pi_w$ is the Dirichlet distribution $ \mathcal D(\alpha_{1,J}, \cdots, \alpha_{J,J}) $ with $ c_3J^{-a}\leq \alpha_{i,J}\leq c_4$, for $a$, $c_3$ and $c_4$ three positive constants,  then \eqref{cond:w} holds, see for instance  \citet{castillo:rousseau:2015}. Also, consider the following hierarchical prior allowing some the of $w_j$'s to be equal to 0. 
Set $$ Z_j \stackrel{iid}{\sim} \text{Be}(p), \quad j \leq J, \quad s_z = \sum_{j=1}^J Z_j$$
and $(j_1, \cdots , j_{s_z})$ the indices corresponding to $Z_j =1$. Then, 
 \begin{equation*}
 \begin{split}
(w_{j_1}, \cdots w_{j_{s_z}}) &\sim \mathcal D(\alpha_{1,J},\cdots, \alpha_{s_z,J} ), \quad  c_3J^{-a}\leq \alpha_{i,J}\leq c_4\\
 w_j  &= 0 \quad \text{if } Z_j = 0.
\end{split}
\end{equation*}
Regular partition histograms correspond to $t_j = j/J$ for $j \leq J$, in which case we write  $\bar h_{w,J}$ instead of $ \bar h_{w,t,J}$;  while in random partition histograms we put a prior on $(t_1, \cdots, t_J)$. 
% \begin{equation}\label{reg:hist}
% \begin{split}
%\bar h_{w,J} (x) = \delta \sum_{j=1}^J J w_j\1_{I_j}(x), \quad I_j = ((j-1)/J, j/J), \quad \sum_{j=1}^J w_j = 1, \quad \delta \sim \text{Be}(p).
% \end{split}
% \end{equation}
We now consider H\"older balls of smoothness $\beta$ and radius $L_0$, denoted $\mathcal H(\beta, L_0)$, and prove that the posterior concentration rate associated with both types of histogram  priors is  bounded by $\epsilon_T=\epsilon_0(\log T/T)^{\beta/(2\beta+1)}$ for $0<\beta\leq 1$, where $\epsilon_0$ is a constant large enough.
From Remark \ref{rem:supsorm}, we use the version of assumption (i) based on $$B_\infty (\epsilon_T, B) = \left\{(\nu_k,(h_{\ell,k})_{\ell})_k:\quad \max_k|\nu_k-\nu^0_k|\leq\e_T, \, \max_{\ell,k}\|h_{\ell,k}-h^0_{\ell,k}\|_\infty \leq \e_T \right\},$$
and need to verify \eqref{cond-prior2}.
Then applying Lemma 4 of the supplementary material of \citet{castillo:rousseau:2015}, we obtain for all  $\bar h_0 \in \mathcal H(\beta, L_0)$ and  if $\bar h_0$ is not the null function
 $$ \Pi\left( \| \bar h_{w,J}  - \bar h_0\|_\infty \leq 2L_0 J^{-\beta} | J\right) \gtrsim pe^{ - c J \log T} $$
 for some $c>0$ and $\Pi_J(J = J_0 \lfloor (T/\log T)\rfloor^{1/(2\beta+1)} )\gtrsim e^{ -  c_1 J_0(T/\log T)^{1/(2\beta+1)}L_1(T) }$ if $J_0$ is a constant.
  If $\bar h_0= 0$ 
 then 
 $$ \Pi\left( \| \bar h_{w,J}  - \bar h_0\|_\infty =0 \right) = 1-p.$$ 
 This thus  implies that $ \Pi\left(B_\infty (\epsilon_T, B)  \right) \gtrsim pe^{ - c' T\epsilon_T^2} $ for some $c'>0$. This result holds both for the regular grid and random grid histograms with a prior on the grid points  $(t_1, \cdots, t_J)$ given by $(u_1, \cdots, u_J) \sim \mathcal D(\alpha, \cdots, \alpha)$ with $u_j = t_j - t_{j-1}$. 
Then condition \eqref{uT} is verified if $\Pi(\|\rho\| >1 - u)\lesssim  e^{-a' u^{-a}}$ with $a>3/\beta$ and $a'>0$, for $u$ small enough. This condition holds for any $\beta \in (0,1]$ if there exist $a', \tau >0$ such that when $u$ is small enough
\begin{equation}\label{uTadap}
\Pi\left( \|\rho\| > 1 -u\right) \lesssim e^{-a' e^{-1/u^\tau}}.
\end{equation}
 Moreover, set $\mathcal F_{1,T} = \{  \bar h_{w,J}, J \leq J_1(T/\log T)^{1/(2\beta+1)}\}$ for $J_1$ a constant, then for all $\zeta >0$,
$$N(\zeta \epsilon_T, \mathcal F_{1,T}, \| . \|_1) \lesssim J_1(T/\log T)^{1/(2\beta+1)} \log T.$$
Therefore, \eqref{cond-prior2} is checked. We finally obtain the following corollary. 
\begin{cor}[regular partition] \label{coro:reghist}
Under the random histogram prior  \eqref{rand:hist} based on a regular partition and verifying \eqref{condhistreg} and \eqref{cond:w} and if \eqref{uTadap} is satisfied, % with $\varepsilon_T=\epsilon_0 (T/\log T)^{-\beta/(2\beta+1)}$
%for a
%$$\Pi_\rho\left( \| \rho\|> 1 - u \right)\leq e^{-c_0 T \sqrt{\log T} u^{1/3}} $$ for all $u$ small enough, 
 then if for any $k,\ell=1,\ldots,K$, $h_{k,\ell }^0$ belongs to $\mathcal H(\beta,L)$ for $0<\beta \leq 1$, then for any $w_T\to +\infty$,
 $$ \E_0\left[\Pi \left( \| f - f_0\|_1 > w_T(T/\log T)^{-\beta/(2\beta+1)} | N\right)\right] = o(1) .$$
\end{cor}
To extend this result to the case of random partition histogram priors we  consider the same prior on $(J, w_1, \cdots , w_J)$  as in \eqref{condhistreg} and the following condition on the prior on $\underline t =(t_1, \cdots, t_K)$. Writing $u_1 = t_1$, $u_j  = t_j - t_{j-1}$, we have that $\underline u = (u_1, \cdots, u_J)  $ belongs to the $J$-dimensional simplex $\mathcal S_J$ and we consider a Dirichlet distribution on $(u_1, \cdots , u_J)$, $\mathcal D(\alpha, \cdots, \alpha)$ with $\alpha \geq 6$. 
\begin{cor}\label{cor:randparthist}
Consider the random histogram prior  \eqref{rand:hist} based on random partition with a prior on $\underline{w}$ satisfying \eqref{condhistreg} and \eqref{cond:w} and with a Dirichlet prior on $\underline u = (t_{j}- t_{j-1}, j \leq J)$, with parameter $\alpha \geq 6$. If \eqref{uTadap} is satisfied, % with $\varepsilon_T=\epsilon_0 (T/\log T)^{-\beta/(2\beta+1)}$
%for a
%$$\Pi_\rho\left( \| \rho\|> 1 - u \right)\leq e^{-c_0 T \sqrt{\log T} u^{1/3}} $$ for all $u$ small enough, 
 then if for any $k,\ell=1,\ldots,K$, $h_{k,\ell }^0$ belongs to $\mathcal H(\beta,L)$ for $0<\beta \leq 1$, then for any $w_T\to +\infty$,
 $$ \E_0\left[\Pi \left( \| f - f_0\|_1 > w_T(T/\log T)^{-\beta/(2\beta+1)} | N\right)\right] = o(1) .$$
\end{cor} 
The proof of this corollary is given in Section~\ref{sec:Proof:bayes}. In the following section, we consider another family of priors suited for smooth functions $h_{k,\ell}$ and based on mixtures of Beta distributions. 
\subsubsection{Mixtures of Betas} \label{sec:mixbeta}
The following family of prior distributions is inspired by \citet{rousseau:09}. Consider functions 
\begin{equation*}
h_{k,\ell} = \rho_{k,\ell}\left( \int_0^1 g_{\alpha_{k,\ell}, \epsilon}  dM_{k,\ell}(\epsilon)\right)_+, \quad g_{\alpha, \epsilon} (x) =\frac{ \Gamma(\alpha /(\epsilon (1-\epsilon)))}{ \Gamma(\alpha/\epsilon)\Gamma(\alpha/(1-\epsilon))} x^{\frac{\alpha}{1-\epsilon}-1}(1-x)^{ \frac{\alpha}{\epsilon}-1}
\end{equation*}
where $M_{k,\ell}$ are bounded signed measures on $[0,1]$ such that $|M_{k,\ell}|=1$. In other words the above functions are the positive parts of mixtures of Betas distributions with parameterization $(\alpha/\epsilon, \alpha/(1-\epsilon))$ so that $\epsilon$ is the mean parameter. The mixing random measures $M_{k,\ell}$ are allowed to be negative. The reason for allowing $M_{k,\ell}$ to be negative is that $h_{k,\ell}$ is then allowed to be null on sets with positive Lebesgue measure. 
The prior is then constructed in the following way. Writing $h_{k,\ell} = \rho_{k,\ell} \tilde h_{k,\ell} $ we define a prior on $\tilde h_{k,\ell}$ via a prior on $M_{k,\ell}$ and on $\alpha_{k,\ell}$.  In particular we assume that $M_{k,\ell} \stackrel{iid}{\sim} \Pi_M$ and $\alpha_{k,\ell} \stackrel{iid}{\sim} \pi_\alpha$. As in \citet{rousseau:09} we consider a prior on $\alpha$ absolutely continuous with respect to Lebesgue measure and with density satisfying: there exists $b_1 , c_1, c_2, c_3, A,  C >0$ such that for all $u$ large enough, 
\begin{equation}\label{cond:alpha}
\begin{split}
\pi_\alpha(c_1 u < \alpha < c_2u ) &\geq Ce^{ -b_1 u^{1/2}}\\
\pi_\alpha(\alpha < e^{-Au}) + \pi_\alpha (\alpha > c_3 u ) &\leq C e^{ -b_1 u^{1/2}}.
\end{split}
\end{equation}
There are many ways to construct discrete signed measures on $[0,1]$, for instance, writing
 \begin{equation}
 M = \sum_{j=1}^J r_j p_j \delta_{\epsilon_j},
 \end{equation}
the prior on $M$ is then defined by 
$ J \sim \Pi_J$ and conditionally on $J$,  
$$r_j\stackrel{iid}{\sim} \text{Ra}(1/2), \quad \epsilon_j \stackrel{iid}{\sim} G_\epsilon, \quad (p_1, \cdots, p_J) \sim \mathcal D(a_1, \cdots, a_J),$$
where Ra denotes the Rademacher distribution taking values $\{-1, 1\}$ each with probability $1/2$. 
Assume that $G_\epsilon$ has  positive continuous density on $[0,1]$ and that there exists $A_0>0$  such that $\sum_{j=1}^J a_j \leq A_0$. We have the following corollary.
%-------------------------------
\begin{cor}\label{prop:mix:beta}
Consider a prior as described above. Assume that for all $k, \ell\leq K$ $h_{k,\ell}^0  = (g_{k,\ell}^0)_+$ for some functions $g_{k,\ell}^0 \in \mathcal H(\beta, L_0)$ with $\beta >0$. If condition \eqref{uTadap} holds
and if $G_\epsilon$ has density with respect to Lebesgue measure verifying  
 $$  x^{ A_1}(1-x)^{A_1} \lesssim g_\epsilon(x) \lesssim x^3(1-x)^3, \quad \text{for some } \quad A_1\geq 3,$$
 then, for any $w_T\to +\infty$,
\begin{equation*}
\E_0 \left[\Pi( \|f - f_0\|_1 >  w_TT^{-\beta/(2\beta+1)} (\log T)^{5\beta/(4\beta+2)} \sqrt{\log \log T } | N)\right]  = o(1).
\end{equation*}
\end{cor}
Note that in the context of density estimation, $T^{-\beta/(2\beta+1)}$ is the minimax rate and we expect that it is the same for Hawkes processes. 
%%%%%%%%%%%%%%%%%%%%%%%%%%
%%%%%%%%%%%%%%%%%%%%%%%%%%
% !TEX root =BNP_Hawkes-submitted.tex

%%%%%%%%%%%%%%%%%%%%%%%%
%%%%%%%%%%%%%%%%%%%%%%%%
\section{Numerical illustration in the neuroscience context}\label{sec:numerical Hawkes}
It's now well-known that neurons receive and transmit signals as electrical impulses called {\it action potentials}. Although action potentials can vary somewhat in duration, amplitude and shape, they are typically treated as identical stereotyped events in neural coding studies. Therefore, an action potential sequence, or {\it spike train}, can be characterized simply by a series of all-or-none point events in time. Multivariate Hawkes processes have been used in neuroscience to model spike trains of several neurons and in particular to model functional connectivity between them through mutual excitation or inhibition \citep{JNM}. 
In this section, we conduct a simulation study mimicking the neural context, through appropriate choices of parameters. The protocol is similar to the setting proposed in Section 6 of \citep{HRR}.
%Hawkes processes are used in neurosciences
%to model action potentials trains of neurons. In a few words, neurons communicate through
%sequences of action potentials. Contemporary models assume that the information is conveyed by the
%action potentials' times of occurrence rather than by the action potentials' waveforms. The series of
%occurrences times are assumed to be the realization of a non homogeneous point process. The multivariate
%Hawkes processes allow to take into account the dependences/interactions between neurons
%(namely excitation or inhibition) \citep{HRR}. In this section, we conduct a simulation study, choosing the
%parameters so that the simulated data mimic action potentials trains.
%%%%%%%%%%%%%%%%%
\subsection{Simulation scenarios}
We consider three simulation scenarios involving respectively $K=2$ and $K=8$ neurons. The scenarios are roughly similar to the one tested in \cite{HRR}.  Following the notations introduced in the previous sections, for any $(k, \ell) \in \{1,\dots K\}^2$,  $h_{k, \ell}$ denotes  the interaction function of neuron $k$ over  neuron $\ell$.  We now describe the three scenarios. 
 The  upper bound of each $h_{k, \ell}$'s support, denoted $[0,A]$ is set equal to   $A = 0.04$ seconds. 
\begin{itemize}
\item \emph{Scenario 1}: We first consider  $K=2$ neurons and piecewise constant interactions:  
$$h_{1,1} = 30 \cdot \1_{(0,0.02]}, \quad h_{2,1} = 30  \cdot \1_{(0,0.01]} ,\quad h_{1,2} = 30 \cdot  \1_{(0.01,0.02]},\quad h_{2,2} = 0.$$ 
\item \emph{Scenario 2}: In this scenario, we mimic $K=8$ neurons belonging to three independent groups. The non-null interactions are the piecewise constant functions defined as: 
$$h_{2,1}=h_{3,1}=h_{2,2}=h_{1,3}=h_{2,3}=h_{8,5}=h_{5,6}=h_{6,7}=h_{7,8}=30 \cdot \1_{(0,0.02]}.$$ 
In Figure \ref{fig: true graph}, we plot the subsequent interactions directed graph between the $8$ neurons: the  vertices represent the $K$ neurons and  an oriented edge is plotted  from vertex $k$ to vertex $\ell$ if the interaction function $h_{k, \ell}$ is non-null. 
\begin{figure}
\centering
\includegraphics[width=0.3\textwidth]{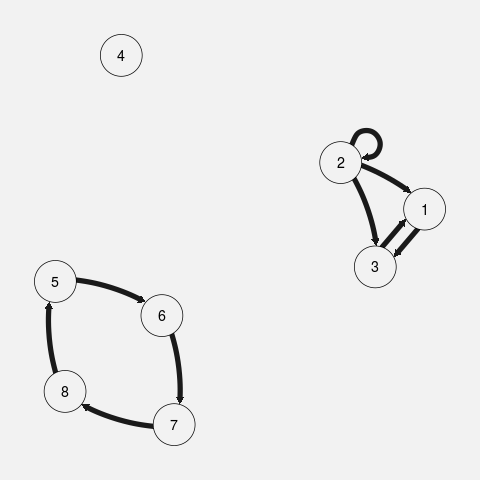}
\caption{\textbf{Scenario 2}. True interaction graph between the $K=8$ neurons. 
A directed edge is plotted  from vertex $\ell$ to vertex $k$ if the  interaction functions $h_{k, \ell}$ is non-null. }
\label{fig: true graph}
\end{figure}
\item \emph{Scenario 3}: Setting $K=2$, we consider non piecewise constant interactions functions defined as: 
 \begin{equation*}
\begin{array}{cclccl}
h_{1,1}(t) &=&100\cdot e^{-100 t}\1_{(0,0.04]}(t), &h_{2,1}(t) &=& 30  \cdot \1_{(0,0.02]}(t) \\
h_{1,2}(t) &=& \frac{1}{ 2 \times0.004\sqrt{2\pi}}e^ {-\frac{(t-0.02)^2}{2 \times 0.004^2}} \cdot  \1_{(0,0.04]}(t), \quad & h_{2,2}(t) &=& 0.
\end{array}
\end{equation*}
\end{itemize}
In all the scenarios, we consider $\nu_{\ell} = 20, \ \ell=1,\dots,K$. For each scenario, we simulate $25$ datasets on the time interval $[0, 22]$ seconds. The Bayesian  inference is performed considering  recordings  on  three possible periods  of  length $T=5$ seconds, $T=10$ seconds and $T=20$ seconds.  For any dataset, we remove the initial period of $2$ seconds --corresponding to $50$ times the length of the support of the $h_{k, \ell}$-functions, assuming that, after this period, the Hawkes processes  have reached their stationary  distribution. %Note that the  parameters $(\nu_{\ell},h_{k, \ell})_{l,k \in \{1,\dots,K\}}$ have been chosen such that, in any scenario,  approximatively $100$ points are recorded on each neuron by period  of two seconds. 
%%%%%%%%%%%%%%%%
\subsection{Prior distribution on $f= (\nu_{\ell},h_{k, \ell})_{l,k \in \{1,\dots,K\}}$}
We use the prior distribution described in Section \ref{sec:prior-model} setting a log-prior distribution on the $\nu_{\ell}$'s of parameter $\mu_\nu, s^2_{\nu}$.
%We use the prior distributions on  on $f = (\nu_{\ell},h_{k, \ell})_{l,k \in \{1,\dots,K\}}$   described in Section \ref{sec:prior-model} setting 
%
%
%$\E[  \log \nu_{\ell}] = 3$ and $\V[\log \nu_{\ell}] = 1$ (resulting into $\mathbb{E}[\nu_{\ell}]\approx 34$, $\sigma[\nu_{\ell}]\approx   47$).  
About the interaction functions  $\left(h_{k, \ell}\right)_{k, \ell \in \{1,\dots,K\}}$, the prior distribution is defined on the set of  piecewise constant functions, $h_{k, \ell}$ being written as follows: 
\begin{equation}\label{eq:form hkl} 
h_{k, \ell}(t) =  \delta^{(k, \ell)} \sum_{j=1}^{J^{(k,\ell)}} \beta^{(k, \ell)}_{j}\1_{[t^{(k, \ell)}_{j-1},t^{(k, \ell)}_{j}]}(t)
\end{equation}
with $t^{(k, \ell)}_{0} = 0$ and  $t^{(k, \ell)}_{J^{(k,\ell)}} = A$.  
Using the notations in Section \ref{sec:prior-model}, we have $\beta^{(k,\ell)}_j =\rho^{(k,\ell)}  \omega^{(k, \ell)}_{j}$. Here,
$\delta^{(k, \ell)}$ is a global parameter  of nullity for $h_{k, \ell}$ :   for all $(k, \ell) \in \{1,\dots, K\}^2$, 
\begin{equation}\label{eq: prior delta} \delta^{(k, \ell)} \sim_{i.i.d} \mathcal{B}ern(p).\end{equation}
%The idea of this parameter is to encourage complete nullity for the  interaction functions and so a sparse interaction network between neurons. 
%\item For non-null intensities functions, we work in a non-parametric framework.  
For all $(k, \ell) \in \{1,\dots, K\}^2$, the number  of steps $(J^{(k,\ell)})$ follows  a translated Poisson prior distribution: 
\begin{equation}\label{eq: prior M} J^{(k,\ell)} |\{\delta^{(k, \ell)}=1\}  \sim_{i.i.d.} 1+ \mathcal{P}(\eta). \end{equation}
To minimize the influence of $\eta$ on the  posterior distribution, we consider an hyperprior distribution on the hyperparameter $\eta$: 
\begin{equation}\label{eq: prior lambda}  \eta \sim \Gamma(a_{\eta}, b_{\eta}). \end{equation}
Given $J^{(k,\ell)}$, we consider a spike and slab prior distribution on $(\beta^{(k, \ell)}_{j})_{j=1,\dots, J^{(k,\ell)}} $. Let    $Z^{(k, \ell)}_{j}  \in \{0,1\}$ denote a sign indicator for each step, we set:   $\forall j \in\{1,\dots, J^{(k,\ell)}\}$: 
\begin{equation}\label{eq: prior alpha}
\begin{array}{ccl}
\P\left(Z^{(k, \ell)}_{j} =z | \delta^{(k, \ell)}=1\right) &=& \pi_{z}, \quad \forall z \in \{0,1\}\\
\beta^{(k, \ell)}_{j} | \{ \delta^{(k, \ell)}=1\}& \sim& Z^{(k, \ell)}_{j} \times \log \mathcal{N}(\mu_{\beta}, s^2_{\beta}).
\end{array}
\end{equation}
%\item On $(s_{k, \ell}^ {(j)})_{j=0\dots J_{\ell k}} $, we consider two possible prior distributions. In the first one, referred as \emph{the regular prior} in the following,  $(s_{k, \ell}^ {(j)})_{j=0\dots J_{\ell k}} $ is set equal to a regular partition of $[0,A]$: 
We consider two prior distributions on $(t_j^{(k,\ell)})_{j=1\dots J^{(k,\ell)}}$. The first one (refered as the Regular histogram prior) is a regular partition of $[0,A]$:  
\begin{equation}\label{eq: s regular}
t^{(k, \ell)}_{j} = \frac{j}{J^{(k,\ell)}} A \quad \quad  \forall j=0,\dots, J^{(k,\ell)} . 
\end{equation}
The second prior distribution is refered as random histogram prior and specifies:  
\begin{equation}\label{eq: s moving}
\begin{array}{ccl}
(u_1, \dots, u_{J^{(k,\ell)}}) &\sim& \mathcal{D}(\alpha'_1, \dots \alpha'_{J^{(k,\ell)}})\\
t^{(k, \ell)}_ {0} &=&0\\
t^{(k, \ell)}_{j} &=& A  \sum_{r=1}^ j u_r, \quad \forall j=1,\dots, J^{(k,\ell)}
\end{array}
\end{equation}
%\end{itemize}
In the simulations studies, we set the following hyperparameters: 
\begin{equation*}
\begin{array}{cclcccl}
\mu_{\beta} &=& 3.5, &s_{\beta}&=&1\\
\mu_{\nu} &=& 3.5, &s_{\nu}&=&1\\
\P(Z^{(k,\ell)}_{j}=1) &=& 1/2,&  \P(\delta^{(k,\ell)}=1)=p&=&1/2\\
\alpha'_j&=&2, \quad \forall j\\
\end{array}
\end{equation*}
%%%%%%%%%%%%%%%%%%
\subsection{Posterior sampling}
The posterior distribution is sampled using a standard Reversible-jump Markov chain Monte Carlo.  % For the sake of reading, we introduce  the following notations : 
%\begin{equation}
%\begin{array}{cclccl}
%\boldsymbol{\nu}^{(c)} &=& (\nu^{c}_1, \dots, \nu^{(c)}_K), & \boldsymbol{h} &=&   (h_{k, \ell})_{ (k, \ell) \in \{1,\dots, K\}^2}\\
%\end{array}
%\end{equation}
\noindent Considering the current parameter $(\boldsymbol{\nu}, \boldsymbol{h})$,  $\boldsymbol{\nu}^{(c)}$ is proposed using a Metropolis-adjusted Langevin proposal. %:   ${\displaystyle \boldsymbol{\nu}^{(c)} :=\boldsymbol{\nu} +\tau\left[\nabla \log \pi(\boldsymbol{\nu})+ \nabla \log L\left(\boldsymbol{\nu}, \boldsymbol{h}\right)\right]+{\sqrt {2\tau }}\xi ^{(i)},}$  
%where $\xi ^{(i)} \sim \mathcal{N}(\boldsymbol 0_K, \boldsymbol I_K)$. 
  For a fixed $J^{(k,\ell)}$, the heights $\beta^{(k,\ell)}_{j}$ are proposed using a random walk proposing null or non-null candidates.  Changes in the number of steps $J^{(k,\ell)}$ are proposed by standard birth and death moves \citep{Green1995}.  In this simulation study, we generate chains of length $30000$ removing the first $10000$ burn-in iterations. The algorithm is implemented in \textsf{R} on an Intel(R) Xeon(R) CPU E5-1650 v3 @ 3.50GHz.

%\noindent The algorithm is implemented in \textsf{R}. We  generated MCMC chains of length  $N = 30000$. 

% The posterior is sampled for both prior distributions on $(\textbf{s}_{\ell,l})_{ (k, \ell) \in \{1,\dots, K\}^2}$ (see equations \ref{eq: s moving} and  \ref{eq: s regular} ). 

%

\vspace{1em}

\noindent The computation times (mean over the $25$ datasets) are given in Table \ref{tab : compu time}. First note that the computation time increases roughly as a linear function of $T$. This is due to the fact that the heavier task in the algorithm is the integration of the conditional likelihood and the computation time of this operation is roughly a linear function of the length of the integration (observation) time interval. Besides, because we implemented  a Reversible Jumps algorithm, the computation time is a stochastic quantity: the algorithm can explore parts of the domain where the number of bins $J_{\ell k}$ is large, thus increasing the computation time. This point can explain the unexpected computation times for $K=2$.  Moreover, we remark that the computation time explodes as $K$ increases (due to the  fact that $K^2$ intensity functions have to be estimated), reaching computation times greater than  a day.

\begin{table}[ht]
\centering

\begin{tabular}{lcccc}
\hline
 &  \multicolumn{2}{c}{K=2}  & K=8 & K=2 with  smooth $h_{k, \ell}$ \\ 
 \hline

 Prior on $t$ &Regular & Random  & Regular & Random  \\
 \hline
  T=5 & 1508.44 & 1002.45 &  & 823.84 \\ 
 T=10 & 1383.72 & 1459.55 & 37225.19 & 1284.93 \\ 
 T=20 &  2529.19 & 2602.48 & 49580.18 & 1897.17 \\ 
   \hline
\end{tabular}
\caption{Mean computation time  (in seconds) of the MCMC algorithms as a function of the scenario, the observation time interval and the prior distribution on $s$. The mean is computed  over the $25$ simulated datasets} 
\label{tab : compu time}
\end{table}

\subsection{Results}

We describe here the results for each scenario. We first present the $\L_1$-distances on  $\lambda^{k}$  and $h_{k, \ell}$ for all $3$ the scenarios, all three length observation times $T$ and the two prior distributions. In Table \ref{tab : L1 distance}, we show the estimated $\L_1$-distances on  $\lambda^{k}$  and $h_{k, \ell}$. More precisely, we evaluate the $\L_1$-distances on the interactions functions
$$ D^{(1)}=  \frac{1}{25} \sum_{sim = 1}^{25}  \widehat{\E} \left[ \frac{1}{K^2} \sum_{ k, \ell=1}^{K} \left\| h_{k,\ell}- h^{0}_{k,\ell}\right \|_1 \biggr| (N^{sim}_t)_{t \in [0,T]}\right]$$ and 
the following stochastic distance : 
$$D^{(2)}  = \frac{1}{25} \sum_{sim=1}^{25} \widehat{\E}\left[  d_{1,T} ( f, f^{0}) )\biggr| (N^{sim}_t)_{t \in [0,T]}  \right],$$
 where $f^{0}$ is the true set of parameters, $d_{1,T}(f,f^0)$ has been defined in Section ~\ref{sec:notation} and the posterior expectations are approximated by Monte Carlo method using the outputs of the Reversible Jumps algorithm.

As expected, the error decreases as $T$ increases. As we will detail later, the random histogram prior  gives better results than the regular prior. Finally, we perform better when the true interaction function $(h_{k, \ell})$ are step functions (due to the form of the prior distribution).   

%\textcolor{red}{ Sophie tu divises par T dans ta distance L1 sur lambda? peut \^etre qu'on devrait pour rendre les resultats plus lisibles}
%\textcolor{blue}{ Oui je divisais bien. J'ai d\'efini plus clairement les quantit\'es dans le tableau en utlisant vos notations}

\begin{table}[ht]
\centering
\begin{tabular}{llcccc}
\cline{2-6}

& &  \multicolumn{2}{c}{K=2}  & K=8 & K=2 with  smooth $h_{k, \ell}$ \\ 
\cline{2-6}
& Prior&Regular & random & Regular & random \\
 \hline
  % \hline

\multirow{3}{*}{$D^(1) $:  stochastic distances} &  T=5 & 11.59 & 9.59 &&  11.75   \\ 
 & T=10 & 7.49 & 6.32 & 5.65 & 9.48 \\ 
 & T=20 & 5.40 & 4.11 & 3.17 & 7.9  \\ 
   \hline
 %  \hline

%1 & \\ 
  %2 &  \\ 
  %3 &8 \\ 
  
\multirow{3}{*}{$D^(2) $: distances on $h_{k, \ell}$} &  T=5 &0.1423 & 0.0996 &  & 0.1431 \\ 
 & T=10 &0.0844 & 0.0578 & 0.1199 & 0.1131 \\ 
 & T=20 & 0.0564 & 0.0336 & 0.0616 & 0.0945 \\ 
   \hline
\end{tabular}
\caption{$\L_1$-distances on  $h_{k, \ell}$ and  $\lambda^k$} 
\label{tab : L1 distance}
\end{table}
%\textcolor{red}{Calculs en train de tourner pour $\lambda^*$
%}

\subsubsection{Results for scenario 1: $K=2$ with step functions}
When $K=2$, we estimate  the parameters using both \emph{regular} and \emph{random} prior distributions on $(t^{(k, \ell)}_j)$ (equations (\ref{eq: s regular}) and (\ref{eq: s moving})). 
One typical posterior distribution of $\nu_{\ell}$ is given in Figure \ref{fig:post nu M=2} (left), for a randomly chosen dataset, clearly showing a smaller variance when the length of the observation interval increases.  We also present the global estimation results, over the 25 simulated datasets. The distribution of the posterior mean estimators for $(\nu_1,\nu_2)$ computed for the 25 simulated datasets $\left(\widehat{\mathbb{E}}\left[\nu_{\ell}| (N^{sim}_t)_{t \in [0,T]}\right]\right)_{sim=1 \dots 25}$ is given in Figure \ref{fig:post nu M=2}  on the right panel, showing an expected  decreasing variance for the estimator as $T$ increases.  On the top panels the posterior is based on the regular grid prior while on the bottom the posterior is based on the random (grid) histogram prior:  the results are equivalent. 

\vspace{1em}

\noindent About  the estimation of the interaction functions, for the same given dataset, the estimation of the $h_{k, \ell}$  is plotted  in Figure \ref{fig:post h M=2} (upper panel)  for the regular prior, with its credible interval. Its corresponding estimation with the random prior is  given in Figure \ref{fig:post h M=2} (bottom panel). For both prior distributions, the functions are globally well estimated, showing a clear concentration when $T$ increases.  The regions where the interaction functions are null are also well identified. The estimation given with the random histogram prior is in general better than the one supplied  by the regular prior. This may be due to several factors. First, the random histogram prior leads to a sparser estimation than the regular one. Secondly, it is easier to design a proposal move in the Reversible Jump algorithm  in the  former case than in the latter context. 

\noindent  Moreover,  the interaction graph is perfectly inferred since the posterior probability for $\delta^{(2,2)}$ to be $0$ is almost $1$. For the 25 dataset, we estimate the posterior probabilities $\widehat{\P}(\delta^{(k, \ell)}=1| (N^{sim}_t)_{t \in [0,T]} )$ for $k, \ell =1,2$ and $sim=1\dots 25$. In Table \ref{tab:post delta M=2}, we display the mean of these posterior quantities. Even for the shorter observation time interval ($T=5)$ these quantities --defining completely the connexion graph--  are well recovered. These results are improved when $T$ increases.  Once again, the  random histogram prior (\ref{eq: s moving})  gives better results.

 \begin{figure}
  \centering
\includegraphics[width=\textwidth,height=0.25\textheight] {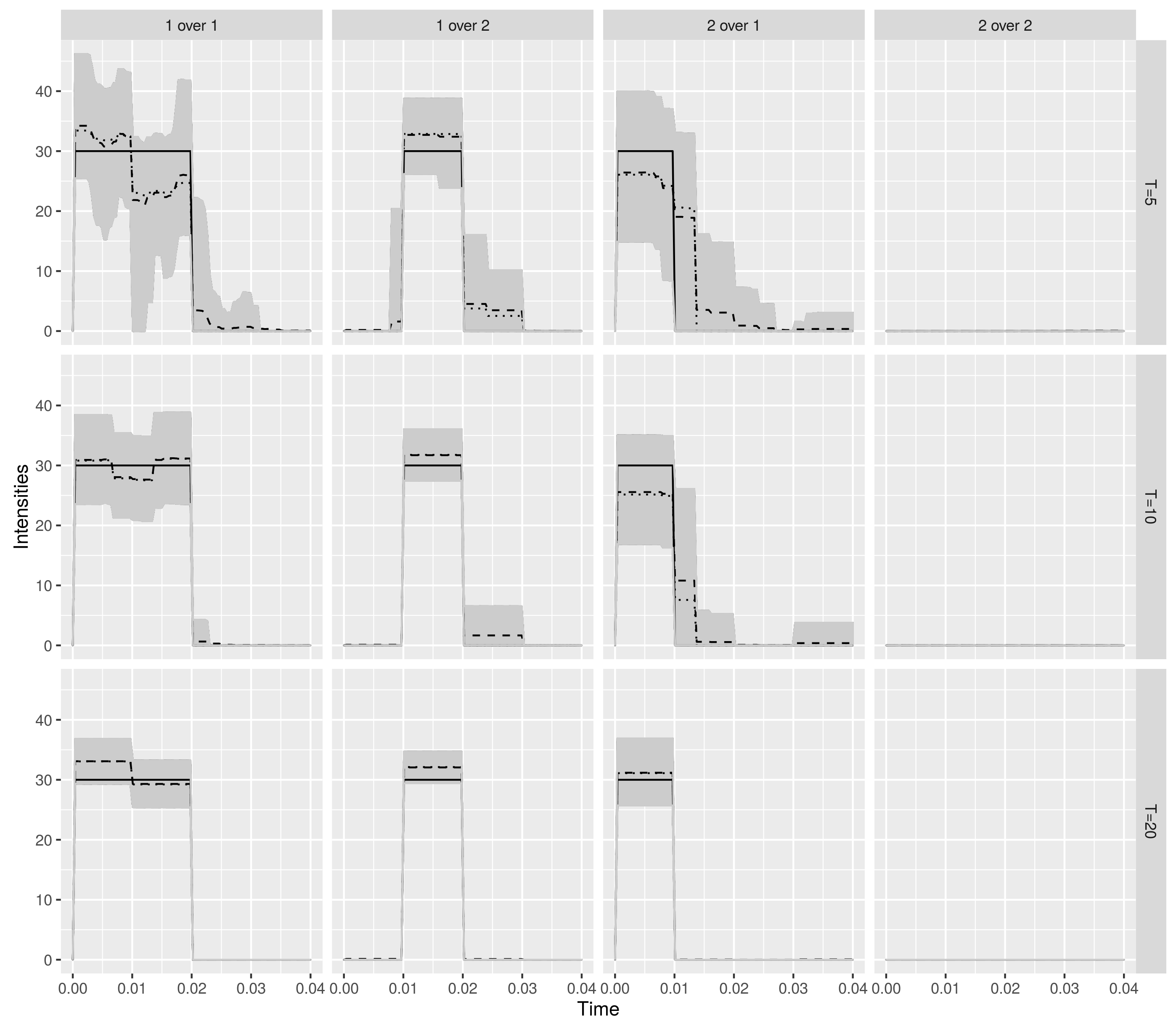}
\includegraphics[width=\textwidth,height=0.25\textheight]{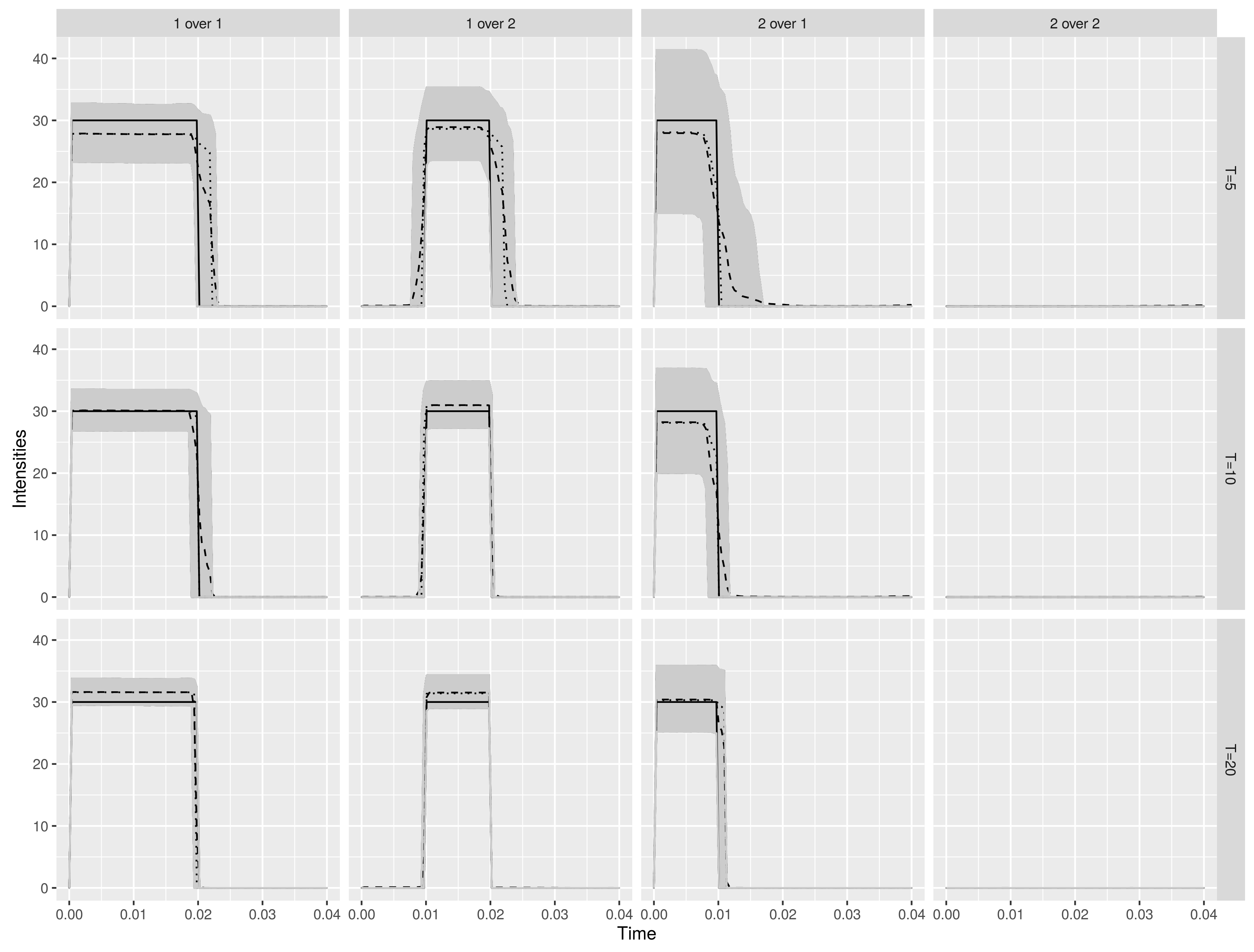}

 \caption{\textbf{Results for scenario 1}: Estimation of the $(h_{k, \ell})_{k, \ell=1,2}$ using the regular prior (upper panel) and the random histogram prior (bottom panel). The gray region indicates the credible region for $h_{k, \ell}(t)$ (delimited by the $5\%$ and $95\%$ percentiles of the posterior distribution).  The true  $h_{k, \ell}$ is in plain line, the posterior expectation and posterior median  for  $h_{k, \ell}(t)$ are in dotted and dashed lines  respectively. }
\label{fig:post h M=2}
\end{figure}

\begin{figure}
  \centering
 \begin{minipage}{0.5\linewidth}
\centering
\includegraphics[width=\textwidth]{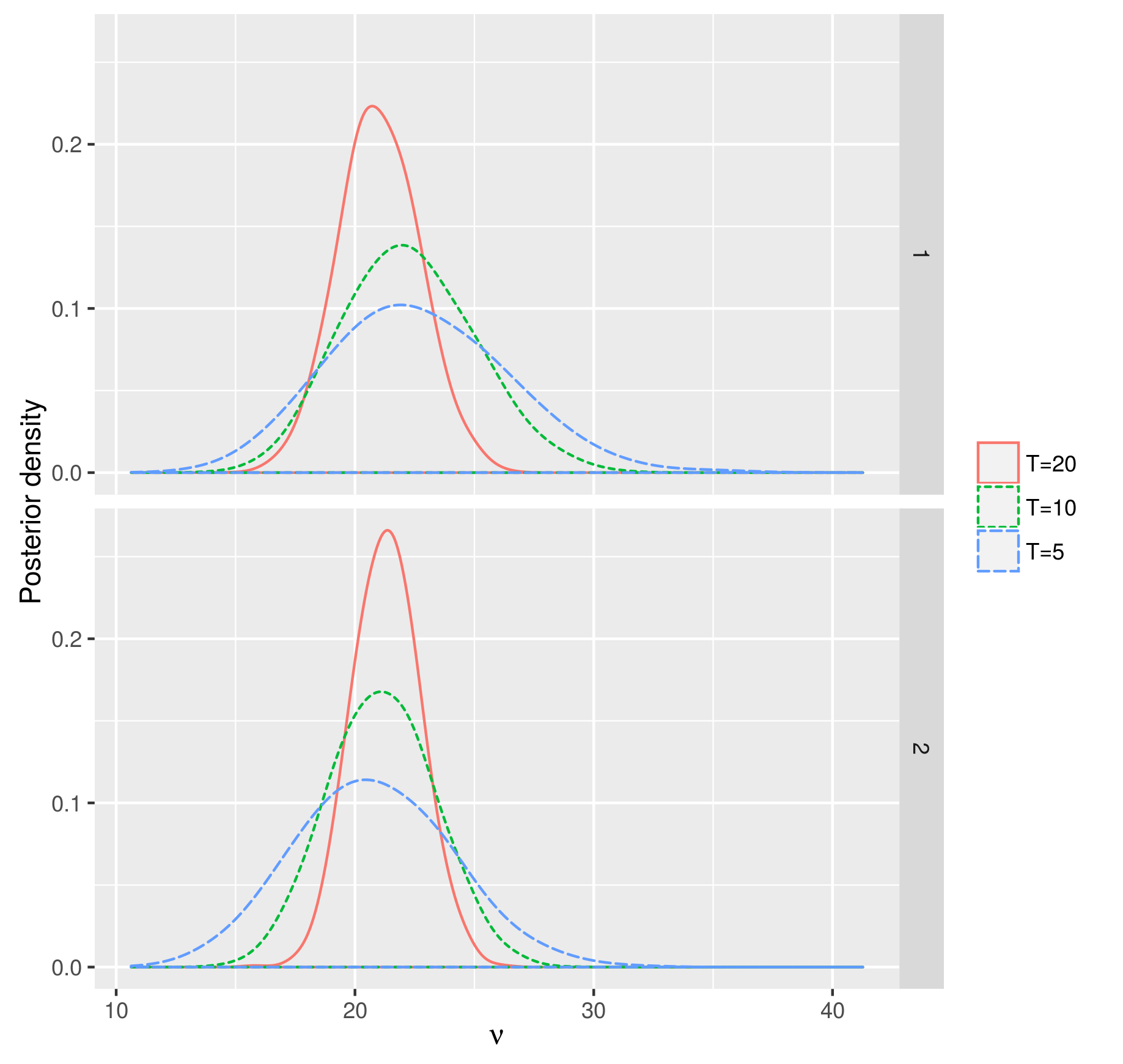}
\end{minipage}\hfill 
\begin{minipage}{0.5\linewidth}
\centering
\includegraphics[width=\textwidth]{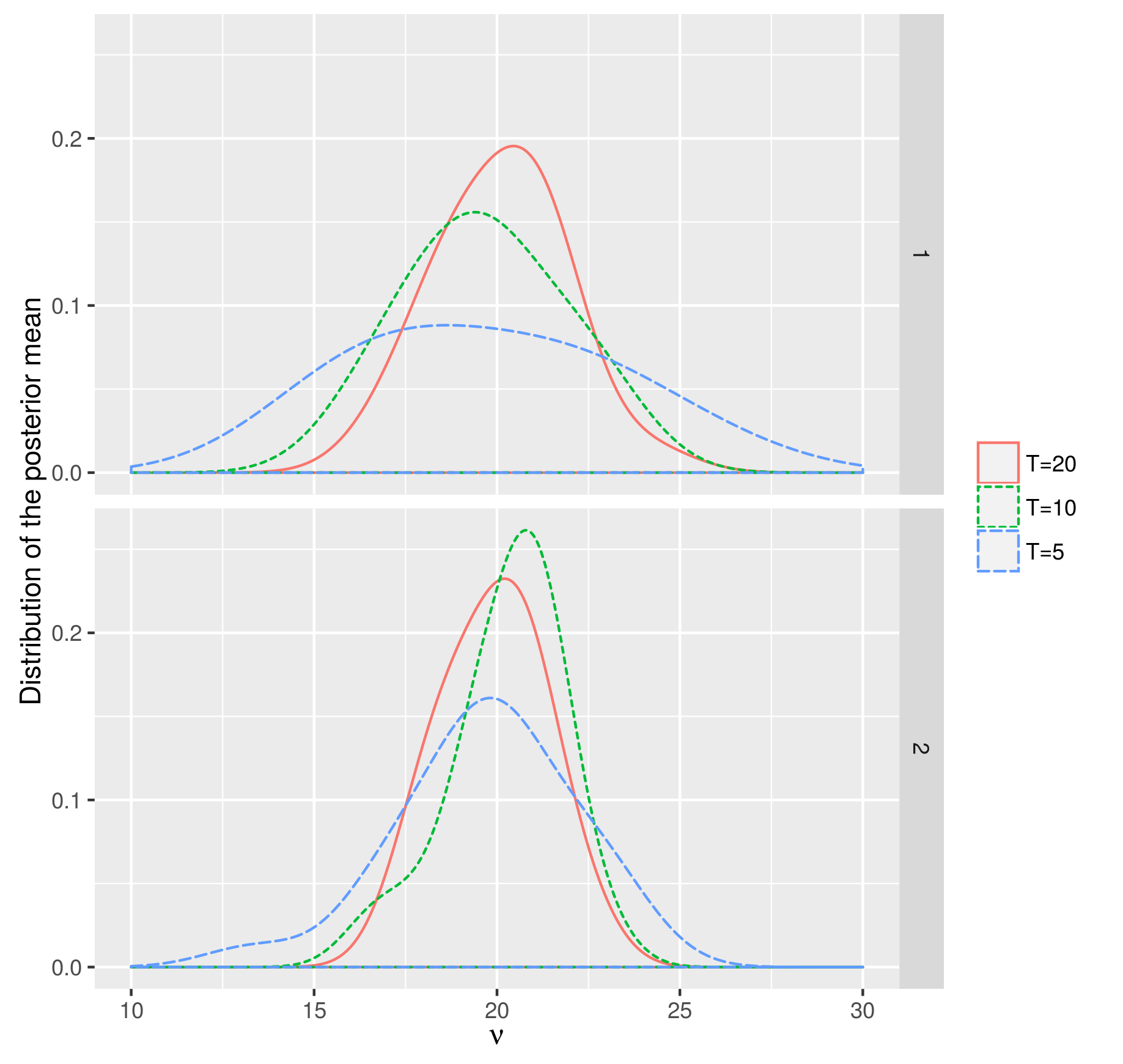}
\end{minipage}\hfill
  \caption{\textbf{Results for scenario 1}: \textit{On the left}, posterior distribution of $(\nu_1,\nu_2)$ with $T=5$, $T=10$ and $T=20$ for one dataset.  \textit{On the right}, distribution of the posterior mean of $(\nu_1,\nu_2)$  $\left(\widehat{\mathbb{E}}\left[\nu_{\ell}| (N^{sim}_t)_{t \in [0,T]}\right]\right)_{sim=1 \dots 25}$  over the 25 simulated datasets. Top : regular histogram; bottom : random histogram}
\label{fig:post nu M=2}
  %\label{fig:sfig2}
\end{figure}

%\begin{figure}
%\centering
%\includegraphics[width=\textwidth,height=0.25\textheight] {estim_h_M2_regular_sim5.png}
%\includegraphics[width=\textwidth,height=0.25\textheight]{estim_h_M2_continuous_sim5.png}
%\caption{\textbf{Scenario 1, K=2}. Estimation of the $(h_{k, \ell})_{k, \ell=1,2}$ using the regular prior (upper panel) and the random histogram prior (bottom panel). The gray region indicates the credible region for $h_{k, \ell}(t)$ (delimited by the $5\%$ and $95\%$ percentiles of the posterior distribution).  The true  $h_{k, \ell}$ is in plain line, the posterior expectation and posterior median  for  $h_{k, \ell}(t)$ are in dotted and dashed lines  respectively. }
%\label{fig:post h M=2}
%\end{figure}
%
%
%
%\begin{figure}
%\centering
%\begin{minipage}{0.5\linewidth}
%\centering
%\includegraphics[width=\textwidth]{post_nu_M2_sim=5.png}
%\end{minipage}\hfill 
%\begin{minipage}{0.5\linewidth}
%\centering
%\includegraphics[width=\textwidth]{post_estimnu_M2.png}
%\end{minipage}\hfill
% 
%\caption{\textbf{Results for scenario 1}. \textit{On the left}, posterior distribution of $(\nu_1,\nu_2)$ with $T=5$, $T=10$ and $T=20$ for one dataset.  \textit{On the right}, distribution of the posterior mean of $(\nu_1,\nu_2)$  $\left(\widehat{\mathbb{E}}\left[\nu_{\ell}| (N^{sim}_t)_{t \in [0,T]}\right]\right)_{sim=1 \dots 25}$  over the 25 simulated datasets. Top : regular histogram; bottom : random histogram}
%\label{fig:post nu M=2}
%\end{figure}
%

\noindent 

\begin{table}[ht]
\centering
\begin{tabular}{crrrrr}
  \hline
 $\ell $ over $k$   && $1$ over $1$ &$1$ over $2$ & $2$ over $1$ & $2$ over $2$ \\ 
\hline

    True value of $\delta^{(k, \ell)}$&  & $1$& $1$& $1$& $0$\\
    \hline
 & Prior& &&&\\
  \hline

\multirow{2}{*}{$T=5$}& Regular & 1.0000 & 0.8970 & 1.0000 & 0.0071 \\ 
 &Continous & 1.0000 & 0.9812 & 1.0000 & 0.0196 \\ 
 \hline 
 \multirow{2}{*}{$T=10$}&Regular& 1.0000 & 0.9954 & 1.0000 & 0.0047 \\ 
  &Continous & 1.0000 & 1.0000 & 1.0000 & 0.0102 \\ 
  \hline
  \multirow{2}{*}{$T=20$} &Regular& 1.0000 & 1.0000 & 1.0000 & 0.0099 \\ 
  &random & 1.0000 & 1.0000 & 1.0000 & 0.0102 \\ 
   \hline
\end{tabular}
\caption{\textbf{Scenario 1, K=2}.  Mean of the posterior estimations:  $ \frac{1}{25}\sum_{sim = 1}^ {25} \widehat{\P}(\delta^{(k, \ell)}=1| (N^{sim}_t)_{t \in [0,T]} )$, for the  three observation time intervals and the  two  prior distributions.}
\label{tab:post delta M=2}
\end{table}

\vspace{1em}

\noindent Finally,  we also have a look at the conditional intensities $ \lambda^k_t=\nu_{k}+\sum_{\ell=1}^K\int_{-\infty}^{t-}h_{\ell, k}(t-u)dN^{(\ell)}_u$.   On Figure  \ref{fig: post intensities  M=2}, we plot  $50$ realizations of the conditional intensity from the posterior distributions. More precisely, for one given dataset,  for $50$ parameters $\theta^{(i)} =\left((h^{(i)}_{k, \ell})_{k, \ell}, (\nu^{(i)}_k)_{k=1\dots K}\right)$ sampled from the posterior distribution (obtained at the end of the MCMC chain), we compute the corresponding $(\lambda^{k(i)}_t)$ and plot them. For the sake of clarity, only the conditional intensity of the first process ($k=1$) is  plotted and we restrict the graph to a short time interval $[3.2,3.6]$.  As noticed before, the conditional intensity is well reconstructed, with a clear improvement of the precision as the length of the observation time $T$ increases.  
\begin{figure}
\centering
\includegraphics[width=\textwidth, height = 0.4 \textheight]{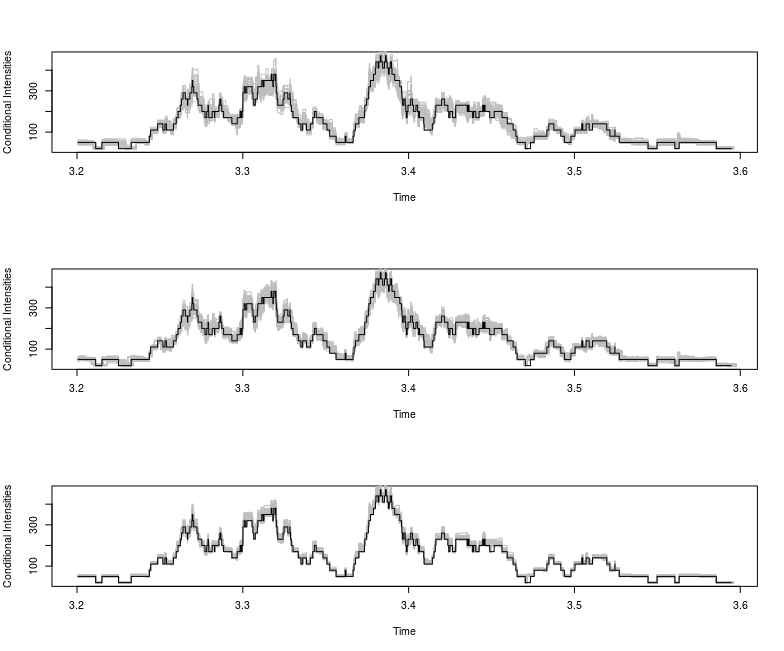}
\caption{\textbf{Scenario 1}. Conditional intensity $ \lambda^1_t$ : 50 realizations of $\lambda^1_t$  from the posterior distribution for one particular dataset and $3$ lengths of observation interval ($T=5$ on the first row $T=10$ on the second row,  and $T=20$ on the third row). True conditional intensity in black plain line. }
\label{fig: post intensities  M=2}
\end{figure}

\vspace{1em}

\subsubsection{Results for scenario 2: $K=8$}
%\textcolor{red}{why don't you present results with random grid histograms?}
In this scenario, we perform the Bayesian inference using  only the regular prior distribution on    $(\textbf{t}^{(k,\ell)})_{ (k, \ell) \in \{1,\dots, K\}^2}$ and two lengths of observation interval ($T=10$ and $T=20$). Here we set  $a_{\eta}= 3$  and $b_{\eta}=1$.  

\vspace{1em}

\noindent The posterior distribution of the $(\nu_{k})_{k=1\dots K}$ for a  randomly chosen dataset is plotted in Figure \ref{fig: post nu M=8}.  The prior distribution is in dotted line and is flat. The posterior distribution concentrates around the true value (here $20$) with a smaller variance when  $T$ increases. 
\begin{figure}
\centering
\includegraphics[height=0.5 \textheight]{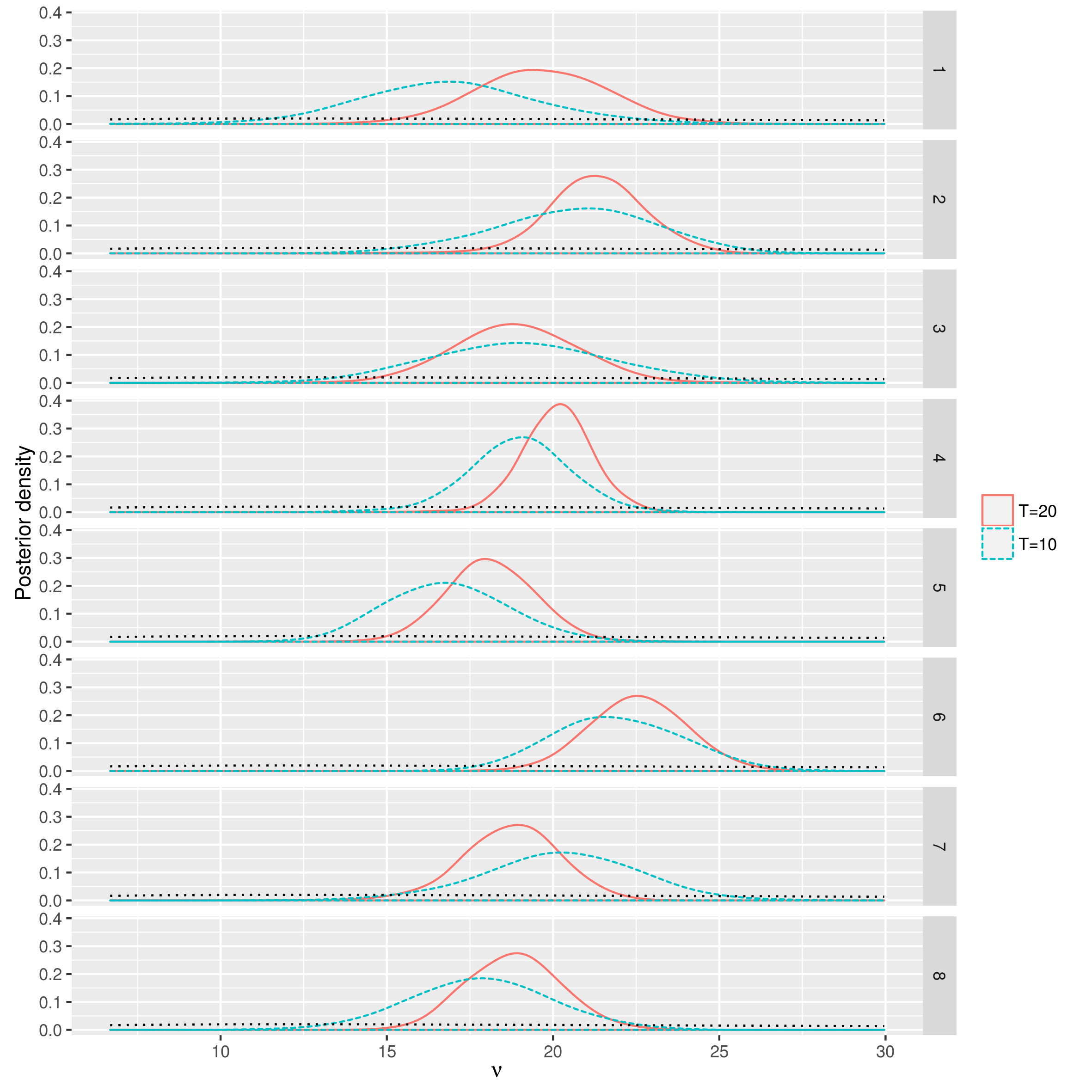}
\caption{\textbf{Scenario 2}: Results on $(\nu_{\ell})_{k=1\dots K}$ for a particular dataset:  Prior distribution (dotted line), Posterior distributions for $T=10$ ( dashed line) and $T=20$ (plain line)}
 \label{fig: post nu M=8}
\end{figure}

\begin{figure}
\centering \includegraphics[height=0.5 \textheight]{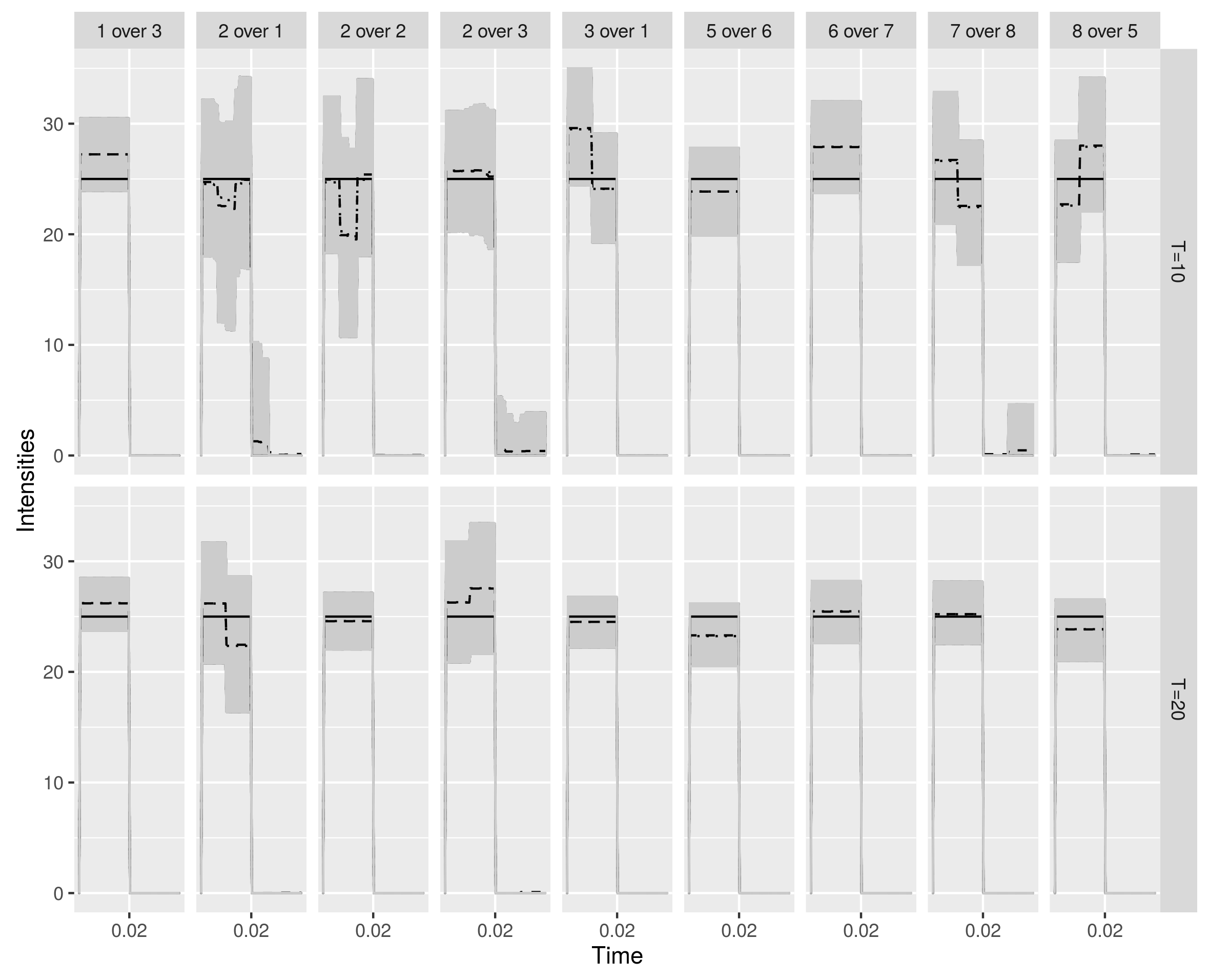}
\caption{\emph{Scenario 2}: Estimation of the non null interaction functions $(h_{k, \ell})_{k, \ell=1,\dots, 8}$ using the regular prior for $T=10$ (upper panel) and $T=20$ (bottom). The gray region indicates the credible region for $h_{k, \ell}(t)$ (delimited by the $5\%$ and $95\%$ percentiles of the posterior distribution).  The true  $h_{k, \ell}$ is in plain line, the posterior expectation and posterior median  for  $h_{k, \ell}(t)$ are in dotted and dashed lines  respectively (often undistinguishable).
}
\label{fig: M=8 post h}
\end{figure}

\vspace{1em}

\noindent In the context of  neurosciences, we are especially interested in recovering the interaction graph of the $K=8$ neurons.  In Figure \ref{fig: M=8: post graph}, we consider the same dataset as the one used in Figure \ref{fig: post nu M=8} and  plot the posterior estimation of the interaction graph, for respectively $T=10$ on the left and $T=20$ on the right.   The width and the gray level of the edges are proportional to the estimated posterior probability $\widehat{\P}(\delta^{(k, \ell)}=1 | (N_t)_{t \in [0,T]})$.  The global structure of the graph is recovered (to be compared to the true graph plotted in  Figure \ref{fig: true graph}).  We observe that the false positive edges appearing when $T=10$ disappear when $T=20$.   In Figure  \ref{fig: M=8: post graph all},  we consider the mean of the estimates of the graph over the 25 datasets. The resulting graph for $T=10$ is on the left and for $T=20$ on the right. 

\noindent Note that, in this example,   for any $(k, \ell)$ such that  the true $\delta^{(k, \ell)}=1$, the estimated posterior probability $\widehat{\P}(\delta^{(k, \ell)}=1 | (N^{sim}_t)_{t \in [0,T]})$ is equal to $1$, for any dataset  and any length of observation interval. In other words, the non-null interactions are perfectly recovered. In a simulation scenario with other  interaction  functions, the results could have been different.

\begin{figure}
\begin{minipage}{0.5\linewidth}
\centering
\includegraphics[width=0.9 \textwidth]{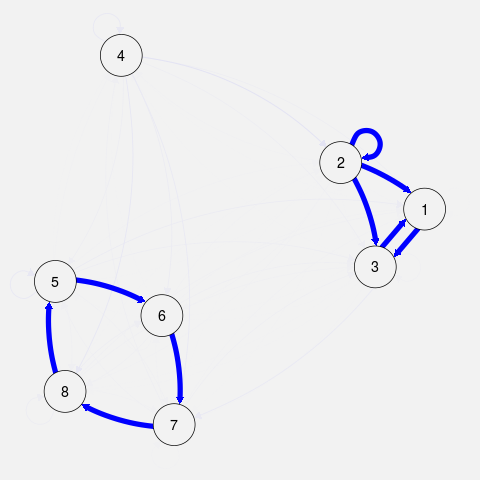}
\end{minipage}\hfill 
\begin{minipage}{0.5\linewidth}
\centering
\includegraphics[width=0.9 \textwidth]{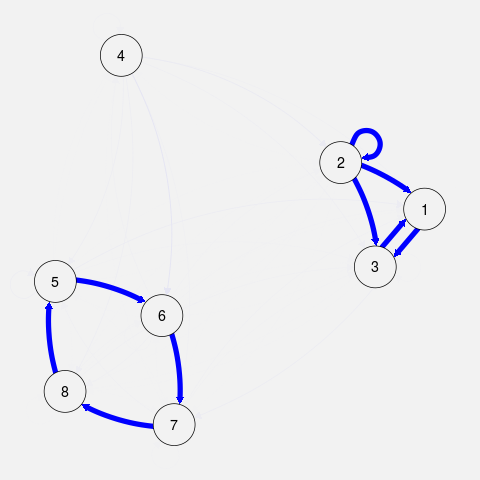}
\end{minipage}
\caption{\textbf{Results for scenario 2  for one given dataset}:  Posterior estimation of the interaction graph for    $T=10$ \textit{on the left} and $T=20$ \textit{on the right}, for one randomly chosen dataset. Level of grey and width of the edges proportional to the posterior estimated probability of  $\widehat{\P}(\delta^{(k, \ell)}=1| (N^{sim}_t)_{t \in [0,T]} )$.}
\label{fig: M=8: post graph}
\end{figure}

\begin{figure}
\begin{minipage}{0.5\linewidth}
\centering
\includegraphics[width=0.9 \textwidth]{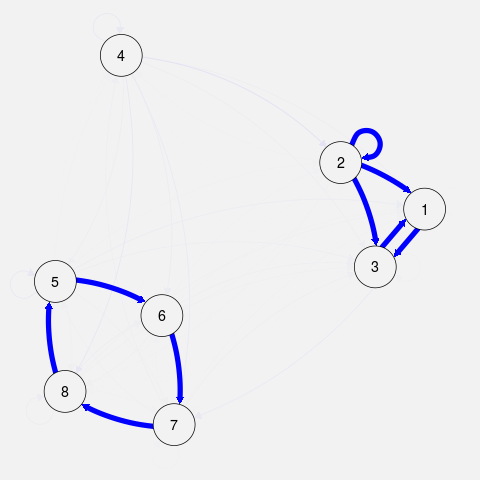}
\end{minipage}\hfill 
\begin{minipage}{0.5\linewidth}
\centering
\includegraphics[width=0.9 \textwidth]{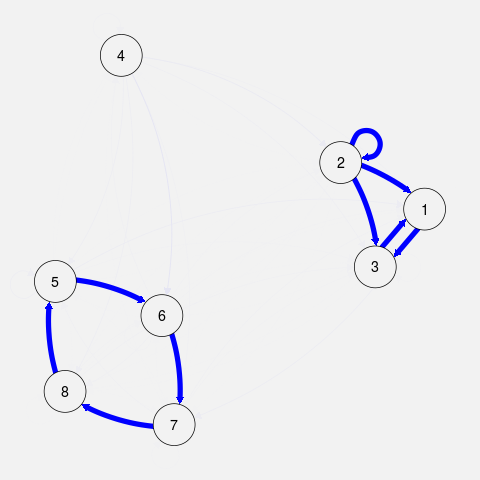}
\end{minipage}
\caption{\textbf{Results for scenario 2 over the $25$ simulated datasets}:  Posterior estimation of the interaction graph   for    $T=10$ \textit{on the left} and $T=20$ \textit{on the right}.    Level of grey and width of the edges are proportional to the posterior estimated probability of $\frac{1}{25}\sum_{sim = 1}^ {25} \widehat{\P}(\delta^{(k, \ell)}=1| (N^{sim}_t)_{t \in [0,T]} )$. }
\label{fig: M=8: post graph all}
 
%\caption{\textbf{Results for scenario 2} : interaction graphs}
%\label{fig:fig}
\end{figure}
%
%\begin{figure}
%\centering
%\begin{minipage}{0.5\linewidth}
%\centering
%\includegraphics[width=0.9 \textwidth]{graph_M8_sim2_T=10.png}
%\end{minipage}\hfill 
%\begin{minipage}{0.5\linewidth}
%\centering
%\includegraphics[width=0.9 \textwidth]{graph_M8_sim2_T=20.png}
%\end{minipage}\hfill
% 
%\caption{\textbf{Results for scenario 2  for one given dataset}. Posterior estimation of the interaction graph for    $T=10$ \textit{on the left} and $T=20$ \textit{on the right}, for one randomly chosen dataset. Level of grey and width of the edges proportional to the posterior estimated probability of  $\widehat{\P}(\delta^{(k, \ell)}=1| (N^{sim}_t)_{t \in [0,T]} )$.  }
%\label{fig: M=8: post graph}
%\end{figure}

%\begin{figure}
%\centering
%\begin{minipage}{0.5\linewidth}
%\centering
%\includegraphics[width=0.9 \textwidth]{graph_M8_sim_all_T=10.png}
%\end{minipage}\hfill 
%\begin{minipage}{0.5\linewidth}
%\centering
%\includegraphics[width=0.9 \textwidth]{graph_M8_sim_all_T=20.png}
%\end{minipage}\hfill
% 
%\caption{\textbf{Results for scenario 2 over the $25$ simulated datasets}. Posterior estimation of the interaction graph   for    $T=10$ \textit{on the left} and $T=20$ \textit{on the right}.    Level of grey and width of the edges are proportional to the posterior estimated probability of $\frac{1}{25}\sum_{sim = 1}^ {25} \widehat{\P}(\delta^{(k, \ell)}=1| (N^{sim}_t)_{t \in [0,T]} )$.}
%\label{fig: M=8: post graph all}
%\end{figure}

\noindent In Figure \ref{fig: M=8 post h}, we plot the posterior means (with credible regions)  of the non-null interaction functions for the   same simulated dataset as  in Figure \ref{fig: M=8: post graph}.   The time intervals where the interaction functions are null are again perfectly recovered. The posterior incertainty around the non-null functions  $h_{k, \ell}$  decreases when $T$ increases.

\subsubsection{Results for scenario 3 : $K=2$ with smooth functions}
In this context, we perform the inference using the random histogram prior distribution (\ref{eq: s moving}). In this case, we set $a_{\eta}= 10$  and $b_{\eta}=1$. thus encouraging a greater number of step in the interactions functions.    The behavior of the posterior distribution of $\nu_{k}$ is the same as in the other examples. In Figure \ref{fig: M=2 smooth  estim nu}, we plot the distribution of $\left(\mathbb E\left[\nu_{k} | (N^{sim}_t)_{t \in [0,T]}\right]\right)_{sim=1\dots 25}$ for $T=5,10,20$ seconds and clearly observe a decrease of the biais and the variance as the length of the observation period increases. Some estimation of the interaction functions   are given in  Figure \ref{fig: M=2 smooth post h}.  Due to the choice of the prior distribution of these quantities, we get a sparse posterior inference.

\begin{figure}
\centering
\includegraphics[width=\textwidth,height=0.5\textheight]{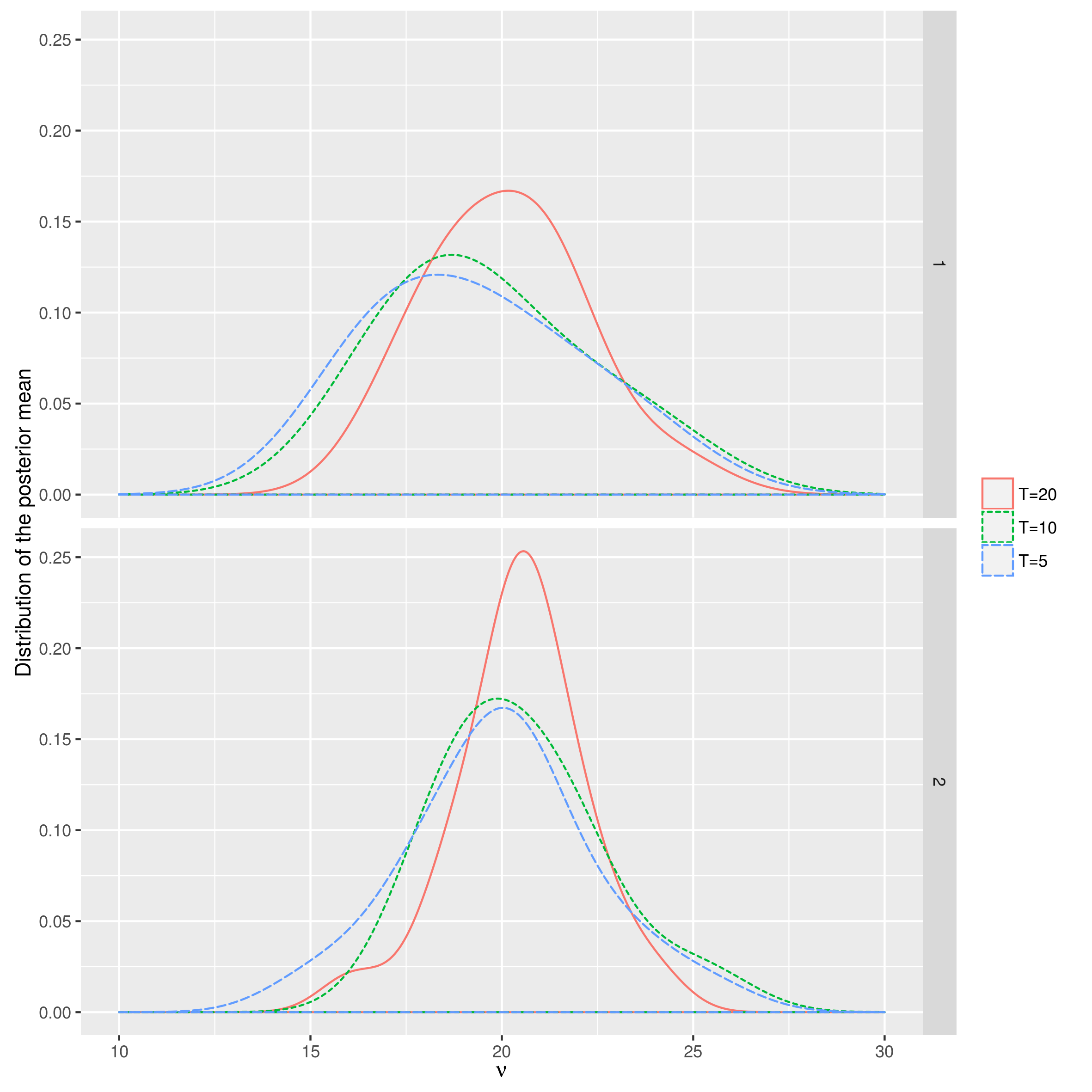}
\caption{\textbf{Results for scenario 3 : smooth interaction functions}: Distributions of $\left(\mathbb E\left[\nu_{k} | (N^{sim}_t)_{t \in [0,T]}\right]\right)_{sim=1\dots 25}$ for $T=5,10,20$ seconds (long dashed, short dashed and plain line respectively).}
\label{fig: M=2 smooth  estim nu}
\end{figure}%
\begin{figure}
\centering
\includegraphics[width=\textwidth,height=0.5\textheight]{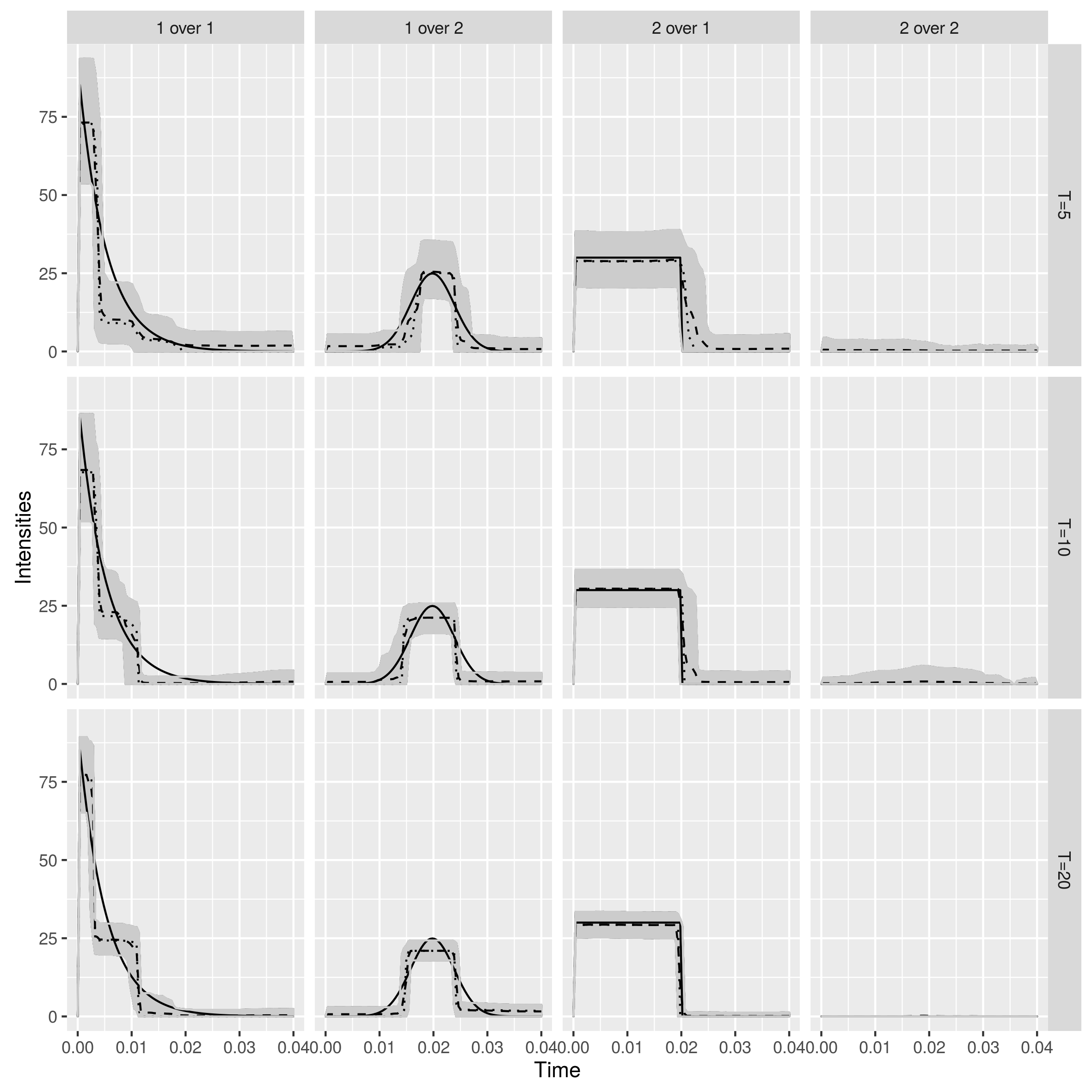}
\caption{\textbf{Results for scenario 3 : smooth interaction functions}: Estimation of the  interaction functions $(h_{k, \ell})_{k, \ell=1,2}$ using the regular prior for $T=10$ (upper panel) and $T=10$ (bottom). The gray region indicates the credible region for $h_{k, \ell}(t)$ (delimited by the $5\%$ and $95\%$ percentiles of the posterior distribution).  The true  $h_{k, \ell}$ is in plain line, the posterior expectation and posterior median  for  $h_{k, \ell}(t)$ are in dotted and dashed lines  respectively (often undistinguishable).}
\label{fig: M=2 smooth post h}
\end{figure}

\label{sec:numeric}
%%%%%%%%%%%%%%%%%%%%%%%%%%
%%%%%%%%%%%%%%%%%%%%%%%%%%
\section{Proofs of Theorems}
%In the sequel, we denote $\square$ a constant (that may depend on $K$) whose value may change from line to line. When it also depends on $f_0$, we denote $\square(f_0)$.
%%%%%%%%%%%%%%%%%%%%%%%%%%%
\subsection{Proof of Theorem~\ref{th:d1}}\label{sec:prthd1}
To prove Theorem~\ref{th:d1}, we apply the general methodology of \citet{ghosal:vdv:07}, with modifications due to the fact that $\exp(L_T(f))$ is the likelihood of the distribution of $(N^k)_{k=1,\ldots,K}$ on $[0,T]$ conditional on $\mathcal G_{0^-}$ and that the metric $d_{1,T} $ depends on the observations. We set $M_T=M\sqrt{\log\log T}$, for $M$ a positive constant.
%We have:
%$$L_T(f)=\sum_{k=1}^K \left(\int_0^T\log\left(\lambda_t^k(f)\right)dN^k_t
%-\int_0^T\lambda_t^k(f)dt\right).$$
Let  $$A_\epsilon= \{ f\in{\mathcal F} ; \ d_{1,T}(f_0,f) \leq K\epsilon\}$$ 
and  for $j\geq 1$, we set
\begin{equation}\label{Sj} 
S_j  =\left\{ f \in  \mathcal F_T; \ d_{1,T} (f, f_0) \in  (Kj \epsilon_T, K(j+1) \epsilon_T] \right\},
\end{equation} 
where $\mathcal F_T= \left\{f=((\nu_k)_{k}, (h_{k,\ell})_{k,\ell})\in \mathcal F ; \ (h_{k,\ell})_{k,\ell})\in\mathcal H_T\right\}$.
So that, for any test function $\phi$, 
\begin{equation*}
\begin{split}
 \Pi\left( A_{M_T\epsilon_T}^c |N \right) &=  \frac{ \int_{A_{M_T\epsilon_T}^c} e^{L_T(f) - L_T(f_0)}d\Pi(f) }{ \int_{\mathcal F} e^{L_T(f) - L_T(f_0)}d\Pi(f) }=: \frac{\Num_T }{ D_T}  \\
 &\leq \1_{\Omega_T^c} +\1_{\left\{D_T<  \frac{\Pi( B(\epsilon_T,T))}{\exp(2(\kappa_T+1)T\epsilon_T^2)}\right\}}+ \phi\1_{\Omega_T}+\frac{e^{2(\kappa_T+1)T\epsilon_T^2}}{ \Pi( B(\epsilon_T,T))} \int_{{\mathcal F}_T^c} e^{L_T(f)-L_T(f_0)}d\Pi(f) \\&+ \1_{\Omega_T} \frac{e^{2(\kappa_T+1)T\epsilon_T^2}}{ \Pi( B(\epsilon_T,T))} \sum_{j=M_T}^\infty \int_{\mathcal F_T}\1_{f \in S_j} e^{L_T(f)-L_T(f_0)}(1-\phi)d\Pi(f)
\end{split}
\end{equation*}
and
\begin{equation*}
\begin{split}
\E_0 \left[\Pi\left( A_{M_T\epsilon_T}^c | N \right)\right]&\leq \P_0 (\Omega_T^c) + \P_0\left( D_T< e^{-2(\kappa_T+1)T\epsilon_T^2} \Pi(B(\epsilon_T,B)) \right) + \E_0[\phi\1_{\Omega_T}]\\
& \quad +\frac{e^{2(\kappa_T+1)T\epsilon_T^2} }{ \Pi( B(\epsilon_T,B))} \left( \Pi(\mathcal F_T^c) + \sum_{j=M_T}^{\infty} \int_{\mathcal F_T} \E_0\left[ \E_f\left[ \1_{\Omega_T } \1_{f \in S_{j}}(1-\phi) | \mathcal G_{0^-} \right] \right]d\Pi(f) \right),
\end{split}
\end{equation*}
since
$$\E_0\left[\int_{{\mathcal F}_T^c} e^{L_T(f)-L_T(f_0)}d\Pi(f) \right]=\E_0\left[\E_0\left[\int_{{\mathcal F}_T^c} e^{L_T(f)-L_T(f_0)}d\Pi(f)| \mathcal G_{0^-}\right] \right]= \E_0\left[\E_f\left[\int_{{\mathcal F}_T^c} d\Pi(f)| \mathcal G_{0^-}\right] \right]=\Pi({\mathcal F}_T^c).$$
Since $e^{(\kappa_T +1)T\epsilon_T^2}e^{L_T(f) - L_T(f_0)}\geq\1_{\left\{L_T(f) - L_T(f_0) \geq -(\kappa_T +1)T\epsilon_T^2\right\}},$
\begin{equation*}
\begin{split}
\P_0\left( D_T \leq e^{-2(\kappa_T +1)T\epsilon_T^2 }\Pi( B(\epsilon_T , B)) \right) &\leq 
\P_0\left(\int_{B(\epsilon_T , B)} e^{L_T(f) - L_T(f_0) } \frac{ d\Pi(f)}{\Pi( B(\epsilon_T , B))} \leq e^{-2(\kappa_T +1)T\epsilon_T^2 }\right)\\
&\leq 
\P_0\left(\int_{B(\epsilon_T , B)} \1_{\left\{L_T(f) - L_T(f_0)\geq -(\kappa_T +1)T\epsilon_T^2\right\}}\frac{ d\Pi(f)}{\Pi( B(\epsilon_T , B))} \leq e^{-(\kappa_T +1)T\epsilon_T^2 }\right)\\
&\leq \frac{\E_0\left[\int_{B(\epsilon_T , B)} \1_{\left\{L_T(f) - L_T(f_0)< -(\kappa_T +1)T\epsilon_T^2\right\}}\frac{ d\Pi(f)}{\Pi( B(\epsilon_T , B))}\right]}{\left(1-e^{-(\kappa_T +1)T\epsilon_T^2 }\right)}\\
&\leq \frac{ \int_{B(\epsilon_T,B) } \P_0\left( L_T(f_0) - L_T(f) > (\kappa_T+1)T\epsilon_T^2 \right) d\Pi(f) }{ \Pi( B(\epsilon_T,B))\left(1-e^{-(\kappa_T +1)T\epsilon_T^2 }\right)  }\\&\lesssim \frac{\log\log(T)\log^3 (T) }{ T\epsilon_T^2 },
\end{split}
\end{equation*}
by using Lemma \ref{lemma:Kullback} of Section~\ref{sec:denominator}.
Remember we have set $\rho_{k,\ell}^0:=\|h_{k,\ell}^0\|_1$ and $\rho_{k,\ell}:=\|h_{k,\ell}\|_1$. Since $h_{k,\ell}$ and $h^0_{k,\ell}$ are non-negative functions,
$\int_{-s}^Ah_{k,\ell}^0(u)du\leq  \rho_{k,\ell}^0,\quad \int_{0}^{T-s} h_{k,\ell}^0(u)du \leq  \rho_{k,\ell}^0,$ and note that 
\begin{eqnarray*}
 T d_{1,T} ( f, f_0)  &=& \sum_{\ell = 1}^K \int_0^T \left| \nu_\ell - \nu_\ell^0 + \sum_{k=1}^K \int_{t-A}^{t^-} (h_{k,\ell} - h_{k,\ell}^0 )(t-s) dN_{s}^{k} \right| dt\\
 &\geq& \sum_{\ell = 1}^K  \left| \int_0^T\left(\nu_\ell - \nu_\ell^0 + \sum_{k=1}^K \int_{t-A}^{t^-} (h_{k,\ell} - h_{k,\ell}^0 )(t-s) dN_{s}^{k}\right) dt\right| \\
 &\geq&\sum_{\ell = 1}^K \left|T(\nu_\ell - \nu_\ell^0)+ \int_0^T\left(\sum_{k=1}^K \int_{t-A}^{t^-} (h_{k,\ell} - h_{k,\ell}^0 )(t-s) dN_{s}^{k}\right) dt\right|,
 \end{eqnarray*}
%Then  for all $\ell \leq d$,  by setting
%$$\rho_{k,\ell}:=\int_{0^+}^Ah_{k,\ell}(u)du=\|h_{k,\ell}\|_1,\quad \rho_{k,\ell}^0:=\int_{0^+}^Ah_{k,\ell}^0(u)du=\|h_{k,\ell}^0\|_1,$$
then for any $\ell=1,\ldots,K$, 
\begin{eqnarray*}
d_{1,T} ( f, f_0) &\geq& \left|\nu_\ell - \nu_\ell^0 + \frac{1}{T}\sum_{k=1}^K\int_0^T  \int_{t-A}^{t^-} (h_{k,\ell} - h_{k,\ell}^0 )(t-s) dN_{s}^{k} dt \right| \\
 &=& \left|\nu_\ell - \nu_\ell^0 + \sum_{k=1}^K  (\rho_{k,\ell} - \rho_{k,\ell}^0)\frac{N^k[0, T-A]}{T}\right.\\ &&  \left. +     \frac{1}{T}\int_{-A}^{0} \int_{-s}^A(h_{k,\ell} - h_{k,\ell}^0 )(u)du dN_{s}^{k} +   \frac{1}{T}\int_{T-A}^{T^-} \int_{0}^{T-s} (h_{k,\ell} - h_{k,\ell}^0 )(u)dudN_{s}^{k}\right|\\
 &=& \left|\nu_\ell+\sum_{k=1}^K  \rho_{k,\ell}\frac{N^k[0, T-A]}{T}+     \frac{1}{T}\int_{-A}^{0} \int_{-s}^Ah_{k,\ell}(u)du dN_{s}^{k} +   \frac{1}{T}\int_{T-A}^{T^-} \int_{0}^{T-s} h_{k,\ell}(u)du dN_{s}^{k}
 \right.\\
 &&-\left.\left(\nu_\ell^0+\sum_{k=1}^K  \rho_{k,\ell}^0\frac{N^k[0, T-A]}{T}+     \frac{1}{T}\int_{-A}^{0} \int_{-s}^Ah_{k,\ell}^0(u)du dN_{s}^{k} +   \frac{1}{T}\int_{T-A}^{T^-} \int_{0}^{T-s} h_{k,\ell}^0(u)du dN_{s}^{k}\right)
 \right|.
\end{eqnarray*}
This  implies for $f\in   S_j $ that
\begin{equation} \label{ineq:nu-rho}
\begin{split}
& \nu_\ell +  \sum_{k=1}^K  \rho_{k,\ell} \frac{N^k[0, T-A]}{T}   \leq  \nu_\ell^0 + \sum_{k=1}^K   \rho_{k,\ell}^0 \frac{ N^k[-A, T]}{ T } +K(j+1)\epsilon_T  \\
& \nu_\ell +  \sum_{k=1}^K  \rho_{k,\ell} \frac{N^k[-A, T]}{T}   \geq  \nu_\ell^0 + \sum_{k=1}^K   \rho_{k,\ell}^0 \frac{ N^k[0, T-A]}{ T }  - K(j+1)\epsilon_T .
\end{split}
\end{equation}
On $\Omega_T$,
$$\sum_{k=1}^K   \rho_{k,\ell}^0 \frac{ N^k[-A, T]}{ T } \leq \sum_{k=1}^K   \rho_{k,\ell}^0 (\mu_k^0+\delta_T),$$
 so that, for $T$ large enough,  for all $j \geq 1$ $S_j \subset \mathcal F_j$ with
$$\mathcal F_j := \{ f \in \mathcal F_T; \ \nu_\ell \leq \mu_\ell^0 +1+ Kj \epsilon_T, \forall \ell \leq K\},$$ 
since
\begin{equation}\label{munu}
\mu_\ell^0 = \nu_\ell^0+  \sum_{k=1}^K   \rho_{k,\ell}^0 \mu_k^0.
\end{equation}
Let $(f_i)_{i=1,\ldots, {\mathcal N}_j}$ be the centering points of a minimal $\L_1$-covering of  $\mathcal F_j$ by balls of radius $\zeta j \epsilon_T$ with $\zeta  = 1/(6N_0)$ (with $N_0$ defined in Section~\ref{sec:main}) and define $\phi_{(j)} = \max_{i=1,\ldots,{\mathcal N}_j} \phi_{f_i,j} $ where $\phi_{f_i,j} $ is the individual test defined in Lemma~\ref{lem:test:d1} associated to $f_i$ and $j$  (see Section~\ref{sec:tests}). Note also that there exists a constant $C_0$ such that 
$${\mathcal N}_j  \leq \left( C_0(1+ j\epsilon_T)/  j\epsilon_T\right)^{K} {\mathcal N}(\zeta j \epsilon_T/2, \mathcal H_T, \| . \|_1)$$
where ${\mathcal N}(\zeta j \epsilon_T/2, \mathcal H_T, \| . \|_1)$ is the covering number of $\mathcal H_T$ by $\L_1$-balls with radius $\zeta j \epsilon_T/2$. There exists $C_K$ such that  if $j\epsilon_T \leq 1$ then    ${\mathcal N}_j \leq C_K e^{- K \log (j \epsilon_T)}{\mathcal N}(\zeta j \epsilon_T/2, \mathcal H_T, \| . \|_1)$ and if $j\epsilon_T>1$ then ${\mathcal N}_j \leq C_K N(\zeta j \epsilon_T/2, \mathcal H_T, \| . \|_1)$. Moreover $j\mapsto {\mathcal N}(\zeta j \epsilon_T/2, \mathcal H_T, \| . \|_1)$ is monotone non-increasing, choosing $j \geq 2\zeta_0/\zeta$, we obtain that 
$${\mathcal N}_j\leq C_K (\zeta/\zeta_0)^Ke^{ K \log T}e^{x_0 T\epsilon_T^2},$$ 
from hypothesis (iii) in Theorem \ref{th:d1}. Combining this with Lemma~\ref{lem:test:d1}, we have for all $j\geq 2\zeta_0/\zeta$,
\begin{equation*}
\begin{split}
\E_0[\1_{\Omega_T}\phi_{(j)}] &\lesssim {\mathcal N}_j e^{-T x_2 (j\epsilon_T \wedge j^2 \epsilon_T^2)} \lesssim e^{ K \log T}e^{x_0 T\epsilon_T^2} e^{- x_2T (j\epsilon_T \wedge j^2 \epsilon_T^2)}\\
\sup_{f \in \mathcal F_j} \mathbb E_{0} \left[\E_f[\1_{\Omega_T} \1_{f\in S_j} (1-\phi_{(j)}) | \mathcal G_{0^-}]\right] &\lesssim e^{-x_2T (j\epsilon_T \wedge j^2 \epsilon_T^2)},
\end{split}
\end{equation*}
for $x_2$ a constant. Set $\phi = \max_{j\geq M_T}\phi_{(j)}$ with $M_T > 2\zeta_0/\zeta$,  then 
$$\E_0[\1_{\Omega_T}\phi] \lesssim e^{ K \log T}e^{x_0 T\epsilon_T^2} \left[\sum_{j=M_T}^{\lfloor \epsilon_T^{-1} \rfloor}  e^{-x_2T\epsilon_T^2 j^2} + \sum_{j\geq \epsilon_T^{-1} }e^{- T x_2\epsilon_Tj }  \right]\lesssim e^{- x_2 T \epsilon_T^2 M_T^2/2 }$$
and 
$$\sum_{j=M_T}^\infty \int_{\mathcal F_T} \E_0\left[ \E_f\left[ \1_{\Omega_T }\1_{ f\in S_j} (1-\phi) |\mathcal G_{0^-} \right]\right]d\Pi(f) \lesssim e^{- x_2 T \epsilon_T^2 M_T^2 /2}.$$
Therefore,
$$\frac{e^{2(\kappa_T+1)T\epsilon_T^2} }{ \Pi( B(\epsilon_T,B))}\sum_{j=M_T}^\infty \int_{\mathcal F_T} \E_0\left[ \E_f\left[ \1_{\Omega_T }\1_{ f\in S_j} (1-\phi) |\mathcal G_{0^-} \right]\right]d\Pi(f) =o(1)$$ if $M$ is a constant large enough, which terminates the proof of Theorem \ref{th:d1}.
%%%%%%%%%%%%%%%%%%%%%%%%
\subsection{Proof of Theorem \ref{th:slices}}
The proof of Theorem \ref{th:slices} follows the same lines  as for Theorem \ref{th:d1}, except that the decomposition of $\mathcal F_T$ is based on the sets $ \mathcal F_j$ and $\mathcal H_{T, i} $, $i\geq 1$ and $j\geq M_T$ for some $M_T>0$. For each $i\geq 1$, $j \geq M_T$, consider $S_{i,j}'$ a maximal set of  $\zeta j \epsilon_T$-separated points in $\mathcal F_j \cap \mathcal H_{T,i}$ (with a slight abuse of notations) and $\phi_{i,j} = \max_{f_1\in S_{i,j}'} \phi_{f_1}$ with $\phi_{f_1}$ defined in Lemma \ref{lem:test:d1}. Then, $$|S_{i,j}'| \leq C_K (\zeta/\zeta_0)^Ke^{K \log (T)}{\mathcal N}(\zeta j\epsilon_T/2, \mathcal H_{T,i}, \|. \|_1).$$ 
Setting $\Num_{T,ij}:=\int_{\mathcal F_T\cap \mathcal H_{T,i}}\1_{f\in S_j} e^{L_T(f) - L_T(f_0)}d\Pi(f)$, using similar computations as for the proof of Theorem~\ref{th:d1}, we have:
\begin{equation*}
\begin{split}
\E_0 \left[\Pi\left( A_{M_T\epsilon_T}^c | N \right)\right]&\leq \P_0 (\Omega_T^c) + \P_0\left( D_T< e^{-2(\kappa_T+1)T\epsilon_T^2} \Pi(B(\epsilon_T,B)) \right) +\frac{e^{2(\kappa_T+1)T\epsilon_T^2} }{ \Pi( B(\epsilon_T,B))} \Pi(\mathcal F_T^c) \\
&\hspace{-1cm} +\E_0\left[\1_{\Omega_T}\sum_{i=1}^{+\infty}\sum_{j=M_T}^{+\infty}\phi_{ij}\frac{\Num_{T,ij}}{D_T}\right]+\frac{e^{2(\kappa_T+1)T\epsilon_T^2} }{ \Pi( B(\epsilon_T,B))}\E_0\left[\1_{\Omega_T}\sum_{i=1}^{+\infty}\sum_{j=M_T}^{+\infty}(1-\phi_{ij})\Num_{T,ij}\right].
\end{split}
\end{equation*}
Assumptions of the theorem allow us to deal with the first three terms. So, we just have to bound the last two ones. Using the same arguments and the same notations as for Theorem~\ref{th:d1},
\begin{eqnarray*}
\E_0\left[\1_{\Omega_T}\sum_{i=1}^{+\infty}\sum_{j=M_T}^{+\infty}(1-\phi_{ij})\Num_{T,ij}\right]&=&\sum_{i=1}^{+\infty}\int_{\mathcal F_T\cap \mathcal H_{T,i}}\sum_{j=M_T}^{+\infty}\E_0\left[\1_{\Omega_T}\1_{f\in S_j} (1-\phi_{ij})e^{L_T(f) - L_T(f_0)}\right]d\Pi(f)\\
&=&\sum_{i=1}^{+\infty}\int_{\mathcal F_T\cap \mathcal H_{T,i}}\sum_{j=M_T}^{+\infty} \mathbb E_{0} \left[\E_f[\1_{\Omega_T} \1_{f\in S_j} (1-\phi_{ij}) | \mathcal G_{0^-}]\right] d\Pi(f)\\
&\lesssim&\sum_{i=1}^{+\infty}\int_{\mathcal F_T\cap \mathcal H_{T,i}}d\Pi(f)\sum_{j=M_T}^{+\infty}  e^{-x_2T (j\epsilon_T \wedge j^2 \epsilon_T^2)} \lesssim e^{- x_2 T \epsilon_T^2 M_T^2 /2}.
\end{eqnarray*}
Now, for $\gamma$ a fixed positive constant smaller than $x_2$, setting $\pi_{T,i}=\Pi(\mathcal H_{T,i})$, we have
\begin{small}
\begin{equation*}
\begin{split}
& \E_0\left[\1_{\Omega_T}\sum_{i=1}^{+\infty}  \sum_{j=M_T}^{+\infty}\phi_{ij}\frac{\Num_{T,ij}}{D_T}\right] \leq \P_0\left( D_T < e^{-2(\kappa_T+1)T\epsilon_T^2} \Pi( B(\epsilon_T,B)) \right) +  \P_0\left(\exists (i,j) ; \sqrt{\pi_{T,i}}\phi_{i,j}>e^{-\gamma T (j\epsilon_T \wedge j^2 \epsilon_T^2)}\cap\Omega_T\right)\\&
\qquad +\sum_{i=1}^{+\infty}\sum_{j=M_T}^{+\infty}e^{-\gamma T (j\epsilon_T \wedge j^2 \epsilon_T^2)}\sqrt{\pi_{T,i}} \frac{e^{2(\kappa_T+1)T\epsilon_T^2} }{ \Pi( B(\epsilon_T,B))} \E_0\left[\1_{\Omega_T}\int_{\mathcal F_T}\1_{f\in S_j} e^{L_T(f) - L_T(f_0)}d\Pi(f|\mathcal H_{T,i}) \right].
\end{split}
\end{equation*}
\end{small}
Now,
\begin{align*}
 \P_0\left(\exists (i,j) ; \sqrt{\pi_{T,i}}\phi_{i,j}>e^{-\gamma T (j\epsilon_T \wedge j^2 \epsilon_T^2)}\cap\Omega_T\right)&\leq \sum_{i=1}^{+\infty}\sqrt{\pi_{T,i}}\sum_{j=M_T}^{+\infty}e^{\gamma T (j\epsilon_T \wedge j^2 \epsilon_T^2)}\E_0[\1_{\Omega_T}\phi_{i,j}]\\
 &\lesssim \sum_{i=1}^{+\infty}\sqrt{\pi_{T,i}}\sum_{j=M_T}^{+\infty}e^{(\gamma-x_2) T (j\epsilon_T \wedge j^2 \epsilon_T^2)+K \log (T)}{\mathcal N}(\zeta j\epsilon_T/2, \mathcal H_{T,i}, \|. \|_1)\\
 &\lesssim e^{(\gamma- x_2) T \epsilon_T^2 M_T^2 /2}\sum_{i=1}^{+\infty}\sqrt{\pi_{T,i}}{\mathcal N}(\zeta_0 \e_T, \mathcal H_{T,i}, \|. \|_1)=o(1).
\end{align*}
 But, we have
\begin{eqnarray*}
\E_0\left[\1_{\Omega_T}\int_{\mathcal F_T}\1_{f\in S_j} e^{L_T(f) - L_T(f_0)}d\Pi(f|\mathcal H_{T,i})  \right] 
&\leq&1
\end{eqnarray*}
and
\begin{eqnarray*}
\E_0\left[\1_{\Omega_T}\sum_{i=1}^{+\infty}\sum_{j=M_T}^{+\infty}\phi_{ij}\frac{\Num_{T,ij}}{D_T}\right]
&\lesssim&\sum_{i=1}^{+\infty}\sqrt{\pi_{T,i}}e^{- \gamma T \epsilon_T^2 M_T^2 }\frac{e^{2(\kappa_T+1)T\epsilon_T^2} }{ \Pi( B(\epsilon_T,B))}+o(1)=o(1),\end{eqnarray*}
for $M$ a contant large enough. This terminates the proof of Theorem \ref{th:slices}.
%%%%%%%%
\subsection{Construction of tests}\label{sec:tests}
As usual, the control of the posterior distributions is based on specific tests. We build them in the following lemma.
\begin{lemma}\label{lem:test:d1}
Let $j\geq 1$, $ f_1\in \mathcal F_j$ and define the test 
$$ \phi_{f_1,j} = \max_{\ell=1,\ldots,K}\left(\1_{\{ N^\ell(A_{1,\ell}) - \Lambda^\ell(A_{1,\ell}; f_0) \geq j T\epsilon_T/8 \}} \vee  \1_{\{ N^\ell(A_{1,\ell}^c) - \Lambda^\ell(A_{1,\ell}^c; f_0) \geq  j T\epsilon_T/8 \}} \right),$$
%Let $N_0 \geq 1 $ and define  $$\Omega_T  = \{ \sum_\ell N^\ell[-A,T] \leq N_0 T \},$$ 
with  for all $\ell \leq K$, $A_{1,\ell} = \{ t \in [0,T]; \, \lambda_t^\ell(f_1) \geq \lambda_t^\ell(f_0) \}$, $ \Lambda^\ell(A_{1,\ell}; f_0)=\int_0^T\1_{A_{1,\ell}}(t)\lambda_t^\ell(f_0)dt$ and $ \Lambda^\ell(A_{1,\ell}^c; f_0)=\int_0^T\1_{A_{1,\ell}^c}(t)\lambda_t^\ell(f_0)dt$. Then
\begin{equation*}
\E_0\left[\1_{\Omega_T} \phi_{f_1,j} \right] + \sup_{\|f - f_1\|_1 \leq  j\epsilon_T  / (6N_0)}\mathbb E_0 \left[\E_f\left[\1_{\Omega_T} \1_{f\in S_j}(1- \phi_{f_1,j})  | \mathcal G_{0^-}\right]\right] \leq (2K+1) \max_\ell e^{-  x_{1,\ell} T j\epsilon_T(\sqrt{\mu_\ell^0 }  \wedge j \epsilon_T)},
\end{equation*} 
with  $N_0$ is defined in Section~\ref{sec:main} and $$x_{1,\ell}  =\min \left(36, 1/(4096\mu^0_\ell),1/\left(1024K\sqrt{\mu_\ell^0}\right)\right).$$ 
\end{lemma}
\begin{proof}[Proof of Lemma~\ref{lem:test:d1}]
Let $j\geq 1$ and $f_1=((\nu_k^1)_{k=1,\ldots,K}, (h_{\ell,k}^1)_{k,\ell=1,\ldots,K}) \in \mathcal F_j$. Let $\ell\in\{1,\ldots,K\}$ and let 
$$\phi_{j,A_{1,\ell}}  = \1_{\left\{N^\ell(A_{1,\ell}) - \Lambda^\ell(A_{1,\ell}; f_0) \geq  jT\epsilon_T/8\right\}}.$$ 
By using %Lemmas~\ref{control} and \ref{lem:N0T} and 
\eqref{munu}, observe that on the event $\Omega_T $,  
\begin{equation*}
\begin{split}
\int_0^T \lambda^{\ell}_s(f_0)ds &=\nu_\ell^0 T + \sum_{k=1}^K \int_0^T \int_{s-A}^{s^-} h_{k,\ell}^0(s-u)dN_u^kds \nonumber\\
& \leq \nu_{\ell}^0 T + \sum_{k=1}^K  \int_{-A}^{T^-} \int_0^T \1_{u<s\leq A+u} h_{k,\ell}^0(s-u)ds dN_u^k
\end{split}
\end{equation*}
and for $T$ large enough,
\begin{equation}\label{lambda0}
\int_0^T \lambda_s^\ell(f_0) ds \leq  \nu_\ell^0 T + \sum_{k=1}^K \rho_{k,\ell}^0 N^k[-A,T]\leq 2 T \mu_\ell^0. 
\end{equation}
Let  $j\leq \sqrt{\mu_\ell^0 } \epsilon_T^{-1}$ and $x = x_1j^2T \epsilon_T^2$, for $x_1$ a constant. We use inequality (7.7) of \cite{HRR}, with $\tau=T$, $H_t= 1_{A_{1,\ell}}(t)$,  $v=2 T \mu_\ell^0$ and
$M_T=N^\ell(A_{1,\ell}) - \Lambda^\ell(A_{1,\ell}; f_0).$ So,
\begin{equation*}
 \P_0\left( \left\{ N^\ell(A_{1,\ell}) - \Lambda^\ell(A_{1,\ell}; f_0) \geq \sqrt{2vx}+\frac{x}{3}\right\} \cap \Omega_T\right) \leq e^{- x_1 j^2 T \epsilon_T^2 }.
\end{equation*}
If $x_1\leq 1/(1024\mu_\ell^0)$ and $x_1\leq 36$, we have that
\begin{equation}\label{x1}
\sqrt{2vx}+\frac{x}{3}= 2\sqrt{\mu_\ell^0  x_1} jT \epsilon_T+ \frac{x_1 j^2 T\epsilon_T^2}{3} \leq 2\sqrt{\mu_\ell^0 x_1}\left(1+ \frac{\sqrt{x_1}}{6}\right) j T \epsilon_T\leq \frac{jT \epsilon_T}{8}.
\end{equation}
 Then
\begin{equation*}
  \P_0\left( \left\{ N^\ell(A_{1,\ell}) - \Lambda^\ell(A_{1,\ell}; f_0) \geq  \frac{jT\epsilon_T}{8} \right\} \cap \Omega_T\right) \leq e^{- x_1 j^2 T \epsilon_T^2 }.
\end{equation*}
If $j \geq \sqrt{\mu_\ell^0 } \epsilon_T^{-1}$, we apply the same inequality but with $x =x_0 j T \epsilon_T$ with $x_0 =  \sqrt{\mu_\ell^0}\times x_1$. Then,
$$\sqrt{2vx}+\frac{x}{3}= 2\sqrt{\mu_\ell^0 x_1 \sqrt{\mu_\ell^0} j  \epsilon_T} T+ \frac{x_1 \sqrt{\mu_\ell^0} j T \epsilon_T}{3} \leq 2\sqrt{\mu_\ell^0  x_1} jT \epsilon_T+ \frac{x_1 \sqrt{\mu_\ell^0} j T \epsilon_T}{3}\leq \frac{jT \epsilon_T}{8},$$
where we have used \eqref{x1}. It implies
\begin{equation*}
 \P_0\left( \left\{ N^\ell(A_{1,\ell}) - \Lambda^\ell(A_{1,\ell}; f_0) \geq  \frac{jT\epsilon_T}{8} \right\} \cap \Omega_T\right) \leq e^{- x_0  jT \epsilon_T}.
\end{equation*}
Finally $\E_0\left[\1_{\Omega_T} \phi_{j,A_{1,\ell}}\right]\leq e^{-  x_1 T j\epsilon_T(\sqrt{\mu_\ell^0 }  \wedge j \epsilon_T)}.$
%Finally, 
%\begin{equation*}
 %\P_0\left( \left\{ N^\ell(A_1) - \Lambda^\ell(A_1; f_0) \geq     jT\epsilon_T /8 \right\} \cap \Omega_T\right) \leq e^{- x_1 j T \epsilon_T }.
%\end{equation*}
Now, assume that 
$$\int_{A_{1,\ell}}( \lambda_t^\ell(f_1) - \lambda_t^\ell(f_0))dt \geq \int_{A_{1,\ell}^c}( \lambda_t^\ell(f_0) - \lambda_t^\ell(f_1))dt .$$
Then 
\begin{equation}\label{case1:d1}
\frac{\| \lambda^\ell(f_1) - \lambda^\ell(f_0) \|_1 }{ 2}:=\frac{\int_0^T|\lambda^\ell_t(f_1) - \lambda^\ell_t(f_0)|dt}{ 2}\leq \int_{A_{1,\ell}}( \lambda_t^\ell(f_1) - \lambda_t^\ell(f_0))dt.
\end{equation}
Let $f= ((\nu_k)_{k=1,\ldots,K}, (h_{\ell,k})_{k,\ell=1,\ldots,K}) \in S_j$ satisyfing $\|f - f_1\|_1 \leq \zeta j\epsilon_T$ for some $\zeta >0$. Then,
\begin{equation}\label{upbound:d1}
\begin{split}
\| \lambda^\ell(f) - \lambda^\ell(f_1)\|_1 &\leq T|\nu_\ell - \nu_\ell^1| + \int_0^T \left|\int_{t-A}^{t^-}\sum_{k} (h_{k,\ell} - h_{k,\ell}^1)(t-u)dN^k_u\right|dt \\
&\leq 
T |\nu_\ell - \nu_\ell^1| +\sum_k \int_0^T \int_{t-A}^{t^-} |(h_{k,\ell} - h_{k,\ell}^1)(t-u)|dN^k_udt \\
&\leq T|\nu_\ell - \nu_\ell^1| + \max_k N^k[-A,T] \sum_k \|h_{k,\ell} - h_{k,\ell}^1\|_1 \leq T N_0 \|f - f_1\|_1 
\end{split}
\end{equation}
and $\| \lambda^\ell(f) - \lambda^\ell(f_1)\|_1 \leq TN_0\zeta  j\epsilon_T$. Since $f\in S_j$, there exists $\ell$ (depending on $f$) such that
 $$\| \lambda^\ell(f) - \lambda^\ell(f_0)\|_1\geq jT\epsilon_T.$$
 This implies in particular that if $N_0 \zeta < 1$, 
\begin{equation*}
\begin{split}
 \| \lambda^\ell(f_1) - \lambda^\ell(f_0)\|_1 & \geq \| \lambda^\ell(f) - \lambda^\ell(f_0)\|_1 - TN_0\zeta  j\epsilon_T \geq (1-N_0\zeta)Tj \epsilon_T.
 % \| \lambda^\ell(f_1) - \lambda^\ell(f_0)\|_1 & \leq \| \lambda^\ell(f) - \lambda^\ell(f_0)\|_1 + TN_0\zeta  j\epsilon_T \leq 3Tj \epsilon_T/2.
  \end{split}
  \end{equation*}
We then have
\begin{equation*}
\begin{split}
 \Lambda^\ell(A_{1,\ell}; f) -  \Lambda^\ell(A_{1,\ell}; f_0) &=  \Lambda^\ell(A_{1,\ell}; f) -  \Lambda^\ell(A_{1,\ell}; f_1)+  \Lambda^\ell(A_{1,\ell}; f_1) -  \Lambda^\ell(A_{1,\ell}; f_0) \\
 & \geq -\| \lambda^\ell(f) - \lambda^\ell(f_1)\|_1 + \int_{A_{1,\ell}} ( \lambda_t^\ell(f_1) - \lambda_t^\ell(f_0) ) dt  \\
 &\geq - \| \lambda^\ell(f) - \lambda^\ell(f_1)\|_1 + \frac{\| \lambda^\ell(f_1) - \lambda^\ell(f_0) \|_1 }{ 2}\\ 
&\geq -TN_0\zeta  j\epsilon_T+ \frac{(1-N_0\zeta)Tj \epsilon_T}{2}=(1/2-3N_0\zeta/2)Tj \epsilon_T.
\end{split}
\end{equation*}
Taking $\zeta=1/(6N_0)$ leads to
\begin{eqnarray*}
\E_f\left[\1_{f\in S_{j}} (1 - \phi_{j,A_{1,\ell}} ) \1_{\Omega_T} |\mathcal G_{0^-}\right]&=&\E_f\left[ \1_{f\in S_{j}}  \1_{\left\{N^\ell(A_{1,\ell}) - \Lambda^\ell(A_{1,\ell}; f_0) <  jT\epsilon_T/8\right\}} \1_{\Omega_T} | \mathcal G_{0^-} \right]\\
& \leq& \E_f\left[ \1_{f\in S_{j}} \1_{\left\{N^\ell(A_{1,\ell}) - \Lambda^\ell(A_{1,\ell}; f)  \leq -jT \epsilon_T/8 \right\}}\1_{\Omega_T} | \mathcal G_{0^-} \right]\\
& \leq& \E_f\left[\1_{\left\{N^\ell(A_{1,\ell}) - \Lambda^\ell(A_{1,\ell}; f)  \leq -jT \epsilon_T/8 \right\}}\1_{\Omega_T} | \mathcal G_{0^-} \right].
\end{eqnarray*}
Note that we can adapt  inequality (7.7) of \cite{HRR}, with $H_t = \1_{A_{1,\ell}}(t)$ to the case of conditional probability given $\mathcal G_{0^-}$ since the process $E_t$ defined in the proof of Theorem 3 of  \cite{HRR}, being a supermartingale, satisfies $\mathbb E_f[E_t|\mathcal G_{0^-}] \leq E_0 =1$ and, given that 
from \eqref{ineq:nu-rho} and \eqref{lambda0},
$$\int_0^T \lambda^\ell_s(f) ds \leq  \nu_\ell T + \sum_{k=1}^K \rho_{k,\ell} N^k[-A,T]\leq 2 T \mu_\ell^0+K(j+1)T\epsilon_T=:\tilde v$$
for $T$ large enough, we obtain:
\begin{equation*}
 \E_f\left[  \1_{\left\{N^\ell(A_{1,\ell}) - \Lambda^\ell(A_{1,\ell}; f)  \leq -\sqrt{2\tilde vx}-\frac{x}{3} \right\}}\1_{\Omega_T} | \mathcal G_{0^-} \right] \leq e^{- x}.
\end{equation*}
We use the same computations as before, observing that $\tilde v=v+K(j+1)T\epsilon_T$.\\
If  $j\leq \sqrt{\mu_\ell^0 } \epsilon_T^{-1}$ we set $x = x_1j^2T \epsilon_T^2$, for $x_1$ a constant. Then,
\begin{eqnarray*}
\sqrt{2\tilde vx}+\frac{x}{3}&\leq&\sqrt{2 vx}+\frac{x}{3}+\sqrt{2K(j+1)T\e_T x}\\
&\leq&2\sqrt{\mu_\ell^0  x_1} jT \epsilon_T+ \frac{x_1 j^2 T\epsilon_T^2}{3} +\sqrt{2K(j+1) \e_Tx_1}jT\e_T\\\
&\leq&2\sqrt{\mu_\ell^0 x_1}\left(1+ \frac{\sqrt{x_1}}{6}\right) j T \epsilon_T+2\sqrt{Kj\e_T x_1} j T \epsilon_T\\
&\leq&\left(2\sqrt{\mu_\ell^0 x_1}\left(1+ \frac{\sqrt{x_1}}{6}\right) +2\sqrt{K\sqrt{\mu_\ell^0 } x_1} \right)j T \epsilon_T.
\end{eqnarray*}
Therefore, if $x_1\leq \min \left(36, 1/(4096\mu^0_\ell),1/\left(1024K\sqrt{\mu_\ell^0}\right)\right)$, then
$$\sqrt{2\tilde vx}+\frac{x}{3}\leq \frac{jT \epsilon_T}{8}.$$
If  $j\geq \sqrt{\mu_\ell^0 } \epsilon_T^{-1}$, we set $x =x_0 j T \epsilon_T$ with $x_0 =  \sqrt{\mu_\ell^0}\times x_1$. Then,
\begin{eqnarray*}
\sqrt{2\tilde vx}+\frac{x}{3}&\leq&\sqrt{2 vx}+\frac{x}{3}+\sqrt{2K(j+1)T\e_T x}\\
&\leq&2\sqrt{\mu_\ell^0 x_1 \sqrt{\mu_\ell^0} j  \epsilon_T} T+ \frac{x_1 \sqrt{\mu_\ell^0} j T \epsilon_T}{3}+\sqrt{2K(j+1)T\e_T\sqrt{\mu_\ell^0} x_1 j T \epsilon_T}\\
&\leq& 2\sqrt{\mu_\ell^0  x_1} jT \epsilon_T+ \frac{x_1 \sqrt{\mu_\ell^0} j T \epsilon_T}{3}+2\sqrt{K\sqrt{\mu_\ell^0 } x_1}  j T \epsilon_T\leq \frac{jT \epsilon_T}{8}.
\end{eqnarray*}
Therefore, 
\begin{equation*}
  \E_f\left[ \1_{\left\{N^\ell(A_{1,\ell}) - \Lambda^\ell(A_{1,\ell}; f)  \leq -jT \epsilon_T/8 \right\}}\1_{\Omega_T} | \mathcal G_{0^-} \right] \leq e^{-  x_1 T j\epsilon_T(\sqrt{\mu_\ell^0 }  \wedge j \epsilon_T)}.
\end{equation*}
 Now, if $$\int_{A_{1,\ell}}( \lambda_t^\ell(f_1) - \lambda_t^\ell(f_0))dt < \int_{A_{1,\ell}^c}( \lambda_t^\ell(f_0) - \lambda_t^\ell(f_1))dt,$$
then 
$$
\int_{A_{1,\ell}^c}( \lambda_t^\ell(f_1) - \lambda_t^\ell(f_0))dt \geq \frac{\| \lambda^\ell(f_1) - \lambda^\ell(f_0) \|_1 }{ 2}
$$
and the same computations are run with $A_{1,\ell}$ playing the role of $A_{1,\ell}^c$. 
This ends the proof of Lemma~ \ref{lem:test:d1}.
\end{proof}
%%%%%%%%%%%%%%%%%%%%%

\subsection{Control of the denominator}\label{sec:denominator}
The following lemma gives a control of $D_T$.
%\max_{\ell,k}\left\|\frac{h_{\ell,k}}{h^0_{\ell,k}}\right\|_\infty\leq T^B$$
%\textcolor{blue}{Put Lemma:Kullback and the variance in the same lemma}
\begin{lemma}\label{lemma:Kullback}
Let $$KL(f_0,f)=\E_0[L_T(f_0)-L_T(f)].$$
On $B(\e_T, B)$,
\begin{equation}\label{KLcontrol}
0\leq KL(f_0,f)\leq \kappa\log(r_T^{-1}) T\e_T^2,
\end{equation}
for $T$ larger than $T_0$, with $T_0$ some constant depending on $f_0$,
 with
\begin{equation}\label{kappa}
\kappa = 4\sum_{k=1}^K(\nu^0_k)^{-1}\left(3+4K\sum_{\ell=1}^K \left(A\E_0[( \lambda^\ell_0(f_0))^2]+ \E_0[\lambda^\ell_0(f_0)]\right)\right)
\end{equation}
and $r_T$ is defined in \eqref{deltaT}.
\begin{equation}\label{Deno}
\mathbb P_0\left(L_T(f_0) - L_T(f) \geq (\kappa\log(r_T^{-1}) + 1) T\epsilon_T^2  \right)  \leq \frac{ C\log\log(T)\log^3 (T) }{ T\epsilon_T^2 },
\end{equation}
for $C$ a constant only depending on $f_0$ and $B$.
\end{lemma}

\begin{proof}
We consider the set $\tilde\Omega_T$ defined in Lemma~\ref{control} and we set ${\mathcal N}_T=C_\alpha \log T$. 
We have:
\begin{eqnarray*}
KL(f^0,f)&=&\sum_{k=1}^K\E_0\left[\int_0^T\log\left(\frac{\lambda_t^k(f_0)}{\lambda_t^k(f)}\right)dN^k_t
-\int_0^T\left(\lambda_t^k(f_0)-\lambda_t^k(f)\right)dt\right]\\
&=&\sum_{k=1}^K\E_0\left[\int_0^T\log\left(\frac{\lambda_t^k(f_0)}{\lambda_t^k(f)}\right)\lambda_t^k(f_0)dt
-\int_0^T\left(\lambda_t^k(f_0)-\lambda_t^k(f)\right)dt\right]\\
&=&\sum_{k=1}^K\E_0\left[\int_0^T\Psi\left(\frac{\lambda_t^k(f)}{\lambda_t^k(f_0)}\right)\lambda_t^k(f_0)dt\right],
%&=&\sum_{k=1}^K\E_0\left[\int_0^T\left(-\log\left(\frac{\lambda_t^k(f)}{\lambda_t^k(f_0)}\right)-1+\frac{\lambda_t^k(f)}{\lambda_t^k(f_0)}\right)\lambda_t^k(f_0)dt\right].
\end{eqnarray*}
where for $u>0$, $\Psi(u):=-\log(u)-1+u\geq 0$. First, observe that on $\tilde\Omega_T\cap B(\e_T, B)$,
\begin{equation}\label{deltaT}
\frac{\lambda_t^k(f)}{\lambda_t^k(f_0)}\geq\frac{\nu_k}{\nu^0_k+\sum_{\ell=1}^K\int_{t-A}^{t-}h^0_{\ell,k}(t-u)dN^{\ell}(u)}\geq\frac{\min_k\nu^0_k-\e_T}{\max_k\nu^0_k+\max_{\ell,k}\|h_{\ell,k}^0\|_\infty K{\mathcal N}_T}=:r_T.
\end{equation}
Furthermore, observe that for $u\in[r_T,1/2)$, $\Psi(u)\leq\log(r_T^{-1})$, since $r_T=o(1)$.
And for all $u\geq 1/2$, $\Psi(u) \leq (u-1)^2$. Finally, for any $u\geq r_T$,
$$
\Psi(u)\leq 4\log(r_T^{-1})(u-1)^2.
$$
Therefore, on $B(\e_T, B)$, we have
\begin{eqnarray*}
0\leq KL(f^0,f)&\leq& 4\log(r_T^{-1})\sum_{k=1}^K\E_0\left[\int_0^T\frac{ (\lambda_t^k(f_0)-\lambda_t^k(f))^2}{ \lambda_t^k(f_0)}\1_{\tilde\Omega_T} dt\right] 
+ R_T\\
&\leq&4\log(r_T^{-1})\sum_{k=1}^K(\nu^0_k)^{-1}\E_0\left[\int_0^T(\lambda_t^k(f_0)-\lambda_t^k(f))^2dt\right] +R_T
\end{eqnarray*}
where 
$$R_T  = \sum_{k=1}^K\E_0\left[\1_{\tilde\Omega_T^c}\int_0^T\left(-\log\left(\frac{\lambda_t^k(f)}{\lambda_t^k(f_0)}\right)-1+\frac{\lambda_t^k(f)}{\lambda_t^k(f_0)}\right)\lambda_t^k(f_0)dt\right]. $$
We first deal with the first term. Using stationarity of the process and Proposition 2 of \cite{HRR}
\begin{equation*}
\begin{split}
\E_0\left[\int_0^T(\lambda_t^k(f_0)-\lambda_t^k(f))^2dt\right]&\leq 2T(\nu^0_k-\nu_k)^2+2\int_0^T\E_0\left[\left(\sum_{\ell=1}^K\int_{t-A}^{t^-}(h_{\ell,k}-h^0_{\ell,k})(t-u)dN^{\ell}(u)\right)^2\right]dt\\
& \leq 2T\e_T^2+4K\int_0^T\E_0\left[\sum_{\ell=1}^K\left(\int_{t-A}^{t^-}(h_{\ell,k}-h^0_{\ell,k})(t-u)\lambda_u^\ell(f_0)du\right)^2\right]dt\\
& \quad +4K\int_0^T\E_0\left[\sum_{\ell=1}^K\left(\int_{t-A}^{t^-}(h_{\ell,k}-h^0_{\ell,k})(t-u)\left(dN^{\ell}_u - \lambda_u^\ell(f_0)du\right)\right)^2\right]dt\\
& \leq 2T\e_T^2+4K \sum_{\ell=1}^K \|h_{\ell,k}-h^0_{\ell,k}\|_2^2\int_0^T\int_{t-A}^{t^-}\E_0[(\lambda_u^\ell(f_0))^2]dudt \\
& \quad +4K\int_0^T\sum_{\ell=1}^K\int_{t-A}^{t^-}(h_{\ell,k}-h^0_{\ell,k})^2(t-u)\E_0\left[\lambda_u^\ell(f_0)\right]dudt\\
& \leq 2T\e_T^2+4KT\sum_{\ell=1}^K \|h_{\ell,k}-h^0_{\ell,k}\|_2^2\left(A\E_0[( \lambda_0^\ell(f_0))^2]+ \E_0[\lambda_0^\ell(f_0)]\right)\\
&\leq T\e_T^2\left(2+4K\sum_{\ell=1}^K \left(A\E_0[( \lambda_0^\ell(f_0))^2]+ \E_0[\lambda_0^\ell(f_0)]\right)\right).
\end{split}
\end{equation*}
We now deal with $R_T$.
%\begin{eqnarray*}
%R_T&=&\sum_{k=1}^K\E_0\left[\int_0^T\Psi\left(\frac{\lambda_t^k(f)}{\lambda_t^k(f_0)}\right)\lambda_t^k(f_0) \1_{\tilde\Omega_T^c}1_{\left\{\frac{\lambda_t^k(f)}{\lambda_t^k(f_0)}\geq \frac12\right\}}dt\right]+\sum_{k=1}^K\E_0\left[\int_0^T\Psi\left(\frac{\lambda_t^k(f)}{\lambda_t^k(f_0)}\right)\lambda_t^k(f_0) \1_{\tilde\Omega_T^c}1_{\left\{\frac{\lambda_t^k(f)}{\lambda_t^k(f_0)}< \frac12\right\}}dt\right]\\
%&\leq&\sum_{k=1}^K(\nu^0_k)^{-1}\E_0\left[\int_0^T(\lambda_t^k(f_0)-\lambda_t^k(f))^2dt\1_{\tilde\Omega_T^c}\right] 
%\end{eqnarray*}
%Also 
%$$|R_T|\leq \sum_{k=1}^K\E_0\left[\int_0^T\left(\lambda_t^k(f_0)\left|\log\frac{\lambda_t^k(f)}{\lambda_t^k(f_0)}\right|+\lambda_t^k(f_0)+\lambda_t^k(f)\right)dt\times 1_{\Omega_T^c}\right].$$
%Since $\lambda^{k}_t\geq \nu_k^0$ and
%$$0\leq \lambda^{k}_t\leq\nu_k+\sum_{\ell=1}^K\|h_{\ell,k}\|_{\infty}\sup_{t\in[0,T]}N^{\ell}([t-A,t)),$$
We have, on $B(\e_T, B)$, 
\begin{eqnarray} \label{ub:ratio}
\frac{\lambda_t^k(f)}{\lambda_t^k(f_0)}&\leq&(\nu_k^0)^{-1}\left(\nu_k+\sum_{\ell=1}^K\|h_{\ell,k}\|_{\infty}\sup_{t\in[0,T]}N^{\ell}([t-A,t))\right)\\
&\leq&(\nu_k^0)^{-1}\left(\nu^0_k+\epsilon_T+B\sum_{\ell=1}^K\sup_{t\in[0,T]}N^{\ell}([t-A,t))\right).
\end{eqnarray}
Conversely,
\begin{eqnarray} \label{lb:ratio}
\frac{\lambda_t^k(f)}{\lambda_t^k(f_0)}\geq(\nu_k^0-\e_T)\left(\nu^0_k+\sum_{\ell=1}^K\|h^0_{\ell,k}\|_{\infty}\sup_{t\in[0,T]}N^{\ell}([t-A,t))\right)^{-1}.\end{eqnarray}
So, using Lemma~\ref{control}, if $\alpha$ is an absolute constant large enough, $R_T=o(1)$ and
$$R_T=o(T\e_T^2).$$
Choosing $\kappa = 4\sum_{k=1}^K(\nu^0_k)^{-1}\left(3+4K\sum_{\ell=1}^K \left(A\E_0[( \lambda^\ell_0(f_0))^2]+ \E_0[\lambda^\ell_0(f_0)]\right)\right)$
%Choosing $\kappa = 1+\left[ 1 + 2^{K} \left\{A + 2^{K-1}\sum_\ell  E_0([N^{\ell}([-A,0))]^2) +\sum_\ell  \mu_{\ell}^0\right\}\right]$ 
terminates the proof of~\eqref{KLcontrol}.
Note that if $B(\epsilon_T, B)$ is replaced with $B_\infty (\epsilon_T, B) $ (see Remark~\ref{rem:supsorm}) then 
 $$ 
\frac{\lambda_t^k(f)}{\lambda_t^k(f_0)} \leq 1 + \frac{ | \nu_k - \nu_k^0| + \sum_\ell \| h_{\ell,k} - h_{\ell,k}\|_\infty \mathcal N_T}{\nu_k^0} 
$$
and 
 $$ 
\frac{\lambda_t^k(f)}{\lambda_t^k(f_0)} \geq 1 - \frac{ | \nu_k - \nu_k^0| + \sum_\ell \| h_{\ell,k} - h_{\ell,k}\|_\infty \mathcal N_T}{\nu_k^0} 
$$
so that we can take $r_T = 1/2$ and $R_T = o(T\epsilon_T^2)$. 

 We now study 
$$\mathcal L_T:=L_T(f_0) - L_T(f)  - \E_0[L_T(f_0) - L_T(f)].$$ 
We have for any integer $Q_T$ such that $x:=T/(2Q_T)>A$,
 \begin{eqnarray*}
L_T(f_0) - L_T(f)&=&\sum_{k=1}^K\left(\int_0^T\log\left(\frac{\lambda_t^k(f_0)}{\lambda_t^k(f)}\right)dN^k_t-\int_0^T\left(\lambda_t^k(f_0)-\lambda_t^k(f)\right)dt\right)\\
&=&\sum_{q=0}^{Q_T-1}\int_{2qx}^{2qx+x}\sum_{k=1}^K\left(\log\left(\frac{\lambda_t^k(f_0)}{\lambda_t^k(f)}\right)dN^k_t-\left(\lambda_t^k(f_0)-\lambda_t^k(f)\right)dt\right)\\
&&\hspace{1cm}+\sum_{q=0}^{Q_T-1}\int_{2qx+x}^{2qx+2x}\sum_{k=1}^K\left(\log\left(\frac{\lambda_t^k(f_0)}{\lambda_t^k(f)}\right)dN^k_t-\left(\lambda_t^k(f_0)-\lambda_t^k(f)\right)dt\right)\\
&=:&\sum_{q=0}^{Q_T-1}F_q+ \sum_{q=0}^{Q_T-1}\tilde F_q.
 \end{eqnarray*}
Note that $F_q$ is a measurable function of the points of $N$ appearing in $[2qx-A ; 2qx+x)$ denoted by $\mathcal F(N_{|[2qx-A ; 2qx+x)}).$ Using Proposition 3.1 of \cite{RBR}, we consider an i.i.d. sequence $(M_q^x)_{q=0,\ldots,Q_T-1}$ of Hawkes processes with the same distribution as $N$ but restricted to $[2qx-A ; 2qx+x)$ and such that for all $q$, the variation distance between $M_q^x$ and $N_{|[2qx-A ; 2qx+x)}$ is less than $2\P_0(T_e>x-A)$, where $T_e$ is the extinction time of the process. We then set for any $q$,
$$G_q= \mathcal F(M_q^x).$$
We have built an i.i.d. sequence $(G_q)_{q=0,\ldots,Q_T-1}$ with the same distributions as the $F_q$'s. Furthermore, for any $q$, $$\P_0(F_q\not=G_q)\leq 2\P_0(T_e>x-A).$$
%Considering a similar construction denoted $(G_q^2)_{q=0,\ldots,Q_T-1}$ associated with he $F_q^2$'s, 
We now have, by stationarity
  \begin{eqnarray*}
  \P_0(\mathcal L_T\geq T\e_T^2)&=& \P_0 \left(L_T(f_0) - L_T(f)  - \E_0[L_T(f_0) - L_T(f)]\geq T\epsilon_T^2\right)\\
  &=&\P_0 \left(\sum_{q=0}^{Q_T-1}(F_q-\E_0[F_q])+ \sum_{q=0}^{Q_T-1}(\tilde F_q-\E_0[\tilde F_q])\geq T\epsilon_T^2\right)\\
 &\leq&2\P_0 \left(\sum_{q=0}^{Q_T-1}(F_q-\E_0[F_q])\geq T\epsilon_T^2/2\right)\\
 &\leq&2\P_0 \left(\sum_{q=0}^{Q_T-1}(G_q-\E_0[G_q])\geq T\epsilon_T^2/2\right)+2\P_0 \left(\exists q;\ F_q\not=G_q\right)\\
 &\leq&2\P_0 \left(\sum_{q=0}^{Q_T-1}(G_q-\E_0[G_q])\geq T\epsilon_T^2/2\right)+4Q_T\P_0(T_e>x-A).
 \end{eqnarray*}
 We first deal with the first term of the previous expression:
  \begin{eqnarray*}
  \P_0 \left(\sum_{q=0}^{Q_T-1}(G_q-\E_0[G_q])\geq T\epsilon_T^2/2\right)&\leq&\frac{4}{T^2\e_T^4}\var_0\left(\sum_{q=0}^{Q_T-1}G_q\right)\\
  &\leq&\frac{4}{T^2\e_T^4}\sum_{q=0}^{Q_T-1}\var_0\left(G_q\right)\\
  &\leq&\frac{4Q_T}{T^2\e_T^4}\var_0\left(G_0\right)=\frac{4Q_T}{T^2\e_T^4}\var_0\left(F_0\right).
  \end{eqnarray*}
Now, by setting $d{\mathcal M}^{(k)}_t=dN^k_t-\lambda_t^k(f_0)dt$,
\begin{eqnarray*}
\var_0\left(F_0\right)&\leq&\E_0\left[F_0^2\right]\\
&\leq&\E_0\left[\left(\sum_{k=1}^K\int_0^\frac{T}{2Q_T}\log\left(\frac{\lambda_t^k(f_0)}{\lambda_t^k(f)}\right)dN^k_t-\sum_{k=1}^K\int_0^\frac{T}{2Q_T}(\lambda_t^k(f_0)-\lambda_t^k(f))dt\right)^2\right]\\
&\lesssim&\sum_{k=1}^K\E_0\left[\left(\int_0^\frac{T}{2Q_T}\Psi\left(\frac{\lambda_t^k(f)}{\lambda_t^k(f_0)}\right)\lambda_t^k(f_0)dt+\int_0^\frac{T}{2Q_T}\log\left(\frac{\lambda_t^k(f_0)}{\lambda_t^k(f)}\right)d{\mathcal M}^{(k)}_t\right)^2\right]\\
&\lesssim&\sum_{k=1}^K\E_0\left[\left(\int_0^\frac{T}{2Q_T}\Psi\left(\frac{\lambda_t^k(f)}{\lambda_t^k(f_0)}\right)\lambda_t^k(f_0)dt\right)^2\right]+\E_0\left[\left(\int_0^\frac{T}{2Q_T}\log\left(\frac{\lambda_t^k(f_0)}{\lambda_t^k(f)}\right)d{\mathcal M}^{(k)}_t\right)^2\right]\\
&\lesssim&\sum_{k=1}^K\frac{T}{Q_T}\E_0\left[\int_0^\frac{T}{2Q_T}\Psi^2\left(\frac{\lambda_t^k(f)}{\lambda_t^k(f_0)}\right)(\lambda_t^k(f_0))^2dt
\right]+\E_0\left[\int_0^\frac{T}{2Q_T}\log^2\left(\frac{\lambda_t^k(f_0)}{\lambda_t^k(f)}\right)\lambda^k_t(f_0)dt\right].
\end{eqnarray*}
Note that on $\tilde\Omega_T$, for any $t\in [0; T/(2Q_T)]$,
$$0\leq \Psi\left(\frac{\lambda_t^k(f)}{\lambda_t^k(f_0)}\right)\lambda_t^k(f_0)\leq C_1(B,f_0){\mathcal N}_T^2,$$
where $C_1(B,f_0)$ only depends on $B$ and $f_0$.
Then,
$$
\E_0\left[\1_{\tilde\Omega_T}\int_0^\frac{T}{2Q_T}\Psi^2\left(\frac{\lambda_t^k(f)}{\lambda_t^k(f_0)}\right)(\lambda_t^k(f_0))^2dt
\right]\leq C_1(B,f_0){\mathcal N}_T^2\times\E_0\left[\1_{\tilde\Omega_T}\int_0^\frac{T}{2Q_T}\Psi\left(\frac{\lambda_t^k(f)}{\lambda_t^k(f_0)}\right)\lambda_t^k(f_0)dt
\right]$$
and using same arguments as for the bound of $KL(f^0,f)$, the previous term is bounded by $\log(r_T^{-1}){\mathcal N}_T^2\times (T/Q_T)\e_T^2$ up to a constant.
Since for any $u\geq 1/2$, we have $|\log(u)|\leq 2|u-1|$, we have for any $u\geq r_T$,
$$|\log(u)|\leq 2\log(r_T^{-1})|u-1|$$
and
\begin{eqnarray*}
\E_0\left[\1_{\tilde\Omega_T}\int_0^\frac{T}{2Q_T}\log^2\left(\frac{\lambda_t^k(f_0)}{\lambda_t^k(f)}\right)\lambda^k_t(f_0)dt\right]&\leq&4\log^2(r_T^{-1})(\nu_k^0)^{-1}\E_0\left[\1_{\tilde\Omega_T}\int_0^\frac{T}{2Q_T}(\lambda^k_t(f_0)-\lambda^k_t(f))^2dt\right]\\
&\lesssim&\log^2(r_T^{-1})(T/Q_T)\e_T^2. 
\end{eqnarray*}
By taking $\alpha \geq 2 $ and using Lemma \ref{control}, we obtain:
$$
\E_0\left[\1_{\tilde\Omega_T^c}\int_0^\frac{T}{2Q_T}\Psi^2\left(\frac{\lambda_t^k(f)}{\lambda_t^k(f_0)}\right)(\lambda_t^k(f_0))^2dt
\right]+\E_0\left[\1_{\tilde\Omega_T^c}\int_0^\frac{T}{2Q_T}\log^2\left(\frac{\lambda_t^k(f_0)}{\lambda_t^k(f)}\right)\lambda^k_t(f_0)dt\right]=o(TQ_T^{-1}\e_T^2).$$
Finally,
$$\var_0\left(F_0\right)\leq C_2(B,f_0)\log(r_T^{-1}){\mathcal N}_T^2\times (T/Q_T)^2\e_T^2.$$
for $C_2(B,f_0)$ a constant only depending on $B$ and $f_0$, and
$$ \P_0(\mathcal L_T\geq T\e_T^2)\leq 8C_2(B,f_0)\log(r_T^{-1}){\mathcal N}_T^2\times (T/Q_T)\times(1/(T\e_T^2)+4Q_T\P_0(T_e>x-A).$$
It remains to deal with the last term of the previous expression. The proof of Proposition~3 of \cite{HRR} shows that there exists a constant $D$ only depending on $f_0$ such that if we take
$x=D\log T$, which is larger than $A$ for $T$ large enough, then
$$4Q_T\P_0(T_e>x-A)=o(T^{-1}).$$
We now have 
$$\log(r_T^{-1}){\mathcal N}_T^2\times (T/Q_T)=O(\log\log(T)\log^3(T)),$$
which ends the proof of the lemma.
\end{proof}
%%%%%%%%%%%%%%%%%%%
\subsection{Proof of Theorem \ref{th:L1}}\label{sec:pr:thL1}
Define  
 $$ A_{L_1}(w_T \varepsilon_T) = \{ f \in \mathcal F; \ \|f-f_0\|_1 \leq w_T \varepsilon_T\},$$
then 
\begin{equation*}
\Pi\left( A_{L_1}(w_T\varepsilon_T)^c |N \right) \leq \Pi( A_{\varepsilon_T}^c  |N ) + \Pi\left( A_{L_1}(w_T\varepsilon_T)^c\cap A_{\varepsilon_T} |N \right). 
\end{equation*}
Using Assumption (i), we just need to prove that 
\begin{equation}\label{th:L1:eq0}
\E_0\left[\1_{\Omega_{1,T}} \Pi\left( A_{L_1}(w_T\varepsilon_T)^c\cap A_{\varepsilon_T} |N \right) \right]= o(1)
\end{equation}
for some well chosen set $\Omega_{1,T} \subset \Omega_T$ such that 
\begin{equation}\label{omega1}
\P_0(\Omega_{1,T}^c\cap \Omega_T ) = o(1).
\end{equation}
Using \eqref{ineq:nu-rho}, there exists $C_0$ such that for all $f\in  A_{\varepsilon_T}$, on $\Omega_T$,
$$\sum_\ell \nu_\ell + \sum_{\ell, k} \rho_{\ell, k} \leq C_0.$$ 
%Using \eqref{ineq:nu-rho} we have for all $\ell$,  $ \nu_\ell + \sum_{k} \rho_{k,\ell} \mu_k^0 \leq \mu_\ell^0+ O(\varepsilon_T +\delta_T)$. 
Therefore,   on $\Omega_T$,
$$A_{L_1}(w_T\varepsilon_T)^c\cap A_{\varepsilon_T}\subset \{ f \in \mathcal F; \  \| f -f_0\|_1 > w_T \varepsilon_T; \sum_\ell (\nu_\ell + \sum_k \rho_{\ell,k} )\leq C_0\}.$$
We set $u_T:=u_0 (\log T)^{1/6} \varepsilon_T^{1/3}$ with $u_0$ a large constant to be chosen later. Let  $\mathcal F_T = \{ f \in \mathcal F; \|\rho\|\leq 1 - u_T\}$.
%\textcolor{red}{Let  $\mathcal F_T = \{ f \in \mathcal F; \|\rho\|\leq 1 - u_T\}$.} 
From Assumption (ii),
$$\Pi(\mathcal F_T^c) \leq e^{-2c_1T \varepsilon_T^2}$$
for $T$ large enough.  Following the same lines as in the proof of Theorem~\ref{th:d1}, we then have 
\begin{equation}\label{grosse}
\begin{split}
& \E_0\left[\1_{\Omega_{1,T}} \Pi\left( A_{L_1}(w_T\varepsilon_T)^c\cap A_{\varepsilon_T} |N \right)\right] \leq 
\P_0(D_T < e^{-c_1 T \varepsilon_T^2 } ) \\
& + e^{c_1 T \varepsilon_T^2}\int_{ A_{L_1}(w_T\varepsilon_T)^c\cap \mathcal F_T} \E_0\left[ \P_f \left( \Omega_{1,T}\cap \{d_{1,T}(f ,f_0) \leq \varepsilon_T\} | \mathcal G_{0^-}\right)\right]d\Pi(f)  + e^{-c_1T\varepsilon_T^2},
\end{split}
\end{equation}
where $\P_f$ denotes the stationary distribution when the true parameter is $f$.
We will now prove that $\P_f \left( \Omega_{1,T}\cap \{d_{1,T}(f ,f_0) \leq \varepsilon_T\} | \mathcal G_{0^-}\right) e^{c_1 T \varepsilon_T^2} = o_{P_0}(1)$ for all $f\in  A_{L_1}(w_T\varepsilon_T)^c\cap \mathcal F_T$.  Let $Z_{m,\ell}$ be defined by 
 $$ Z_{m,\ell} = \int_{ 2 m T/(2J_T)}^{(2m+1)T/(2J_T)}  \left| \nu_\ell - \nu_\ell^0 + \sum_{k=1}^K \int_{t-A}^{t^-} (h_{k,\ell} - h_{k,\ell}^0 )(t-s) dN_{s}^k \right| dt$$
with $J_T$ such that $J_T = \lfloor\kappa_0 (\log T)^{-1}Tu_T^2\rfloor$ and $\kappa_0$ a constant chosen later. Note that $J_T\to +\infty$ and $T/J_T\to+\infty$ when $T\to+\infty$. Since 
$Td_{1,T} ( f, f_0) \geq \max_{1\leq\ell \leq K} \sum_{m=1}^{J_T-1} Z_{m,\ell}$ we have that 
\begin{small}
\begin{equation*}
\begin{split}
\P_f \left(\Omega_{1,T} \cap\{ d_{1,T}(f, f_0) \leq \varepsilon_T\}| \mathcal G_{0^-}  \right) & \leq  \min_{1\leq\ell\leq K} \P_f \left( \Omega_{1,T}  \cap \left\{\sum_{m=1}^{J_T-1} Z_{m,\ell} \leq \varepsilon_T T \right\}| \mathcal G_{0^-} \right)  \\
&\leq \min_{1\leq\ell\leq K} \P_f \left( \left.\Omega_{1,T}  \cap \left\{\sum_{m=1}^{J_T-1} (Z_{m,\ell}-\E_f[Z_{m,\ell}]) \leq \varepsilon_T T - (J_T-1)\E_f[Z_{1,\ell}] \right\}\right|  \mathcal G_{0^-} \right).
\end{split}
\end{equation*}
\end{small}
From Lemma \ref{lem:EZm} we have that there exists $\ell $  (depending on $f$ and $f^0$)  such that 
$\E_f[Z_{1,\ell}] \geq C T \|f-f_0\|_1 /J_T$ for some $C>0$  so that if $f \in A_{L_1}(w_T\varepsilon_T)^c$ then, since $w_T\to+\infty$, 
\begin{small}
\begin{equation*}
\begin{split}
\P_f \left(\Omega_{1,T} \cap\{ d_{1,T}(f, f_0) \leq T\varepsilon_T\}| \mathcal G_{0^-}  \right) 
&\leq \max_\ell \P_f \left(\left. \Omega_{1,T}  \cap \left\{\sum_{m=1}^{J_T-1} [Z_{m,\ell}-\E_f[Z_{m,\ell}]] \leq  - \frac{CT \|f - f_0\|_1}{2} \right\}\right| \mathcal G_{0^-} \right).  
\end{split}
\end{equation*}
\end{small}
The problem in dealing with the right hand side of the above inequality is that the $Z_{m, \ell}$'s are not independent. We therefore show that 
we can construct independent random variables $\tilde Z_{m,\ell}$ such  that, conditionally on $\mathcal G_{0^-}$, 
$\sum_{m=1}^{J_T-1} (Z_{m,\ell}-\E_f[Z_{m,\ell}])$ is close to  $\sum_{m=1}^{J_T-1} (\tilde Z_{m,\ell}-\E_f[\tilde Z_{m,\ell}]) $ on $\Omega_{1,T}$. 
For all $1 \leq m\leq J_T-1$, define 
 $N^{0, m}$   the sub-counting measure of $N$ generated from the ancestors of any type born on $[(2m-1)T/(2J_T), (2m+1)T/(2J_T)]$ and the $K$-multivariate point process $\bar N^m$ defined by $$\bar N^m = N -N^{0,m}.$$ Denote $$\tilde Z_{m, \ell} = \int_{ 2 m T/(2J_T)}^{(2m+1)T/(2J_T)}  \left| \nu_\ell - \nu_\ell^0 + \sum_{k=1}^K \int_{t-A}^{t^-} (h_{k,\ell} - h_{k,\ell}^0 )(t-s) dN_{s}^{0,m,k} \right| dt,$$
 where {$N^{0,m,k}$ if the $k$th coordinate of $N^{0,m}$.
Observe that if $I_m =  [2mT/(2J_T)-A, (2m+1)T/(2J_T)]$, then $\bar N^m(I_m)$ is the number of points of $\bar N^m$ lying in $I_m$. We have:
\begin{equation}\label{diff:ZtildeZ}
\begin{split}
 |Z_{m, \ell} - \tilde Z_{m, \ell}| & =\left| \int_{2 m T/(2J_T)}^{(2m+1)T/(2J_T)} \left( \left| \nu_\ell - \nu_\ell^0 + \sum_{k=1}^K \int_{t-A}^{t^-} (h_{k,\ell} - h_{k,\ell}^0 )(t-s) dN_{s}^{k} \right|\right.\right.\\ &
\qquad \left.\left.-\left| \nu_\ell - \nu_\ell^0 + \sum_{k=1}^K \int_{t-A}^{t^-} (h_{k,\ell} - h_{k,\ell}^0 )(t-s) dN_{s}^{0,m,k} \right|\right)dt
 \right|
\\ 
& \leq  \1_{\bar N^m(I_m)\neq 0} \sum_{k=1}^K \int_{ 2 m T/(2J_T)}^{(2m+1)T/(2J_T)} \int_{t-A}^{t^-} |(h_{k,\ell} - h_{k,\ell}^0 )(t-s)| d\bar N_{s}^{m,k}  dt \\
 &\leq \1_{\bar N^m(I_m)\neq 0} \sum_{k=1}^K \|h_{k,\ell} - h_{k,\ell}^0 \|_1 \bar N^{m,k}(I_m) \leq \|f - f_0\|_1 \bar N^m(I_m).
 \end{split}
 \end{equation}
Let   $\Omega_{1,T}  = \Omega_T \cap \{ \sum_{m=1}^{J_T-1}  \bar N^m(I_m) \leq C T/8\}$. In Lemma \ref{lem:barNm}, we prove that 
there exists $\tilde c_0 $ such that 
 $$ \mathbb P_0\left( \Omega_{1,T}^c \cap \Omega_T \right) \leq e^{- C \tilde c_0 T}.$$
and \eqref{omega1} is satisfied. Using \eqref{diff:ZtildeZ}, we have on $\Omega_{1,T} $
\begin{equation}\label{diff:Z}
 |Z_{m, \ell} - \tilde Z_{m, \ell}| \leq \|f-f_0\|_1 C T/8.
 \end{equation}
Lemma \ref{lem:barNm} proves that there exists a constant $\kappa_0>0$ (see the definition of $J_T$) such that
$$\sum_{m=1}^{J_T-1}\E_f[\bar N^m(I_m) ] \leq C T/8,$$
so that 
$$\sum_{m=1}^{J_T-1}|\E_f[Z_{m, \ell}] - \E_f[\tilde Z_{m, \ell}]| \leq\sum_{m=1}^{J_T-1}\E_f|Z_{m, \ell} - \tilde Z_{m, \ell}| \leq \|f-f_0\|_1 \sum_{m=1}^{J_T-1}\E_f[\bar N^m(I_m) ]  \leq C \|f-f_0\|_1T/8
$$
and 
\begin{align*}
\P_f \left(\Omega_{1,T} \cap\{ d_{1,T}(f, f_0) \leq T\varepsilon_T\}| \mathcal G_{0^-}  \right) 
&\leq \max_\ell \P_f \left(\left. \Omega_{1,T}  \cap \left\{\sum_{m=1}^{J_T-1} [Z_{m,\ell}-\E_f[Z_{m,\ell}]] \leq  - \frac{CT \|f - f_0\|_1}{2} \right\}\right| \mathcal G_{0^-} \right)\\
&\leq \P_f \left( \left.\sum_{m=1}^{J_T-1} (-\tilde Z_{m,\ell}+\E_f(\tilde Z_{m,\ell})) \geq  CT \|f-f_0\|_1 /4 \right| \mathcal G_{0^-}   \right).
\end{align*}
Since by construction the $\tilde Z_{m, \ell}$ are positive, independent, identically distributed and independent of $\mathcal G_{0^-}$, the Bernstein inequality gives
%hence for all $s>0$
%\begin{small}
%\begin{align*}
% \P_f \left( \left.\sum_{m=1}^{J_T} (-\tilde Z_{m,\ell}+\E_f(\tilde Z_{m,\ell})) \geq  CT \|f-f_0\|_1 /4 \right| \mathcal G_0   \right)&   \leq 
%e^{-sCT \|f-f_0\|_1 /4 + sJ_T \E_f(\tilde Z_{1,\ell}) } \left( \E_f (e^{-s \tilde Z_{m,\ell}} )\right)^{J_T}  \\
%& \leq e^{-sCT \|f-f_0\|_1 /4+ sJ_T \E_f(\tilde Z_{1,\ell})} \left( 1 -s \E_f(\tilde Z_{1,\ell}) + s^2 \E_f(\tilde Z_{1,\ell}^2)/2\right)^{J_T}\\
%&=e^{-sCT \|f-f_0\|_1 /4+ sJ_T \E_f(\tilde Z_{1,\ell})}e^{J_T\log \left( 1 -s \E_f(\tilde Z_{1,\ell}) + s^2 \E_f(\tilde Z_{1,\ell}^2)/2\right)}\\
%& \leq  e^{-s C T \|f-f_0\|_1/4 } e^{ s^2 J_T \E_f(\tilde Z_{1,\ell}^2)/2},
%\end{align*}
%\end{small}
%since for $x\geq 0$, $e^{-x}\leq 1-x+x^2/2$ and for any $x>-1$, $\log(1+x)\leq x$.  Finally, minimizing the right hand side of the previous expression with respect to $s$, we obtain:
$$\P_f \left( \left.\sum_{m=1}^{J_T-1} (-\tilde Z_{m,\ell}+\E_f(\tilde Z_{m,\ell})) \geq  CT \|f-f_0\|_1 /4 \right| \mathcal G_{0^-}   \right)   \leq  e^{ - \frac{C^2 T^2  \|f-f_0\|_1^2 }{ 32 (J_T-1) \E_f(\tilde Z_{1,\ell}^2)}}.$$
We have to bound $\E_f(\tilde Z_{1,\ell}^2).$ Observe that
\begin{align*}
\tilde Z_{m, \ell} &\leq \int_{ 2 m T/(2J_T)}^{(2m+1)T/(2J_T)}  \left| \nu_\ell - \nu_\ell^0\right| dt + \int_{ 2 m T/(2J_T)}^{(2m+1)T/(2J_T)}  \sum_{k=1}^K \int_{t-A}^{t^-} \left|(h_{k,\ell} - h_{k,\ell}^0 )(t-s) \right|dN_{s}^{0,m,k}  dt\\
&\leq\frac{T}{2J_T} \left| \nu_\ell - \nu_\ell^0\right| +\sum_{k=1}^K \|h_{k,\ell} - h_{k,\ell}^0 \|_1 N^{0,m,k}(I_m)
\end{align*}
and
\begin{equation*}
\begin{split}
\E_f\left[\tilde Z_{1,\ell}^2 \right] &\leq  \frac{ T^2}{ 2J_T^2}|\nu_\ell - \nu_\ell^0|^2 + 2K \sum_{k=1}^K \|h_{k,\ell}-h_{k,\ell}^0\|_1^2 \E_f[N^{0,1,k}(I_1)^2] \\
& \leq \frac{ T^2}{J_T^2} \|f-f_0\|_1^2 \left(\frac{1}{2} + \frac{2K\max_k\E_f[N^{0,1,k}(I_1)^2]J_T^2 }{ T^2} \right).
\end{split}
\end{equation*}
We then have to bound
$T^{-2}J_T^2\max_k\E_f[N^{0,1,k}(I_1)^2]$. Using notations of Lemma~\ref{lem:barNm}, we have:
\begin{align*}
\E_f[N^{0,1,k}(I_1)^2] &\leq \E_f\left[\left(\sum_{\ell=1}^K\sum_{T/(2J_T)\leq p\leq 3T/(2J_T)} \sum_{k=1}^{B_{p,\ell}} W_{k,p}^\ell\right)^2\right]\\
&\leq \frac{KT}{J_T}\sum_{\ell=1}^K\sum_{T/(2J_T)\leq p\leq 3T/(2J_T)}\E_f\left[\left( \sum_{k=1}^{B_{p,\ell}} W_{k,p}^\ell\right)^2\right]\\
&\leq \frac{KT}{J_T}\sum_{\ell=1}^K\sum_{T/(2J_T)\leq p\leq 3T/(2J_T)}\E_f\left[\E_f\left[\left( \sum_{k=1}^{B_{p,\ell}} W_{k,p}^\ell\right)^2\left| B_{p,\ell}\right.\right]\right]\\
&\leq \frac{KT^2}{J_T^2}\sum_{\ell=1}^K(\nu_\ell^2+\nu_\ell)\E_f[(W^{\ell})^2].
\end{align*}
We now bound $\E_f[(W^{\ell})^2]$ by using Lemma~\ref{lem:Laplacetransform}. Without loss of generality, we can assume that $\|\rho\|>1/2$. We take $t =\frac{1 - \|\rho\|}{2\sqrt{K}} \log \left(\frac{1+\|\rho\|}{2\|\rho\|} \right)$ and
$$\E_f[(W^{\ell})^2]\leq2t^{-2}\E_f[\exp(tW^\ell)]\lesssim t^{-2}\lesssim (1-\|\rho\|)^{-4}$$
and
$$T^{-2}J_T^2\max_k\E_f[N^{0,1,k}(I_1)^2]\lesssim (1-\|\rho\|)^{-4}.$$
Therefore, since $f\in{\mathcal F}_T$, there exists a constant $C'_K$ only depending on $K$ such that
\begin{equation*}
\begin{split}
\P_f \left( \left.\sum_{m=1}^{J_T-1} (-\tilde Z_{m,\ell}+\E_f(\tilde Z_{m,\ell})) \geq  CT \|f-f_0\|_1 /4 \right| \mathcal G_{0^-}   \right)  & \leq  e^{ - C'_KJ_T(1-\|\rho\|)^4}\leq  e^{ - C'_KJ_Tu_T^4}\\
& \lesssim  e^{ - C_K' \kappa_0 (\log T)^{-1}Tu_T^6} \lesssim e^{ - C_K' \kappa_0  u_0^6T \e_T^2} 
\end{split}
\end{equation*}
%\textcolor{red}{and by taking 
%\begin{equation}\label{cond-rho}
%T\e_T^2=o(J_T(1-\|\rho\|)^4),
%\end{equation}
where the last inequality follows from the definition of $u_T$ and $J_T$. We obtain the desired bound as soon as $u_0$ is large enough. 
$$\mathbb P_f \left( \Omega_{1,T} \cap \{d_{1,T} \leq T \varepsilon_T\} |\mathcal G_{0^-} \right)=o(e^{-c_1T\varepsilon_T^2}).$$
Using \eqref{grosse} and Assumption (i), we then have that \eqref{th:L1:eq0} is true, which proves the theorem.

\subsection{Proof of Corollary \ref{coro:L1hatf}} \label{pr:corL1}
Let $w_T\to+\infty$. The proof of Corollary \ref{coro:L1hatf} follows from the usual convexity argument, so that 
$$\|\hat f - f_0\|_1 \leq w_T \varepsilon_T + \E^\pi\left[ \|f - f_0\|_1 \1_{\|f - f_0\|_1 > w_T \varepsilon_T} |N\right],$$
together with a control of the second term of the right hand side similar to the proof of Theorem \ref{th:L1}. We write 
\begin{equation*}
\begin{split}
 \E^\pi\left[ \|f - f_0\|_1 \1_{\|f - f_0\|_1 > w_T \varepsilon_T} |N\right]& \leq \E^\pi\left[ \|f - f_0\|_1 \1_{A_{L_1}(w_T \varepsilon_T)^c} \1_{A_{\varepsilon_T}}|N\right]
 +\E^\pi\left[ \|f - f_0\|_1  \1_{A_{\varepsilon_T}^c}|N\right]
 \end{split}
 \end{equation*}
 and since $ \int  \|f-f_0\|_1 d\Pi(f)\leq \|f_0\|_1+ \int  \|f\|_1 d\Pi(f)<\infty$,
\begin{equation*}
\begin{split}
\P_0 & \left( \E^\pi\left[ \|f - f_0\|_1\1_{A_{L_1}(w_T \varepsilon_T)^c} \1_{A_{\varepsilon_T}} |N\right] > w_T \varepsilon_T \right) \leq \P_0\left( \Omega_{1,T}^c \right) + \P_0\left( D_T< e^{-c_1T\epsilon_T^2}\right)\\
& \quad  + \frac{e^{c_1T\epsilon_T^2}}{w_T \varepsilon_T} \int_{ A_{L_1}(w_T \varepsilon_T)^c}  \|f-f_0\|_1\E_0\left[  \P_f\left( \Omega_{1,T} \cap \{ d_{1,T}(f_0, f)\leq \varepsilon_T\}\right) \vert \mathcal G_{0^-} \right]d\Pi(f) \\
&\leq o(1) + o(1) \int  \|f-f_0\|_1 d\Pi(f)=o(1), 
\end{split}
\end{equation*}
where the last inequality comes from the proof of Theorem \ref{th:L1}. Similarly, using the proof of Theorem \ref{th:d1},
\begin{equation*}
\begin{split}
\P_0 & \left( \E^\pi\left[\|f - f_0\|_1\1_{A_{\varepsilon_T}^c}  |N\right] > w_T \varepsilon_T \right) \leq \P_0\left( \Omega_{T}^c \right) + \P_0\left( D_T< e^{-c_1T\epsilon_T^2}\right) + \E_0[\1_{\Omega_T} \phi] \\
& \quad  +\frac{e^{c_1T\epsilon_T^2}}{w_T \varepsilon_T} \int_{ A_{L_1}(w_T \varepsilon_T)^c}  \|f-f_0\|_1\E_0\left[  \E_f\left[(1-\phi)\1_{ \Omega_{T}} \1_{\{ d_{1,T}(f_0, f)> \varepsilon_T\}}\right] \vert \mathcal G_{0^-} \right]d\Pi(f) \\
&\leq o(1) + o(1) \int  \|f-f_0\|_1 d\Pi(f), 
\end{split}
\end{equation*}
and $\P_0(\|\hat f - f_0\|_1 > 3w_T \varepsilon_T)=o(1)$. Since this is true for any $w_T\to+\infty$, this terminates the proof.

\subsection{Technical lemmas}\label{sec:lemmas}
\subsubsection{Control of the number of occurrences of the process on a fixed interval}
\begin{lemma}\label{control}
For any $M\geq 1$, for any $\alpha>0$, there exists a constant $C_\alpha$ only depending on $f_0$ such that for any $T>0$, the set
$$\tilde\Omega_T=\left\{\max_{\ell\in\{1,\ldots,K\}}\sup_{t\in [0,T]}N^{\ell}([t-A,t))\leq C_\alpha\log T\right\}$$
satisfies
$$\P_0(\tilde\Omega_T^c)\leq T^{-\alpha}$$
and for any $1\leq m\leq M$
$$\E_0\left[ \max_{\ell\in\{1,\ldots,K\}}\sup_{t\in [0,T]}\left(N^{\ell}([t-A,t))\right)^m\times 1_{\tilde\Omega_T^c}\right]\leq 2T^{-\alpha/2},$$
for $T$ large enough.
\end{lemma}
\begin{proof}
For the first part, we split the interval $[-A;T]$ into disjoint intervals of length A and we use Proposition 2 of \cite{HRR}.
For the second part, we set $$X:=\max_{\ell\in\{1,\ldots,K\}}\sup_{t\in [0,T]}\left(N^{\ell}([t-A,t))\right)\times 1_{\tilde\Omega_T^c}\geq 0$$
and the equality
\begin{eqnarray*}
\E_0[X^m]&=&\int_0^{+\infty}mx^{m-1}\P_0(X>x)dx\\
&=&\int_0^{C_\alpha\log T}mx^{m-1}\P_0\left(X>x\right)dx+\int_{C_\alpha\log T}^{+\infty}mx^{m-1}\P_0\left(X>x\right)dx\\
&\leq&m(C_\alpha\log T)^{m-1}\int_0^{C_\alpha\log T}\P_0(\tilde\Omega_T^c)dx+\int_{C_\alpha\log T}^{+\infty}mx^{m-1}\P_0\left(X>x\right)dx\\
&\leq&m(C_\alpha\log T)^m T^{-\alpha}+\int_{C_\alpha\log T}^{+\infty}mx^{m-1}\P_0\left(X>x\right)dx.
\end{eqnarray*}
Furthermore, for $T$ large enough,
\begin{eqnarray*}
\int_{C_\alpha\log T}^{+\infty}mx^{m-1}\P_0(X>x)dx
&\leq&\int_{C_\alpha\log T}^{+\infty}mx^{m-1}\P_0\left(\max_{\ell\in\{1,\ldots,K\}}\sup_{t\in [0,T]}\left(N^{\ell}([t-A,t))\right)>x\right)dx\\
&\leq&\int_{C_\alpha\log T}^{+\infty}mx^{m-1}\P_0\left(\max_{\ell\in\{1,\ldots,K\}}\sup_{t\in [0,e^{x/C_\alpha}]}\left(N^{\ell}([t-A,t))\right)>x\right)dx\\
&\leq&\int_{C_\alpha\log T}^{+\infty}mx^{m-1}\exp(-\alpha x/C_\alpha)dx\leq T^{-\alpha/2}.
\end{eqnarray*}
\end{proof}
\subsubsection{Control of $N[0,T]$}
Let $k\in\{1,\ldots,K\}$. We have the following result.
\begin{lemma} \label{lem:N0T}
For any $k\in\{1,\ldots,K\}$, for all $\alpha >0$ there exists $\delta_0>0$ such that 
$$\P_0\left( \left| \frac{N^k[0, T]}{T}- \mu_k^0\right| \geq\delta_0 \sqrt{\frac{(\log T)^3}{T}}  \right) = O(T^{-\alpha}).$$
\end{lemma}
\begin{proof}[Proof of Lemma \ref{lem:N0T}]
 We use Proposition~3 of \cite{HRR} and notations introduced for this result.
We denote $N[-A,0)$ the total number of points of $N$ in $[-A,0)$, all marks included. 
%We also denote $\mu_0=\E_0[N[-A,0)]$ and $\sigma_0^2=\var_0(N[-A,0))$, which are finite quantities (Proposition~2 of \cite{HRR}).
Let $\delta_T:= \delta_0 \sqrt{(\log T)^3/T} $, with $\delta_0$ a constant. We have:
\begin{small}
\begin{equation} \label{ineq:N0T:1}
\P_0\left( \left| \frac{N^k[0, T]}{T}- \mu_k^0\right|>\delta_T\right) \leq
 \P_0\left( \left|N^k[0, T]- \int_0^T\lambda_t^k(f_0)dt\right|>\frac{T\delta_T}{2}\right)+\P_0\left( \left| \int_0^T[\lambda_t^k(f_0)-\mu_k^0]dt\right|>\frac{T\delta_T}{2}\right)
\end{equation}
\end{small}
and we observe that
$$\lambda_t^k(f_0)=\nu_k^0+\int_{t-A}^{t^-}\sum_{\ell=1}^K h_{\ell,k}^{0}(t-s)dN^{\ell}_s=Z\circ\mathfrak{S}_t(N),$$
with $Z(N)=\lambda^k_0(f_0)$, where $\mathfrak{S}$ is the shift operator introduced in Proposition~3 of \cite{HRR}. We then have
$$Z(N)\leq b(1+N[-A,0))$$
with
$$b=\max_k\max\{\nu_k^0,\max_\ell \|h^0_{\ell,k}\|_\infty\}.$$
So, for any $\alpha>0$, the second term of  \eqref{ineq:N0T:1} is $O(T^{-\alpha})$ for $\delta_0$ large enough depending on $\alpha$ and $f_0$. The first term is controlled by using Inequality (7.7) of \cite{HRR} with $\tau=T$, $x=x_0T \delta_T^2$, $H_t= 1$,  $v=\mu_k^0 T+T\delta_T/2$ and
$$M_T=N^k[0, T]- \int_0^T\lambda_t^k(f_0)dt.$$
We take $x_0$ a positive constant such that
$\sqrt{8\mu_k^0x_0}<1,$
so that, for $T$ large enough
$$\frac{T\delta_T}{2}\geq \sqrt{2vx}+x/3.$$
Therefore, %  $\delta_0$ is a positive constant  large enough, 
we have
\begin{eqnarray*}
 \P_0\left( \left|M_T\right|>\frac{T\delta_T}{2}\right)&\leq& \P_0\left( \left|M_T\right|\geq \sqrt{2vx}+x/3 \ \mbox{and} \ \int_0^T\lambda_t^k(f_0)dt\leq v\right)+\P_0\left(\int_0^T\lambda_t^k(f_0)dt> v\right)\\
 &\leq& 2\exp(-x)+\P_0\left( \left| \int_0^T[\lambda_t^k(f_0)-\mu_k^0]dt\right|>\frac{T\delta_T}{2}\right)\\
  &\leq& 2\exp(-x_0\delta_0^2(\log T)^3)+O(T^{-\alpha})=O(T^{-\alpha}),
\end{eqnarray*}
which terminates the proof.
%
%\bigskip
%
%\textcolor{red}{The first term is controlled using Theorem~3  of \cite{HRR} with 
%$H = 1 $, $B=1$. Then there exists constants $\gamma_0, \gamma_1$ with  
%\begin{equation*}
%\begin{split}
%V &\leq  \gamma_0\int_0^T\lambda_t^k(f_0)dt + \gamma_1 T   = \gamma_0 \mu_k^0 T + \gamma_1T   + \gamma_0 \int_0^T[\lambda_t^k(f_0)-\mu_k^0]dt \\
%&\leq  \gamma_0 [\mu_k^0 +\delta_T] T +\gamma_1T \leq C_0 T 
%\end{split}
%\end{equation*}
%for some constant $C_0$, with probability going to 1. This leads to, with $x = T \delta_T^2$
%\begin{equation*}
%\begin{split}
% \P_0\left( \left|N^k[0, T]- \int_0^T\lambda_t^k(f_0)dt\right|>  \sqrt{3C_0} T\delta_T+ T\delta_T^2/3 \right) \leq C e^{-T \delta_T^2 } + O(T^{-\alpha})
%\end{split}
%\end{equation*}}
%. \textcolor{blue}{ Vincent : peux tu verifier mon inegalite ?}
\end{proof}
\subsubsection{Lemma on $\E_f[Z_{1,\ell}]$}
We have the following result which is useful to prove Theorem~\ref{th:L1}. 
\begin{lemma} \label{lem:EZm}
For for all $f \in \mathcal F_T$ such that $d_{1,T}(f,f_0)\leq \varepsilon_T$, there exists~$\ell$ (depending on $f$ and $f^0$) such that on $\Omega_T$,
$$\E_f[Z_{1, \ell}] \geq C\frac{T}{J_T}\|f-f^0\|_1,$$
where $C$ is a constant depending on $f^0$. 
\end{lemma}
\begin{proof}
By using the first bound of \eqref{ineq:nu-rho}, we observe that on $\Omega_T$, for any $\ell$, since $\inf_\ell\nu_\ell^0>0$, then $\inf_\ell\mu_\ell^0>0$ (by using \eqref{munu}) and we obtain that $\sum_{k=1}^K\rho_{k,\ell}$ and $\sum_{k=1}^K\nu_k$ are bounded. Therefore $\|f\|_1$ is bounded.
%%%%%%%%%%%%%%%%%%%%%%%%%%%%%%%%
On $\Omega_T$, since $\varepsilon_T\geq\delta_T$, still using \eqref{ineq:nu-rho}, for any $\ell$,
$$\nu_\ell+\sum_{k=1}^K\rho_{k,\ell}\mu_k^0-M\varepsilon_T\leq \nu_\ell^0+\sum_{k=1}^K\rho_{k,\ell}^0\mu_k^0\leq \nu_\ell+\sum_{k=1}^K\rho_{k,\ell}\mu_k^0+M\varepsilon_T$$
for $M$ a constant large enough. By using the formula $$\nu_\ell+\sum_{k=1}^K\rho_{k,\ell}\mu_k=\mu_\ell,\quad \nu_\ell^0+\sum_{k=1}^K\rho_{k,\ell}^0\mu_k^0=\mu_\ell^0,$$ we obtain
$$ \left|(\mu_\ell - \mu_\ell^0)-\sum_k \rho_{k , \ell} (\mu_k - \mu_k^0)\right| \leq  M\varepsilon_T,$$
which means that 
$$\|(I_d-\rho^T)(\mu-\mu^0)\|_\infty\leq M\varepsilon_T.$$
Therefore, since $\|\rho\|=\|\rho^T\|$ ($\rho\rho^T$ and $\rho^T\rho$ have the same eigenvalues),
\begin{align*}
\|\mu-\mu_0\|_2&=\|(I_d-\rho^T)^{-1}(I_d-\rho^T)(\mu-\mu_0)\|_2\\
&\leq(1-\|\rho\|)^{-1}\sqrt{K}\|(I_d-\rho^T)(\mu-\mu^0)\|_\infty\\
&\leq(1-\|\rho\|)^{-1}\sqrt{K}M\varepsilon_T.
\end{align*}
%Therefore, since \textcolor{red}{$\|\rho\|\leq ??$} $\|\mu\|_2$ is bounded and then $\|f\|_1$ is bounded.
%%%%%%%%%%%%%%%%%%%%%%%%%%%%%%%%%%%%
%$$\nu_\ell^0-\delta_T-\varepsilon_T\leq \nu_\ell+\sum_{k=1}^K(\mu_k^0-\mu_k)\rho_{k,\ell}+\sum_{k=1}^K\mu_k\rho_{k,\ell}\leq \nu_\ell^0+\delta_T+\varepsilon_T,$$
%with $\mu_\ell=\E_f[\lambda_t^{\ell}(f)]$.
%%%%%
%First note that by \eqref{ineq:nu-rho}, for all $\ell$, 
%$$ |(\mu_\ell - \mu_\ell^0)-\sum_k \rho_{k , \ell} (\mu_k - \mu_k^0)| \leq  M( \varepsilon_T\vee (\log T)^{3/2}/\sqrt{T} ) = M \delta_T,$$
%therefore $\|\mu -\mu_0 \| \leq  (1-\|\rho\|)^{-1}  M \delta_T \sqrt{K}$ 
%\textcolor{red}{Combining this with $\| \rho\| > 1 - b/\varepsilon_T$ together with $\delta_T \asymp \varepsilon_T$  implies that there exists a constant $C_0$ such that 
%$ \| \mu\| \leq \|\mu_0\| + M \sqrt{K}\delta_T/(b \varepsilon_T)\leq C_0$ and $\| f\|_1 \leq C_0$.}   
Since $f\in{\mathcal F}_T$, $1 - \|\rho\|\geq u_T \gtrsim  \varepsilon_T^{1/3} (\log T)^{1/6}$. Therefore, $\mu$ is bounded.
 As in \cite{HRR}, we denote $\mathbb Q_f$ a measure such that under $\mathbb Q_f$ the distribution of the
full point process restricted to $(-\infty, 0]$ is identical to the
distribution under $\P_f$ and such that on $(0, \infty)$ the process consists of
independent components each being a homogeneous Poisson process with
rate 1. Furthermore, the Poisson processes should be 
independent of the process on $(-\infty,0]$. From Corollary 5.1.2 in
\cite{Jacobsen:2006} the likelihood process is given by
$$\mathcal{L}_t(f) = \exp\left(Kt - \sum_{k=1}^K \int_0^t \lambda_u^k(f)
  du + \sum_{k=1}^K\int_0^t \log (\lambda_u^k(f))
 dN_u^k\right).$$
%and we have for $t \geq 0$ the relation
%\begin{equation} \label{eq:ChangeOfMeasure}
% \E_{\P} \kappa_t({\bf f})^2 =  \E_{\mathbb Q} \kappa_t({\bf f})^2 \mathcal{L}_t,
%\end{equation}
%where $\E_{\P}$ and $\E_{\mathbb Q}$ denote the expectation with respect to $\P$ and $\mathbb Q$  respectively. 
%\textcolor{red}{Assumptions:
%\begin{itemize}
%\item[-] For any $\ell'$, $\nu_{\ell'} \geq \nu_{\ell'}^0 /2 >0.$
%\item[-] $\|f\|_1\leq C_0$
%\item[-] $\sum_k\mu_k\leq C_0$.
%\end{itemize}
%}
Let $\tau>0$  satisfying
$$0<\frac{A\tau K^2}{1-\tau K}<\frac12 \quad \mbox{ and } \quad \tau \leq \frac{ \min_{\ell'} \nu_{\ell'}^0}{ 2 C_0' },$$ %\left(1 \wedge \textcolor{red}{\min_\ell \nu_\ell^0/C_0'}\right).$$
with $C'_0$ an upper bound of $\|f-f_0\|_1$.
\begin{itemize}
\item Assume that for any $\ell'$, $\left| \nu_{\ell'} - \nu_{\ell'}^0\right| <\tau \|f - f_0\|_1 $. Then, for any $\ell'$,
$$\left| \nu_{\ell'} - \nu_{\ell'}^0\right| <\tau \|f - f_0\|_1 =\tau\left(\sum_k\left| \nu_k - \nu_k^0\right|+\sum_{k,\ell}\|h_{k,\ell}-h_{k,\ell}^0\|_1\right)
$$
and
$$\left| \nu_{\ell'} - \nu_{\ell'}^0\right| \leq\sum_{\ell}\left| \nu_{\ell} - \nu_{\ell}^0\right| 
<\frac{\tau K}{1-\tau K}\sum_{k,\ell}\|h_{k,\ell}-h_{k,\ell}^0\|_1.
$$
Let $\ell$ such that
$$\sum_k\|h_{k,\ell}-h_{k,\ell}^0\|_1=\max_{\ell'}\left\{\sum_k\|h_{k,\ell'}-h_{k,\ell'}^0\|_1\right\}.$$
Then, for any $\ell'$,
\begin{equation}\label{C1}
\left| \nu_{\ell'} - \nu_{\ell'}^0\right| <\frac{\tau K^2}{1-\tau K}\sum_k\|h_{k,\ell}-h_{k,\ell}^0\|_1,
\end{equation}
and
\begin{eqnarray}\label{C2}
\|f-f^0\|_1&=&\sum_{\ell'}\left| \nu_{\ell'} - \nu_{\ell'}^0\right|+\sum_{\ell'}\sum_k\|h_{k,\ell'}-h_{k,\ell'}^0\|_1\nonumber\\
&\leq&\left(\frac{\tau K^2}{1-\tau K}+K\right)\sum_k\|h_{k,\ell}-h_{k,\ell}^0\|_1.
\end{eqnarray}
We denote
$$\Omega_k=\left\{\max_{k'\neq k} N^{k'}[0,A]=0,\quad N^{k}[0,A]=1,\quad N^{k'}[-A,0]\leq aA \mu_{k'} \,\forall k'\right\},$$
where $a$ is a fixed constant chosen later.
We then have
\begin{eqnarray*}
\E_f[Z_{m, \ell}] &=&\frac{ T}{ 2J_T}\E_f\left[\left| \nu_\ell - \nu_\ell^0 + \sum_{k=1}^K\int_{0}^{A^-} (h_{k,\ell} - h_{k,\ell}^0 )(A-s) dN^{k}_{s} \right|\right]\\
&\geq&\frac{ T}{ 2J_T} \sum_k \E_f\left[\1_{\max_{k'\neq k} N^{k'}[0,A]=0}\1_{ N^{k}[0,A]=1}  \left| \nu_\ell - \nu_\ell^0 + \int_{0}^{A^-} (h_{k,\ell} - h_{k,\ell}^0 )(A-s) dN^{k}_{s} \right|\right]\\
&\geq& \frac{ T}{ 2J_T} \sum_k \E_{{\mathbb Q}_f}\left[{\mathcal L}_A(f) \1_{\max_{k'\neq k} N^{k'}[0,A]=0}\1_{ N^{k}[0,A]=1}\left| \nu_\ell - \nu_\ell^0 + \int_{0}^{A^-} (h_{k,\ell} - h_{k,\ell}^0 )(A-s) dN^{k}_{s} \right|\right]\\
&\geq& \frac{ T}{ 2J_T} \sum_k\E_{{\mathbb Q}_f}\left[{\mathcal L}_A(f)\1_{\Omega_k}\left| \nu_\ell - \nu_\ell^0 + \int_{0}^{A^-} (h_{k,\ell} - h_{k,\ell}^0 )(A-s) dN^{k}_{s} \right|\right].
\end{eqnarray*}
Note that on $\Omega_k$,
\begin{eqnarray*}
{\mathcal L}_A(f)&:=&\exp\left(KA-\sum_{k'}\int_0^A\lambda_t^{k'}(f)dt+\sum_{k'}\int_0^A\log (\lambda_t^{k'}(f))dN_t^{k'}\right)\\
&\geq&\nu_k\exp(KA)\exp\left(-\sum_{k'}\int_0^A\lambda_t^{k'}(f)dt\right)\\
&\geq&\nu_k\exp(KA)\exp\left(-\sum_{k'}\int_0^A\left(\nu_{k'}+\int_{t-A}^{t-}\sum_{k''}h_{k''k'}(t-u)dN^{k''}_u\right)dt\right)\\
&\geq&\nu_k\exp\left(KA-A\sum_{k'}\nu_{k'}\right)\exp\left(-\int_{-A}^{A^-}\sum_{k',k''}\rho_{k''k'}dN^{k''}_u\right)\\
&\geq&\nu_k\exp\left(KA-A\sum_{k'}\nu_{k'}\right)\exp\left(-aA\sum_{k''}\mu_{k''}\sum_{k'}\rho_{k''k'}-\sum_{k'}\rho_{kk'}\right).
\end{eqnarray*}
Since on $\mathcal F_T$,
$$\nu_k\exp\left(KA-A\sum_{k'}\nu_{k'}\right)\exp\left(-aA\sum_{k''}\mu_{k''}\sum_{k'}\rho_{k''k'}-\sum_{k'}\rho_{kk'}\right) \geq \nu_k e^{- Ka A C_1} \geq \nu_k^0e^{- Ka A C_1} /2  \geq C(f_0),$$
where $C_1$ and $C(f_0)$  are some constants, we have, by definition of ${\mathbb Q}_f$,
\begin{eqnarray*}
I_k&:=&\E_{{\mathbb Q}_f}\left[{\mathcal L}_A(f)\1_{\Omega_k}\left| \nu_\ell - \nu_\ell^0 + \int_{0}^{A^-} (h_{k,\ell} - h_{k,\ell}^0 )(A-s) dN^{k}_{s} \right|\right]\\
%&\geq&C(f_0)\E_{{\mathbb Q}_f}\left[\1_{\max_{k'\neq k} N^{k'}[0,A]=0}\1_{ N^{k}[0,A]=1}  \left| \nu_\ell - \nu_\ell^0 + \int_{0}^{A^-} (h_{k,\ell} - h_{k,\ell}^0 )(A-s) dN^{k}_{s} \right|\1_{\{N^{k'}[-A,0]\leq aA \mu_{k'} \, \forall k'\}}\right]\\
&\geq&C(f_0)\E_{{\mathbb Q}_f}\left[\1_{ N^{k}[0,A]=1}  \left| \nu_\ell - \nu_\ell^0 + \int_{0}^{A^-} (h_{k,\ell} - h_{k,\ell}^0 )(A-s) dN^{k}_{s} \right|\right]\\
&&\times{\mathbb Q}_f(N^{k'}[-A,0]\leq aA \mu_{k'}\, \forall k')\times{\mathbb Q}_f(\max_{k'\neq k} N^{k'}[0,A]=0).
\end{eqnarray*}
Under ${\mathbb Q}_f$, $N^{k}[0,A]\sim \mbox{Poisson}(A)$. If $U\sim Unif([0,A])$,
\begin{eqnarray*}
\E_{{\mathbb Q}_f}\left[\1_{ N^{k}[0,A]=1}  \left|  \int_{0}^{A^-} (h_{k,\ell} - h_{k,\ell}^0 )(A-s) dN^{k}_{s} \right|\right]&=&\E\left[  \left|   (h_{k,\ell} - h_{k,\ell}^0 )(A-U) \right|\right]{\mathbb Q}_f(N^{k}[0,A]=1)\\
&=&\frac{1}{A}\int_0^A \left|   (h_{k,\ell} - h_{k,\ell}^0 )(A-s) \right|ds\times Ae^{-A}\\
&=&e^{-A} \|h_{k,\ell}-h_{k,\ell}^0\|_1.
\end{eqnarray*}
We also have, using \eqref{C1},
\begin{eqnarray*}
\E_{{\mathbb Q}_f}\left[\1_{ N^{k}[0,A]=1}  \left|  \nu_\ell - \nu_\ell^0\right|\right]&=&Ae^{-A} \left|  \nu_\ell - \nu_\ell^0\right| \leq Ae^{-A} \frac{\tau K^2}{1-\tau K}\sum_k\|h_{k,\ell}-h_{k,\ell}^0\|_1.
\end{eqnarray*}
Furthermore,
$${\mathbb Q}_f(\max_{k'\neq k} N^{k'}[0,A]=0)=\exp(-(K-1)A),$$
and
\begin{eqnarray*}
{\mathbb Q}_f(N^{k'}[-A,0]\leq a A\mu_{k'} \, \forall k')&\geq& 1 -\sum_{k'} {\mathbb Q}_f\left( N^{k'}[-A,0]>  aA \mu_{k'} \right)\\ &\geq& 1 - \sum_{k'} \frac{ \mu_{k'} A}{ aA \mu_{k'}} = 1 - \frac{K}{ a }=\frac12,
\end{eqnarray*}
with $a=2K$. Finally,
\begin{eqnarray*}
I_k&\geq&\frac12C(f_0) \exp(-KA)\left(1-\frac{A\tau K^2}{1-\tau K}\right)\|h_{k,\ell}-h_{k,\ell}^0\|_1
%&\geq&\frac{\nu_k}{2}\exp\left(-A\sum_{k'}\nu_{k'}-aA\sum_{k''}\mu_{k''}\sum_{k'}\rho_{k''k'}-\sum_{k'}\rho_{kk'}\right)\left(1-\frac{A\tau K^2}{1-\tau K}\right)\|h_{k,\ell}-h_{k,\ell}^0\|_1
\end{eqnarray*}
and using \eqref{C2},
\begin{eqnarray*}
\E_f[Z_{m, \ell}] &\geq&\frac{ T}{ 2J_T} \sum_k I_k\\
&\geq&\frac{ T}{ 2J_T} \frac12C(f_0) \exp(-KA)\left(1-\frac{A\tau K^2}{1-\tau K}\right)\sum_k\|h_{k,\ell}-h_{k,\ell}^0\|_1\\
&\geq& C\frac{T}{J_T}\|f-f^0\|_1,
\end{eqnarray*}
where $C$ depends on $f_0$.
\item We now assume that there exists $\ell$ such that
$$\left| \nu_\ell - \nu_\ell^0\right| \geq \tau \|f - f_0\|_1.$$  In this case, using similar arguments, still with $a=2K$,
\begin{equation*}
\begin{split}
\E_f[Z_{m, \ell}]&\geq\frac{T}{2J_T} \P_f[\{\max_k N^{k}[0,A]=0\}]\left| \nu_\ell - \nu_\ell^0\right| \\
 &\geq  \frac{\tau  T}{2J_T}\|f-f_0\|_1 \E_{{\mathbb Q}_f}\left[{\mathcal L}_A(f)\1_{\{\max_k N^{k}[0,A]=0\}}\right]\\ &\geq  \frac{\tau  T}{2J_T}\|f-f_0\|_1\E_{{\mathbb Q}_f}\left[{\mathcal L}_A(f)\1_{\{\max_k N^{k}[0,A]=0\}} \1_{\{N^k[-A,0]\leq aA\mu_k,\, \forall k\}}\right]\\
&\geq \frac{\tau  T}{2J_T}\|f-f_0\|_1\exp\left(KA-A\sum_{k'}\nu_{k'}-aA\sum_{k''}\mu_{k''}\sum_{k'}\rho_{k''k'}\right)\E_{{\mathbb Q}_f}\left[ \1_{\{N[0,A]=0\}} \1_{\{\forall k\, N^k[-A,0]\leq aA \mu_k\}}\right]\\
&\geq \frac{\tau  T}{4J_T}\|f-f_0\|_1\exp\left(-A\sum_{k'}\nu_{k'}-aA\sum_{k'}(\mu_{k'}-\nu_{k'})\right) \geq C \frac{T}{J_T}\|f-f_0\|_1
\end{split}
\end{equation*}
for $C$ depending on $f_0$.  Lemma \ref{lem:EZm} is proved.
\end{itemize}
\end{proof}
\subsubsection{Upper bound for the Laplace transform of the number of points in a cluster} 
In the next lemma, we refine the proof of Lemma 1 of \cite{HRR}. Given an ancestor of type $\ell$, we denote $W^{\ell}$ the number of points in its cluster. We have the following result.
\begin{lemma} \label{lem:Laplacetransform}
Assume $\|\rho\| < 1$ and consider $t$ such that $0\leq t \leq\frac{1 - \|\rho\|}{2\sqrt{K}} \log \left(\frac{1+\|\rho\|}{2\|\rho\|} \right)$.
Then, we have for any $\ell\in\{1,\ldots,K\}$,
$$\E_f[\exp(t W^{\ell})]\leq \frac{1+\|\rho\|}{2\|\rho\|}.$$
Moreover, if $\|\rho\|\leq 1/2$, then there exist two absolute constants $c_0$ and $C_0$ such that if $\sqrt{K}t\leq c_0$, then $\E_f[\exp(t W^{\ell})]\leq C_0$.
%Assume $\|\rho\| < 1$ and consider $t$ such that $t <\frac{1 - \|\rho\|}{1+\|\rho\|} \log \left(\frac{1+\|\rho\|}{2\|\rho\|} \right)/\sqrt{K}$.
%Then, we have for any $\ell\in\{1,\ldots,K\}$,
%$$\E_f[\exp(t W^{\ell})]\leq \exp\left( \frac{ 2t \sqrt{K} }{1 -  \|\rho\|}\right).$$
%In particular, with $t$ such that $t \sqrt{K} \leq\frac{(1-\|\rho\|)^2}{2} \log \left(\frac{1+\|\rho\|}{2\|\rho\|} \right)$, $\E_f[\exp(t W^{\ell})]$ is bounded by a constant only depending on $K$.
Finally,
\begin{equation*}
\E_f[W^{\ell}]= \1^T(I - \rho^T)^{-1}\bf e_\ell.
\end{equation*}
\end{lemma}
\begin{proof}[Proof of Lemma \ref{lem:Laplacetransform}]
 We introduce $K^\ell(n)\in\R^K$ the vector of the number of descendants of the $n$th generation from a single ancestral point of type $\ell$, with $K^\ell(0)={\bf e}_\ell$, where $({\bf e}_\ell)_k=\1_{k=\ell}$. More precisely, $(K^\ell( n))_k$ is the number of descendants of the $n$th generation and of the type $k$ from a single ancestral point of type $\ell$.
Then,
$$W^{\ell}=\1^T\times\sum_{n=0}^\infty K^\ell(n).$$
We now set for any $\theta\in\R^K$,
$$\phi_\ell(\theta)=\log\left(\E_f[\exp(\theta^TK^\ell(1))]\right)$$
and 
$$\phi(\theta)=(\phi_1(\theta),\ldots,\phi_K(\theta))^T.$$
Note that $$K^\ell(1)_j \sim{\mathcal P}\left(\rho_{\ell,j}\right), \quad \forall j \leq K$$
and
$$\phi_\ell(\theta)=\sum_{j=1}^K \log\left(\E_f[\exp(\theta_j K^\ell(1)_j)]\right)
=\sum_{j=1}^K\rho_{\ell,j}(\exp(\theta_j)-1).$$
Therefore,
$$(D\phi(\theta))_{\ell,j}:=\frac{\partial \phi_\ell(\theta)}{\partial\theta_j}=\rho_{\ell,j}\exp(\theta_j)$$
and for any $x\in\R^K$, since $\|\rho\|:=\sup_{x, \|x\|_2=1}\|\rho x\|_2$,
\begin{eqnarray*}
\|D\phi(\theta)x\|_2^2&=&\sum_{\ell=1}^K\left(\sum_{j=1}^K\rho_{\ell,j}\exp(\theta_j)x_j\right)^2\\
&=&\sum_j\sum_{j'}(\rho^T\rho)_{j,j'}\exp(\theta_j)x_j\exp(\theta_{j'})x_{j'}\\
&=&v^T\rho^T\rho v\\
&\leq& \|\rho\|^2\|v\|^2_2=\|\rho\|^2\sum_{j=1}^Kx_j^2\exp(2\theta_j)
\end{eqnarray*}
with $v$ the vector of $\R^K$ such that $v_j=\exp(\theta_j)x_j$.
 So,
\begin{eqnarray*}
|\|D\phi(\theta)\||&\leq&\|\rho\|\max_{j}\exp(|\theta_j|) \leq \|\rho\|e^{\|\theta\|_2}.
\end{eqnarray*}
So, by applying the mean value theorem,
$$\|\phi(\theta)\|_2=\|\phi(\theta)-\phi(0)\|_2\leq \|\rho\|e^{\|\theta\|_2}\|\theta\|_2.$$
We use a modification of the arguments in the proof of Lemma 1 of \cite{HRR}. Writing $g_1(\theta) = \theta + \phi(\theta)$, we have for $n\geq 3$:
\begin{align*}
\E_f\left[e^{\theta^T (\sum_{k=0}^n K^\ell(k))} \right] &= \E_f\left[e^{\theta^T (\sum_{k=0}^{n-1} K^\ell(k))}\E_f\left[e^{\theta^TK^\ell(n)} |K^\ell(n-1),\ldots, K^\ell(1) \right]\right]\\
&=
\E_f\left[ e^{\theta^T ( \sum_{k=0}^{n-2} K^\ell(k) )} e^{(\theta+ \phi(\theta))^T K^\ell(n-1) } \right]= \E_f\left[ e^{\theta^T ( \sum_{k=0}^{n-2} K^\ell(k) )} e^{g_1(\theta)^T K^\ell(n-1) } \right]\\
& = \E_f\left[e^{\theta^T ( \sum_{k=0}^{n-3} K^\ell(k) )} e^{(\theta+ \phi(g_1(\theta)))^T K^\ell(n-2) } \right] =  \E_f\left[e^{\theta^T ( \sum_{k=0}^{n-3} K^\ell(k) )} e^{g_2(\theta)^T K^\ell(n-2) } \right] \\
&= \E_f\left[e^{\theta^TK^\ell(0)} e^{g_{n-1}(\theta)^T K^\ell(1) } \right] = e^{(g_n(\theta)_\ell)}, 
\end{align*}
with the induction formula: $g_n(\theta) = \theta + \phi(g_{n-1}(\theta))$ for $n\geq 2$. In particular, 
$$\|g_1(\theta)\|_2 \leq \|\theta\|_2 ( 1 + \|\rho\| e^{\|\theta\|_2})\quad\mbox{and}\quad \|g_n(\theta)\|_2 \leq \|\theta \|_2+ \|\rho\| e^{\|g_{n-1}(\theta)\|_2}\|g_{n-1}(\theta)\|_2.$$ We now set $C:= (1 + \|\rho\|)/(1-\|\rho\|)>1$. Then, 
 if $\|g_{n-1}(\theta) \|_2\leq \|\theta\|_2 ( 1 + C) $,
$$\|g_n(\theta)\|_2 \leq \|\theta\|_2 ( 1 +\| \rho\|(1+C)e^{\|\theta\|_2 (1+C) } ) \leq \|\theta\|_2 ( 1+ C)$$
as soon as 
\begin{equation}\label{rayon}
\|\theta\|_2 \leq (1+C)^{-1} \log (C/(\|\rho\| ( 1+ C)))=\frac{1 - \|\rho\|}{2} \log \left(\frac{1+\|\rho\|}{2\|\rho\|} \right).
\end{equation}
Since $\|\rho\|<1$, the previous upper bound is positive. Note that under \eqref{rayon}, $\|\theta\|_2 \leq  \log (C/\|\rho\|)$, and
$$
\|g_1(\theta)\|_2 \leq \|\theta\|_2 ( 1 + \|\rho\| e^{\|\theta\|_2})
\leq  \|\theta\|_2 ( 1 + \|\rho\| e^{\log(C/\|\rho\|)})\leq \|\theta\|_2(1+C).$$
We finally obtain that under \eqref{rayon},
$$\|g_n(\theta)\|_2 \leq \|\theta\|_2 ( 1+ C), \quad \forall\, n \geq 1.$$
Since for any $m$, $n\mapsto \sum_{k=0}^n (K^\ell(k))_m$ is increasing and $W^{\ell}=\1^T\times\sum_{n=0}^\infty K^\ell(n),$ we have by
monotone convergence that for $t>0$,
$$\E_f[\exp(t W^{\ell})]= \lim_{n \to \infty}\exp(g_n(t \mathbf{1})_{\ell}).$$ 
By the previous result, the right hand side is bounded if
$t$ is  small enough. More precisely, for all $0<t \leq (1+C)^{-1} \log (C/(\|\rho\| ( 1+ C)))/\sqrt{K}$,
$$\E_f[\exp(t W^{\ell})]\leq \exp(t\sqrt{K}(1+C))\leq \frac{C}{\|\rho\|(1+C)}=\frac{1+\|\rho\|}{2\|\rho\|}.$$
The second point is obvious in view of previous computations.
Moreover, since $\E_f[W^{\ell}] = \sum_{n=0}^\infty\E_f[\1^TK^\ell(n)] $ and since for any $v\in \mathbb R^K$
$$\E_f[v^TK^\ell(n) | K^\ell(0), \ldots , K^\ell(n-1)]= \sum_{j=1}^K\sum_{k=1}^K K^{\ell}(n-1)_{j}v_k\rho_{j,k} = v^T\rho^T K^\ell(n-1).$$
We obtain by induction that $\E_f[\1^TK^\ell (n)]  =\1^T (\rho^T)^n \bf e_\ell$ and taking the limit, since $\|\rho\| <1$, 
\begin{equation*}
\E_f[W^{\ell}]= \1^T(I - \rho^T)^{-1}\bf e_\ell.
\end{equation*}
\end{proof}
%%%
\subsubsection{Lemma on $\bar N^m$}
\begin{lemma} \label{lem:barNm}
There exists $\tilde c_0 $ such that  for all $c_0>0$  such that for $T$ large enough,
 $$ \mathbb P_0\left( \sum_{m=1}^{J_T-1} \bar N^m(I_m)> c_0 T \right) \leq e^{-\tilde c_0 c_0 T} .$$
 Furthermore, there exists a constant $\kappa_0>0$ (see the definition of $J_T$) such that
 $$
\sum_{m=1}^{J_T-1}\E_f[\bar N^{m}(I_m)]=o(T).$$
\end{lemma}
\begin{proof}[Proof of Lemma \ref{lem:barNm}]
We use computations of  the proof of Proposition 2 of \citet{HRR}.
To bound $\bar N^m(I_m)$, first observe that  we only consider points of $N$ whose ancestors are born before $(2m-1)T/(2J_T)$, i.e. the distance between the occurrence of an ancestor and $I_m$ is at least $2mT/(2J_T)-A-(2m-1)T/(2J_T)=T/(2J_T)-A$ since $$I_m=\left[\frac{2mT}{2J_T}-A, \frac{(2m+1)T}{2J_T}\right].$$
Using the cluster representations of the process, for any $p\in{\mathbb Z}$ and for any $\ell\in\{1,\ldots, K\}$, we consider $B_{p,\ell}$ the number of ancestors of type $\ell$ born in the interval $[p, p+1]$. The $B_{p,\ell}$'s are iid Poisson random variables with parameter $\nu_{\ell}$. We have
\begin{small}
$$\sum_{m=1}^{J_T-1}\bar N^m(I_m)\leq \sum_{\ell=1}^K\sum_{p\in {\mathcal J}_T^+}\sum_{k=1}^{B_{p,\ell}}\left(
W_{p,k}^\ell-\frac1A\left(\frac{T}{2J_T}-A\right)\right)_++\sum_{\ell=1}^K\sum_{p=-\infty}^0\sum_{k=1}^{B_{p,\ell}}\left(
W_{p,k}^\ell-\frac1A\left(-p-1+\frac{T}{J_T}-A\right)\right)_+,$$
\end{small}
where $W_{p,k}^\ell$ is the number of points in the cluster generated by the ancestor $k$ which is of type $\ell$ and
$${\mathcal J}_T^+=\{p: \ 1\leq p\leq T-T/(2J_T)\}$$ since 
$$\bigcup_{m=1}^{J_T-1} I_m\subset \left[\frac{T}{J_T}-A, T-\frac{T}{2J_T}\right].$$ 
For the first term of the previous right hand side, we have used same arguments as \citet{HRR} and the lower bound of the distance determined previously. For the second term of the right hand side, since $p\leq 0$, this lower bound is at least $-p-1+\frac{T}{J_T}-A$. Conditioned on the $B_{p,\ell}$'s, the variables $(W_{p,k}^{\ell})_{k}$ are iid with same distribution as $W^{\ell}$ introduced in Lemma~\ref{lem:Laplacetransform}. Furthermore, by Lemma~\ref{lem:Laplacetransform} applied with $f=f_0$, since $\|\rho_0\|<1$, we know that for $t_0>0$ small enough (only depending on $\|\rho_0\|$ and $K$), 
$$\E_0[\exp(t_0 W^{\ell})]\leq C_0,$$ where $C_0$ is a constant. So, for any $c>0$,
\begin{align*}
{\mathcal P}_{T,1}&:=\P_0\left(\sum_{\ell=1}^K\sum_{p\in {\mathcal J}_T^+}\sum_{k=1}^{B_{p,\ell}}\left(
W_{p,k}^\ell-\frac1A\left(\frac{T}{2J_T}-A\right)\right)_+\geq cT\right)\\
&\leq\exp(-t_0  c T)\prod_{\ell=1}^K\prod_{p\in {\mathcal J}_T^+}\E_0\left[\prod_{k=1}^{B_{p,\ell}}\E_0\left[\exp\left(t_0 \left(
W_{p,k}^\ell-\frac{T}{2AJ_T}+1\right)_+\right)| B_{p,\ell}\right]\right]\\
&\leq\exp(-t_0  c T)\prod_{\ell=1}^K\prod_{p\in {\mathcal J}_T^+}\E_0\left[(H_{\ell}(t_0 ))^{ B_{p,\ell}}\right]=\exp\Big(-t_0  c T+\sum_{\ell=1}^K\sum_{p\in {\mathcal J}_T^+}\nu_{\ell}^0(H_{\ell}(t_0)-1)\Big),
\end{align*}
where
$$H_{\ell}(t_0 ):=\E_0\left[\exp\left(t_0 \left(
W^\ell-\frac{T}{2AJ_T}+1\right)_+\right)\right],$$
satisfying
\begin{align*}
H_{\ell}(t_0 )&\leq \P_0\left(W^\ell\leq \frac{T}{2AJ_T}-1\right)+\exp(t_0 -Tt_0 /(2AJ_T))\E_0\left[\exp\left(t_0 
W^\ell\right)\right]\\
&\leq 1+C_0\exp(t_0 -Tt_0 /(2AJ_T)).
\end{align*}
Therefore,
$$\sum_{\ell=1}^K\sum_{p\in {\mathcal J}_T^+}\nu_{\ell}^0(H_{\ell}(t_0 )-1)\lesssim  (T-T/(2J_T)\exp(-Tt_0 /(2AJ_T)) \lesssim e^{ - C' \kappa_0 \log T} =o(t_0  cT)$$  
by choosing $\kappa_0$ large enough and then
$${\mathcal P}_{T,1}\lesssim \exp(-t_0  c T/2).$$
Similarly,
\begin{align*}
{\mathcal P}_{T,2}&:=\P_0\left(\sum_{\ell=1}^K\sum_{p=-\infty}^0\sum_{k=1}^{B_{p,\ell}}\left(
W_{p,k}^\ell-\frac1A\left(-p-1+\frac{T}{J_T}-A\right)\right)_+\geq cT\right)\\
%&\leq\exp(-t  c T)\prod_{\ell=1}^K\prod_{p\in {\mathcal J}_T^+}\E_f\left[\prod_{k=1}^{B_{p,\ell}}\E_f\left[\exp\left(t \left(
%W_{p,k}^\ell-\frac{T}{AJ_T}+1+\frac{1}{A}+\frac{p}{A}\right)_+\right)| B_{p,\ell}\right]\right]\\
%&\leq\exp(-t  c T)\prod_{\ell=1}^K\prod_{p\in {\mathcal J}_T^+}\E_f\left[(H_{\ell}(t ))^{ B_{p,\ell}}\right]=
&\leq\exp\Big(-t_0  c T+\sum_{\ell=1}^K\sum_{p=-\infty}^0\nu_{\ell}^0(\tilde H_{\ell,p}(t_0 )-1)\Big),
\end{align*}
where
$$\tilde H_{\ell,p}(t_0 ):=\E_0\left[\exp\left(t_0 \left(
W^\ell-\frac{T}{AJ_T}+1+\frac{1}{A}+\frac{p}{A}\right)_+\right)\right],$$
satisfying
$$
\tilde H_{\ell,p}(t_0 )\leq 1+C_0\exp(t_0 +t_0 /A-Tt_0 /(AJ_T)+t_0  p/A).
$$
Therefore,
$$\sum_{\ell=1}^K\sum_{p=-\infty}^0\nu_{\ell}^0(\tilde H_{\ell,p}(t_0 )-1)\lesssim \exp(-Tt_0 /(AJ_T))=o(t_0  cT)$$
and then
$${\mathcal P}_{T,2}\lesssim \exp(-t_0  c T/2).$$
Finally, there exists $\tilde c_0$ (only depending on $t_0$, so only depending on $\|\rho_0\|$ and $K$) such that  for all $c_0>0$  such that for $T$ large enough
 $$ \mathbb P_f\left( \sum_{m=0}^{J_T-1} \bar N^m(I_m)> c_0 T \right) \leq e^{-\tilde c_0 c_0 T} $$
and the first part of the lemma is proved. 

\bigskip

For the second part, we only consider the case $1/2\leq \|\rho\|<1$. The case $ \|\rho\|<1/2$ can be derived easily using following computations. We have:
$$
\sum_{m=1}^{J_T-1}\E_f[\bar N^{m}(I_m)] ={\mathcal E}_{T,1}+{\mathcal E}_{T,2},
$$
with
$$
{\mathcal E}_{T,1}:=\E_f\left[\sum_{\ell=1}^K\sum_{p\in {\mathcal J}_T^+}\sum_{k=1}^{B_{p,\ell}}\left(
W_{p,k}^\ell-\frac1A\left(\frac{T}{2J_T}-A\right)\right)_+\right]$$
and, with $t =\frac{1 - \|\rho\|}{2\sqrt{K}} \log \left(\frac{1+\|\rho\|}{2\|\rho\|} \right)\gtrsim(1-\|\rho\|)^2\gtrsim u_T^2$ on ${\mathcal F}_T$, since for $x>0$, $x\leq e^x$, by using Lemma~\ref{lem:Laplacetransform},
\begin{align*}
{\mathcal E}_{T,2}&:=\E_f\left[\sum_{\ell=1}^K\sum_{p=-\infty}^0\sum_{k=1}^{B_{p,\ell}}\left(
W_{p,k}^\ell-\frac1A\left(-p-1+\frac{T}{J_T}-A\right)\right)_+\right]\\
&= \sum_{\ell=1}^K\nu_\ell\sum_{p=-\infty}^0\E_f\left[\left(
W_{p,k}^\ell-\frac1A\left(-p-1+\frac{T}{J_T}-A\right)\right)_+\right]\\
&=t^{-1}\sum_{\ell=1}^K\nu_\ell\sum_{p=-\infty}^0\E_f\left[e^{t\left(
W_{p,k}^\ell-\frac1A\left(-p-1+\frac{T}{J_T}-A\right)\right)}\right]\\
&\lesssim t^{-1}e^{-\frac{tT}{AJ_T}}(1-e^{-\frac{t}{A}})^{-1}\E_f\left[e^{tW^\ell}\right]\sum_{\ell=1}^K\nu_\ell\\
&\lesssim (1 - \|\rho\|)^{-4}e^{-\frac{(1 - \|\rho\|)^2T}{AJ_T}}\sum_{\ell=1}^K\nu_\ell \lesssim e^{- \kappa_0^{-1} C'' \log T }(\log T)^{-2/3}\varepsilon_T^{-4/3},
\end{align*}
for $C''$ depending on $A$ and $K$.   Similarly,
$${\mathcal E}_{T,1}\lesssim T(1 - \|\rho\|)^{-2}e^{-\frac{(1 - \|\rho\|)^2T}{2AJ_T}}\sum_{\ell=1}^K\nu_\ell.$$
%\textcolor{red}{under \eqref{uT}, $T/J_T\sim_{T\to+\infty}\kappa_0^{-1}u_T^{-2}\log T$ when $T\to+\infty$, therefore for $\kappa_0$ small enough,
Choosing $\kappa_0$ small enough, 
$$\sum_{m=1}^{J_T-1}\E_f[\bar N^{m}(I_m)]=o(T).$$

\end{proof}
%%%%%%%%%%%%%%%%%%%%%%%%%%%%%%%%
\subsection{Proofs of results of Section~\ref{sec:prior-model}}\label{sec:Proof:bayes}
This section is devoted to the proofs of results of Section~\ref{sec:prior-model}.
\subsubsection{Proof of Corollary~\ref{cor:randparthist}}
The main difference with the case of the regular partition is the control of the $\L_1$-entropy. This is more complicated than the regular grid histogram prior and we apply instead Theorem \ref{th:slices}.  Because of the equivalence between the parameterization in $t$ or in $u$, we sometimes $\bar h_{w,t,J}$ as $\bar h_{w,u,J}$. Let $J $ and $(\underline{w},\underline{u})$ and $(\underline{w'},\underline{u'})$ belonging to $\mathcal S_J^2$. 
Then, for all $\zeta>0$, if $\delta = \delta' = 1$, $|t_j' - t_j | \leq \zeta\epsilon_T \min( |t_j - t_{j-1}| , |t_{j} - t_{j+1}|)$ for all $j$ and $\sum_j |w_j - w_j'|\leq \epsilon_T$ then 
\begin{equation*}
\begin{split}
\|\bar h_{w,t,J} -\bar h_{w',t',J} \|_1& \leq \|\bar h_{w,t,J} -\bar h_{w',t,J} \|_1 + \|\bar h_{w',t,J} -\bar h_{w',t'J} \|_1\\
& \leq \sum_{j=1}^J  |w_j-w_j'| + 4\sum_{j=1}^J\zeta\epsilon_T w_j'. % \leq 4  \ell((\underline{w}, \underline{u}),(\underline{w'}, \underline{u'}))
\end{split}
\end{equation*}
Consider  $e_T >0 $ and $\mathcal U_{J,T} = \{ \underline u \in \mathcal S_J, \min_j u_j\geq e_T \}$, under the Dirichlet prior on $\underline u$ 
 \begin{equation*}
 \begin{split}
 \Pi_u(\mathcal U_{J,T}^c | J) &\leq   \sum_{j\leq J} \Pi( u_j \leq e_T )  = \sum_{j=1}^J  Prob\left( \mbox{ Beta}(\alpha, (J-1)\alpha) \leq e_T \right) \lesssim Je_T^{\alpha}\leq e^{-c T \epsilon_T^2} 
 \end{split}
 \end{equation*}
  if $\log e_T \leq -(c/\alpha+1)T \epsilon_T^2$ if $J\leq J_1 (T/\log T)^{1/(2\beta+1)}=:J_{1,T}$. We define $\mathcal F_{1,T} = \{ \bar h_{w,u,J}, J\leq J_{1,T} ; \underline u \in \mathcal U_{J,T}\}$. 
%%%%
To apply Theorem \ref{th:slices}, we need to construct the slices $\mathcal H_{T,i} $ of $\mathcal F_{1,T}$. Let $e_{T,\ell} = e_T^{1/\ell}$ for $1 \leq \ell \leq L = \log (e_T)/\log \tau $ and $0<\tau <1 $ is fixed and $e_{T, L+1}= 1$. Without loss of generality we can assume that $ \log (e_T)/\log \tau\in \mathbb N$. For  $(u_1, \cdots , u_J)$ let $k_i$ be defined by $u_i \in (e_{T, k_i}, e_{T, k_i+1})$ and $ (N_1, \cdots, N_{L})$ be given by  $\mbox{card}\{ j, u_j \in (e_{T,\ell}, e_{T,\ell+1})\}= N_\ell$ so that $\sum_{\ell} N_\ell = J$ and consider a configuration $\sigma = (k_1, \cdots, k_J)$;  denote by  $\mathcal U_{J,T}(\sigma)$ the set of $\underline u \in \mathcal S_J$ satisfying the configuration $\sigma$, we define $\mathcal H_{T,\sigma, J} = \{ (\underline{w},\underline{u})\in \mathcal S_J \times \mathcal U_{J,T}(\sigma) \}$ and $\mathcal H_{T, \sigma}$ the collection of $\mathcal H_{T, \sigma, J}$ with $J \leq J_{1,T}$.
We have, by symmetry for all $\sigma = (k_1, \cdots, k_J)$ compatible with $(N_1, \cdots, N_L)$ writing $\bar N_\ell = N_1+\cdots+N_\ell $ 
\begin{equation*}
\begin{split}
\Pi_J\left(\mathcal U_{J,T}(\sigma) \right) &= \Pi_J\left( \cap_{\ell = 1}^{L} \{ (u_{\bar N_{\ell-1}+1}, \cdots, u_{\bar N_\ell}) \in  (e_{T,\ell-1}, e_{T,\ell})^{N_\ell} \}\right)\\
&\leq \frac{ \Gamma(\alpha J) }{ \Gamma(\alpha)^J}\prod_{\ell=1}^Le_{T}^{(\alpha-1)/(\ell+1)} \mbox{Vol}\left( \cap_{\ell = 1}^{L} \{ (u_{\bar N_{\ell-1}+1}, \cdots, u_{\bar N_\ell}) \in  (e_{T,\ell-1}, e_{T,\ell})^{N_\ell} \}\right)\\
&\leq \frac{ \Gamma(\alpha J) }{ \Gamma(\alpha)^J}\prod_{\ell=1}^{L-1} e_{T}^{(\alpha-1)N_\ell/(\ell+1)} e_{T}^{N_\ell/(\ell+1)} 
\end{split}
\end{equation*}
We now construct a net $(\underline u^{(j)}, j \leq N_{\sigma,J})$ such that for all $\underline u \in \mathcal U_{J,T}(\sigma)$ there exists $\underline u^{(j)}$ satisfying $|t_i - t_i^{(j)}| \leq \epsilon_T u_i^{(j)}\wedge u_{i+1}^{(j)}$ for all $i$, with $t_i = \sum_{\ell=1}^iu_\ell$. If   $ |t_i - t_i^{(j)}| \leq \epsilon_T e_{T, k_i} \wedge e_{T, k_{i+1}}$ then $ |t_i - t_i^{(j)}| \leq \epsilon_T u_i^{(j)}\wedge u_{i+1}^{(j)}$. Therefore, given a configuration $(k_1, \cdots, k_J)$ compatible with $(N_1, \cdots, N_L)$, we can cover $\mathcal U_{J,T}(\sigma)$ using 
$$ N_J(\sigma) \leq \prod_{i=1}^Je_{T}^{1/(k_i+1)-1/(k_i\wedge k_{i+1}) } \leq \prod_{\ell=1}^{L}e_{T,\ell+1}^{N_\ell}e_{T,\ell}^{-2N_\ell}.$$
The covering number of $\mathcal S_J$ by balls of radius $\zeta\epsilon_T$ is bounded by $ \left( \frac{ 1 }{ \zeta \epsilon_T}\right)^{J}$ and
\begin{equation*}
\begin{split}
I_T &:= \sum_{J}\sqrt{\Pi(J)} N(\epsilon_T, \mathcal H_{T,\sigma, J} ) \\
& \leq
\sum_{J}\sqrt{\Pi(J)} \left( \frac{ 1 }{ \zeta \epsilon_T}\right)^{J}  \sum_\sigma N_J(\sigma) \sqrt{\Pi_j(\mathcal U_{J,T}(\sigma))}\\
& \lesssim \sum_{J}\sqrt{\Pi(J)}\left( \frac{ 1 }{ \zeta \epsilon_T}\right)^{J}  \sum_{(N_1,\cdots, N_{L})} \frac{ J!\Gamma(\alpha J) }{ \Gamma(\alpha)^JN_1!\cdots N_{L}!} \exp\left[ \log e_T\sum_{\ell=1}^{L-1}N_\ell\left( \frac{\alpha+2}{2(\ell+1)} -\frac{ 2 }{ \ell} \right)\right]e_T^{-2N_L/L}\\
&\lesssim  \left( \frac{1 }{ \zeta \epsilon_T}\right)^{J_{1,T}} e^{  2\alpha J_{1,T}\log J_{1,T} } \sum_{J=1}^{J_{1,T} }\sum_{(N_1,\cdots, N_{L})}\frac{ J! }{N_1!\cdots N_{L}!} \prod_{\ell=1}^{L} p_\ell^{N_\ell} \prod_{\ell=1}^{L-1}\frac{e_T^{N_\ell\left( \frac{\alpha+2}{2(\ell+1)} -\frac{ 2 }{ \ell} \right)}}{p_\ell} \frac{e_T^{-2N_L/L}}{ p_{L}p_{L-1}}
\end{split}
\end{equation*}
for any $p_1,\cdots, p_{L}\geq 0$ with $\sum_{\ell=1}^{L+1}p_\ell=1$. Taking $p_L = 1/(L+1)$ and since $\alpha \geq 6$, $ \frac{\alpha+2}{2(\ell+1)} -\frac{ 2 }{ \ell} \geq 0$ for all $\ell \geq 1$, leading to 
\begin{equation*}
\begin{split}
I_T \lesssim \tau^{-2J_{1,T}} \left( \frac{ 1 }{ \zeta \epsilon_T}\right)^{J_{1,T}} e^{  2\alpha J_{1,T}\log J_{1,T}+ (L+1)\log (L+1) } \sum_{J=1}^{J_{1,T} }\sum_{(N_1,\cdots, N_{L})}\frac{ J! }{N_1!\cdots N_{L}!} \prod_{\ell=1}^{L} p_\ell^{N_\ell} \lesssim e^{ K J_{1,T} \log T} 
\end{split}
\end{equation*}
for some $K>0$ and condition \eqref{cond:slices} is verified. 
%%%%
\subsubsection{Proof of Corollary~\ref{prop:mix:beta}}
The proof is based on \citet{rousseau:09}, where mixtures of Beta densities are studied for density estimation, and using Theorem \ref{th:slices}. 
Note that for all $h_1, h_2$ 
 $$  |(h_1(x))_+ - (h_2(x))_+| \leq |h_1(x) - h_2(x) |$$
 so that Proposition \ref{prop:mix:beta} is proved by studying
  $$ \tilde B(\epsilon_T, B) = \left\{(\nu_k,(g_{\ell,k})_{\ell})_k:\quad \max_k|\nu_k-\nu^0_k|\leq\e_T, \, \max_{\ell,k}\|g_{\ell,k}-g^0_{\ell,k}\|_2 \leq \e_T, \,\max_{\ell,k}\|g_{\ell,k} \|_\infty \leq B  \right\}$$
  in the place of $B(\epsilon_T, B)$ and by controlling the $\L_1$-entropy associated to 
  $$\mathcal G_{1,T} = \left\{ g_{\alpha, P};\quad P=\sum_{j=1}^{J} p_j \delta_{\epsilon_j}, \, \epsilon_j \in [e_1, 1 - e_1]; \, \alpha \in [\alpha_{0T}, \alpha_{1T}] ;\,  \sum_j |p_j|=1, \, J\leq J_{1,T} \right\}$$
  where 
  $$e_1 = e^{-a_0 T\epsilon_T^2} , \quad \alpha_{0T} = \exp\left( - Tc_0 \epsilon_T^2\right); \quad \alpha_{1T} = \alpha_1 T^{2}\epsilon_T^4, \quad J_{1,T}  = J_1 T^{1/(2\beta+1)}(\log T)^{(\beta-2)/(4\beta+2)},$$ with  $c_0, \alpha_1, a_0, J_1>0$ and 
  $g_{\alpha, P} = \int_0^1 g_{\alpha, \epsilon} dP(\epsilon)$.  From the proof of Theorem 2.1 in \citet{rousseau:09}, we have that for all $c_2>0$ we can choose $a_0, c_0, \alpha_1>0$ such that 
  $ \Pi\left( \mathcal G_{1,T}^c \right) \leq e^{- c_2 T \epsilon_T^2 } $ and $\mathcal G_{1,T}$ can be cut into  the following slices: 
 we group the components into the intervals $[e_\ell, e_{\ell+1}]$ or $[1-e_{\ell+1}, 1 - e_{\ell}]$ with $e_\ell = e_0^{1/\ell}$ and $e_{L_T} = T^{-t}$, for some $t>0$, and the interval $[e_{L_T}, 1 - e_{L_T}]$. For each of these intervals we denote $N(\ell)$ the number of  components which fall into the said interval, $N(\ell) = \sum_{i=1}^J \1_{\epsilon_i \in (e_\ell, e_{\ell+1}) \cup (1 -e_{\ell+1}, 1 - e_\ell)}$ if $\ell\leq J_T$,  and $N(L_T+ 1)=\sum_{i=1}^k \1_{\epsilon_i \in (e_{L_T},1 -e_{L_T})}$ . Let $J \leq J_{1,T}$
   $\mathcal G_{1,\sigma}(J)  = \{ g_{\alpha, P} \in \mathcal G_{1,T}; \, N(\ell) = k_\ell, \, \sum_{\ell=1}^{L_T+1} k_\ell =J\}$ with $\sigma  $ denoting the configuration $(k_1, \cdots, k_{L_T+1})$. From   \citet{rousseau:09} Section 4.1, for all  $\zeta>0$,  we have 
   $$N(\zeta \epsilon_T, \mathcal G_{1,\sigma}(J),\|\cdot\|_1) \leq (\zeta \epsilon_T)^{-k_{L_T}} \prod_{\ell=1}^{L_T-1} \left( \frac{(\log e_{\ell+1} - \log e_\ell) }{ \zeta \epsilon_T e_\ell }\right)^{k_\ell}   \leq (\zeta \epsilon_T)^{-k_{L_T}} \prod_{\ell=1}^{L_T-1} \left( \frac{\log(1/e_1) }{ \ell(\ell+1)\zeta \epsilon_T e_1^{1/\ell} }\right)^{k_\ell}  $$
   and 
   $$ \sqrt{\Pi(\mathcal G_{1,\sigma})(J) }  \leq \sqrt{\Pi_J(J) } \frac{ \Gamma(J+1)^{1/2} }{ \prod_{\ell =1}^{L_T+1} \Gamma(k_\ell)^{1/2} } \prod_{\ell=1}^{L_T} p_{T, \ell}^{k_\ell/2}, \quad p_{T,\ell} \leq c (e_{\ell+1}^{a+1} - e_{\ell}^{a+1} ), \quad \ell\leq L_T-1$$
   and $p_{T,L_T} \leq 1$.
  Since $ e_{\ell+1}^{a+1} - e_{\ell}^{a+1} \leq e_{\ell+1}^{a+1}\leq e_1^{(a+1)/(\ell+1)}$ and since $J! \geq \prod_{\ell=1}^{J_T+1} k_\ell ! $, we obtain 
\begin{equation*}
\begin{split}
\sum_{J\leq J_{1,T} }\sum_{\sigma}N(\zeta \epsilon_T, \mathcal G_{1,\sigma}(J),\|\cdot\|_1) \sqrt{\Pi(\mathcal G_{1,\sigma}(J))} & \lesssim J_{1,T}e^{ C J_{1,T} \log T } \sum_{\sigma}\frac{ J! }{ \prod_{\ell=1}^{L_T+1} k_\ell !} \prod_{\ell=1}^{L_T+1}\left( \frac{ \bar c }{ \ell (\ell+1)}\right)^{k_\ell} \\& =J_{1,T}e^{ C J_{1,T} \log T } 
\end{split}
\end{equation*}
as soon as $a\geq 3$, where $\bar c^{-1} = \sum_{\ell =1}^{L_T+1} 1/(\ell (\ell+1))$. Therefore condition \eqref{cond:slices} is verified.
  We now study the Kullback-Leibler condition (i). 
  Again, we  use Theorem 3.1 in \citet{rousseau:09}, so that for all $f_0 \in \mathcal H(\beta, L)$ and all $\beta >0$ there exists $f_1$ such that 
 $\|f_0  - g_{\alpha, f_1 } \|_\infty \lesssim  \alpha^{-\beta/2}$, when $\alpha $ is large enough and 
 $g_{\alpha, f_1} = \int_0^1 g_{\alpha,\epsilon}f_1(\epsilon )d\epsilon$, and where $f_1 $ is either equal to $f_0$ if $\beta \leq 2$ or $f_1 = f_0 \sum_{j=1}^{\lceil \beta \rceil  - 1} w_j/\alpha^{j/2} $, with $w_j$ a polynomial function with coefficients depending on $f_0^{(l)}$ $l \leq j$. From that, we construct a finite mixture approximation of $g_{\alpha, f_1 }$. Note that even if $f_0$ is positive, $f_1$ is not necessarily so. Hence to use the convexity argument of Lemma A1 of \citet{ghosal:vdv:01} we write 
 $f_1 $ as $m_+f_{1,+} - m_-f_{1,-}$ with $f_{1,+}, f_{1,-} \geq 0$ and probability densities. In the case where $m_- = 0$ then $f_{1,-}=0$. We approximate 
$g_{\alpha, f_{1,+} }$ and $g_{\alpha, f_{1,-} }$ separately. Contrarywise to what happens in \citet{rousseau:09}, here we want to allow $f_0$ to be null in some sub-intervals of $[0,1]$. Hence we adapt the proof of Theorem 3.2 of \citet{rousseau:09} to this set up. Let $f$ be a probability density on $[0,1]$ we construct a discrete approximation of $g_{\alpha, f }$. Let $\epsilon_0 = \alpha^{-H_0}$ for some $H_0>0$ and define $\epsilon_j = \epsilon_0( 1 + B\sqrt{\log \alpha /\alpha})^j$ for $j = 1, \cdots , J_\alpha $ with $J_\alpha = O(\sqrt{\alpha \log \alpha }) $ and $B>0$ a constant.  We then have, from Lemma \ref{lem:discrete:betas} below that there exists a signed measure $P_0$ with at most $N = O(\sqrt{\alpha}(\log \alpha)^{3/2})$ supporting points on $[\epsilon_1, 1- \epsilon_1]$, such that:
$$\|g_{\alpha, P_0} - f_0\|_2 \leq \|g_{\alpha, P_0} - g_{\alpha, f_1}\|_2 + \|g_{\alpha, f_1} - f_0\|_\infty\lesssim  \alpha^{-\beta/2} ; \quad \|g_{\alpha, P_0}\|_\infty \leq \|f_0\|_\infty + o(1) , \quad P_0 = \sum_{i =1}^N p_i \delta_{\epsilon_i}.$$
As in \citet{rousseau:09} Theorem 3.2, we can assume that  $|p_i|\geq \alpha^{-A}$ for some fixed $A$ large enough. Following from Section 4.1 of  \citet{rousseau:09}, There exists $A'>0 $ such that if $P$ satisfies   $\max_i|P(U_i)-p_i| \leq \alpha^{-A'} |p_i|$, with $U_i = [\epsilon_i(1 -\epsilon_i)( 1 - \alpha^{-A'}), \epsilon_i(1 -\epsilon_i)( 1 + \alpha^{-A'})]$ then 
 $$ \|g_{\alpha, P_0}-g_{\alpha,P}\|_2 \leq \alpha^{-\beta/2}, \quad \|g_{\alpha,P}\|_\infty \leq \|f_0\|_\infty +o(1).$$
 As in \citet{rousseau:09}, if $\epsilon_T = \epsilon_0 T^{-\beta/(2\beta+1)} (\log T)^{5\beta/(4\beta+2)}$, then 
$$\Pi\left(\tilde B(\epsilon_T, \|f_0\|_\infty +1)\right)\geq e^{- c_1 T \epsilon_T^2}$$
for some $c_1>0$, which terminates the proof of Corollary~\ref{prop:mix:beta}.
 \begin{lemma} \label{lem:discrete:betas}
 Assume that $f$ is a bounded probability density on $[0,1]$, then for all $B_0>0$ there exists $\tilde N_0>0$ and  a signed measure $P_0$ with at most $N \leq \tilde{N_0} \sqrt{\alpha} (\log \alpha)^{3/2}$ on $[\epsilon_1, 1 - \epsilon_1]$ such that 
 $$\|g_{\alpha, f} - g_{\alpha, P} \|_2 \lesssim \alpha^{-B_0}, \quad \|g_{\alpha, P_0}\|_\infty \lesssim \|f_0\|_\infty + o(1)$$ 
 \end{lemma}
\begin{proof}[Proof of Lemma \ref{lem:discrete:betas}]
On each of the intervals $(\epsilon_{j-1}, \epsilon_j)$ we construct a probability $P_j$ having support on $(\epsilon_{j-1}, \epsilon_j)$ with cardinality smaller than $N_j \leq N_0 \log \alpha$ and such that 
 \begin{equation}\label{discrete1}
 \| g_{\alpha, f_j } - g_{\alpha, P_j} \|_2^2 \lesssim \alpha^{- B_0}, \quad f_j = \frac{f \1_{(\epsilon_{j-1}, \epsilon_j)} }{ \int_{\epsilon_{j-1}}^{\epsilon_j}f(\epsilon) d\epsilon }\end{equation}
 where $B_0$ can be chosen arbitrarily large by choosing $N_0$ large enough.
 To prove \eqref{discrete1} we use the same ideas as in the proof of Theorem 3.2 of \citet{rousseau:09}. For all $j = 2, \cdots, J-2$ on $(\epsilon_{j-1}, \epsilon_j)$, there exists $P_j$ with at most $N_1 \log \alpha$ terms such that 
  if $x\in [0,1]$, 
  $$\left| g_{\alpha, f_j} - g_{\alpha, P_j}\right| (x) \leq \frac{ \alpha^{-H} }{ x ( 1 - x) } $$
  where $H$ can be chosen as large as need be, by choosing $N_1$ large enough. Moreover, let $x\leq \epsilon_0$  or $x > 1 -\epsilon_0$, 
  then for all $\epsilon \in (\epsilon_1, 1 - \epsilon_1)$, if $x < \epsilon_0$ then $x/\epsilon \leq \delta_\alpha = (1 + B\sqrt{\log \alpha/\alpha})^{-1}$ and 
  \begin{equation*}
  \begin{split}
  g_{\alpha, \epsilon}(x) &\lesssim \sqrt{\alpha} \exp\left(  \alpha \left[ \frac{\log (x/\epsilon) }{ 1 -\epsilon} - (\log x) /\alpha  +\frac{  \log ((1-x)/(1-\epsilon)) }{\epsilon}   \right]\right)
%   &\lesssim \frac{ \sqrt{\alpha}}{\epsilon} \exp\left( - \alpha [h(x/\epsilon) + (2\epsilon+\alpha^{-1}) \log (x/\epsilon) ]\right)
  \end{split}
  \end{equation*}
 If $\epsilon_1 \leq  \epsilon  < 1/4$   then  the function $\epsilon \rightarrow  \frac{\log (\epsilon/x) }{ 1 -\epsilon} - \log(\epsilon/ x) /\alpha  +\frac{  \log ((1-\epsilon)/(1-x)) }{\epsilon}   $ is increasing and 
  \begin{equation*}
  \begin{split}
  g_{\alpha, \epsilon}(x) &\lesssim \frac{\sqrt{\alpha}}{\epsilon} \exp\left(  \alpha \left[  \log ( \delta_\alpha) \left( 1 + x \delta_\alpha+\delta_\alpha^2 x^2) + O(x^3) \right) - 1 + \delta_\alpha^{-1}  \right]\right)\\
 &  \lesssim \alpha^{-B^2/3 + H_0 } \lesssim \alpha^{-B^2/4} ,
\end{split}
\end{equation*}
by choosing $B^2 \geq 12 H_0$.
The same reasoning can be applied to $x> 1-\epsilon_0$, which terminates the proof.
\end{proof}
%%%%%%%%%%%%%%%%%%%%%%%%%%
%%%%%%%%%%%%%%%%%%%%%%%%%%
\bibliographystyle{apalike}
\bibliography{biblio}
\end{document}